\documentclass[twoside,11pt]{article}

\usepackage{blindtext}
\usepackage[abbrvbib, preprint]{jmlr2e}
\usepackage{math_commands}

\usepackage{lastpage}
\jmlrheading{25}{2025}{1-\pageref{LastPage}}{1/21; Revised 5/25}{5/25}{21-0000}{Bohan Zhou, Shu Liu, Xinzhe Zuo and Wuchen Li}

\ShortHeadings{Accelerated MCMC}{Zhou, Liu, Zuo and Li}
\firstpageno{1}

\begin{document}

\title{Accelerated Markov Chain Monte Carlo Algorithms on Discrete States}

\author{\name Bohan Zhou \email bhzhou@ucsb.edu \\
       \addr Department of Mathematics\\
       University of California\\
       Santa Barbara, CA 93106, USA.
       \AND
       \name Shu Liu \email sliu11@fsu.edu \\
       \addr Department of Mathematics\\
       Florida State University\\
       Tallahassee, FL 32306, USA.
       \AND
       \name Xinzhe Zuo \email zxz@math.ucla.edu \\
       \addr Department of Mathematics\\
       University of California\\
       Los Angeles, CA 90095, USA.
       \AND
       \name Wuchen Li \email wuchen@mailbox.sc.edu \\     
       \addr Department of Mathematics\\
       University of South Carolina\\
       Columbia, SC 29208, USA.
       }

\editor{My editor}

\maketitle

\begin{abstract}
We propose a class of discrete-state sampling algorithms based on Nesterov's accelerated gradient method, which extends the classical Metropolis-Hastings (MH) algorithm. The evolution of the probability distribution over discrete states governed by MH can be interpreted as a gradient descent direction of the Kullback--Leibler (KL) divergence, via a mobility function and a score function. Specifically, this gradient is defined on a probability simplex equipped with a mobility-dependent discrete Wasserstein-2 metric. This motivates us to study a momentum-based acceleration framework using damped Hamiltonian flows on the simplex, whose stationary distribution matches the discrete target distribution. Furthermore, we design an interacting particle system to approximate the proposed accelerated sampling dynamics. The extension of the algorithm with a general choice of potentials and mobilities is also discussed. In particular, we choose the accelerated gradient flow of the relative Fisher information, demonstrating the advantages of the algorithm in estimating discrete score functions without requiring the normalizing constant while preserving positivity of the probabilities. Numerical examples, including sampling from a Gaussian mixture supported on a lattice and from a distribution on a hypercube, demonstrate the effectiveness of the proposed discrete-state sampling algorithm. 
\end{abstract}

\begin{keywords}
  Markov chain Monte Carlo methods; Metropolis-Hastings algorithms;  Nesterov acceleration methods; Wasserstein metrics on graphs; Ising model; discrete score functions; mobility functions.
\end{keywords}

\newenvironment{AMS}
{\bgroup\leftskip 20pt\rightskip 20pt \small\noindent{\bf AMS MSC:} }%
{\par\egroup\vskip 0.25ex}

\begin{AMS}
  65C05, 60J22, 82M31, 49Q22, 65K10
\end{AMS}

\section{Introduction}

Markov Chain Monte Carlo (MCMC) methods are algorithms to generate samples from a given probability distribution over a sampling space. A key advantage of the MCMC method is that it only requires knowledge of the target distribution up to a constant, making it well-suited for distributions with intractable normalizing constants—such as those arising in high-dimensional sampling spaces or in Bayesian posterior inference problems \citep{Mengersen1996Rates}. This advantage has led to widespread applications of MCMC algorithms in various scientific domains, including Bayesian inference \citep{Robert2004Monte}, computational physics and chemistry (e.g., simulating Ising models; see \cite{Metropolis1953Equations, Hastings1970MC, Landau2015Guide}), financial mathematics (e.g., risk modeling and option pricing; \cite{jackel2002monte, Kou2004Option, Jacquier2002Bayesian}), and machine learning (e.g., energy-based models; \cite{Teh2004Energy, Hinton2006Fast, Salakhutdinov2009Deep}).

In MCMC methods, the core idea is to construct a Markov chain with designed transition probabilities so that its stationary distribution is the target distribution $\pi$. Samples are then generated by running the chain for a sufficient number of steps. A well-known class of MCMC methods is the Langevin Monte Carlo (LMC) algorithm \citep{Roberts1996MALA1, Roberts1996MALA2}, which is based on the overdamped Langevin dynamics. LMC is designed to sample from a continuous sampling space by conducting gradient descent steps in the direction $\nabla \log \pi$, perturbed by additive white noise. It is known that the evolution of the probability distribution $p$ under LMC is a gradient flow of the Kullback-Leibler (KL) divergence on the space of probability distributions with the Wasserstein-2 metric \citep{Villani2009OT}. Here the Wasserstein-2 gradient of the KL divergence represents the score function $\nabla \log \frac{p}{\pi}$. Classical MCMC algorithms frequently encounter challenges such as slow convergence and poor mixing. This has motivated the demand for accelerated sampling algorithms.  In recent decades, the Nesterov accelerated gradient (NAG) method \citep{Nesterov1983Accelerating, Su2016DE, Wibisono2018Sampling}, a refinement of momentum-based optimization, has been shown to improve the convergence speed of gradient descent algorithms. This naturally raises the question of whether a Nesterov-type acceleration can be formulated for sampling, leveraging the Wasserstein gradient flow structure. This question has been investigated for continuous-space sampling using different approaches \citep{Cheng2018Underdamped, Taghvaei2019Accelerated, Ma2021Nesterov, Wang2022Accelerated,  Li2022Hessian,  zuo2024gradient, Majee2025MCMCnet}.  See \citep{Albritton2024Variational, Cao2023Convergence, brigati2024explicit} and the references therein for convergence analysis under various norms. While there exists a rich literature on continuous-space sampling, many crucial sampling problems are defined over discrete spaces,
as seen in statistical physics \citep{Landau2015Guide}, combinatorics \citep{Jerrum1996Markov}, and discrete probabilistic graphical models \citep{Jordan1999Introduction}. A natural question arises:

\textit{How can we design a Nesterov accelerated sampling algorithm on discrete state spaces? What is the accelerated dynamics with discrete score functions?}

In this paper, we mainly focus on sampling from $\pi$ on a finite \textit{discrete} space \(V = \{1,2,\dots,n\}\). The Metropolis-Hastings (MH) algorithm \citep{Metropolis1953Equations, Hastings1970MC} serves as a foundational framework for sampling on discrete space. Given a candidate kernel $(q_{ij})$, leveraging the acceptance ratio, the MH algorithm designs a \textit{time-homogeneous} Markov chain, with a transition rate matrix $Q$ ($Q_{ij}$ depends on $(q_{ij})$) 
that satisfies the \textit{detailed-balance} condition $\pi_j Q_{ji}=\pi_i Q_{ij} $. 

\setstackgap{S}{6pt}
\begin{table}[hptb!]
    \centering
    \begin{tabularx}{\linewidth}{r|c|c}
    & Discrete-time & Continuous-time \\\hline
 \mycell MC   & \stackunder{$\P(X^{(k+1)}=j\mid X^{(k)}=i)=P_{ij}$}{\DA}  & $\P(X(t+h)=j\mid X(t)=i)\approx\delta_{ij}+Q_{ij}h$ \\
 FM & $p^{(k+1)}=p^{(k)}P$ in \eqref{eq:jump} \hspace{4em}\LBA& \stackunder{$\dot{p}(t)=pQ$ in \eqref{eq:Kolmogorov}}{\DA}\\
 GF & $p^{(k+1)}=p^{(k)}(I_n+Q\Delta t)$ in \eqref{eq:jump} & \stackunder{$\dot{p}(t)=-\nabla_p D_{\phi}(p\|\pi)\mathbb{K}(p) $ in \eqref{eq:Wgf_precond}}{\DA}\\
 HF & $p^{(k+1)}=p^{(k)}(I_n+\bar{Q}_{\psi}^r\Delta t)$ in \eqref{eq:Kolmogorov-Ham}, \eqref{eq:AMH-jump} \LA \hspace{0.5em} & \eqref{eq:fisher-general} is an instance of general framework \eqref{eq:AMH-matrix} 
    \end{tabularx}
    \caption{The conceptual flow in this work. ``MC'',``FM'',``GF'',``HF'' are abbreviations for Markov chain, forward master equation, gradient flow and Hamiltonian flow, respectively. We start with the Markov chain and end up with the damped Hamiltonian dynamics originating from Nesterov's accelerated gradient method. \label{tab:flow} }
\end{table}

Analogous to the LMC algorithm on the continuous-state space, which can be regarded as a gradient flow on the probability space, recent work \citep{Maas2011gradient, Erbar2012Ricci, Karatzas2021Trajectorial} has shown that the reversible Markov chain described in the forward master equation can be interpreted as the gradient flow of the KL divergence with respect to the discrete Wasserstein-$2$ metric, induced by a mobility function. The associated discrete Wasserstein-$2$ metric $\mathbb{K}(p)$ allows us to study the MCMC method as a first-order optimization algorithm on the probability simplex.

In this paper, we propose a momentum-based acceleration framework for discrete sampling. Our work is  inspired by the interplay between the Wasserstein gradient flow structure of the MH algorithm and the NAG method. The key ingredients of our approach are summarized in \Cref{tab:flow}. In particular, we consider a \textit{damped Hamiltonian flow}, driven by the relative Fisher information:
\[
\mathrm{I}(p \| \pi)=\frac{1}{4}\sum_{i,j=1}^n \pi_i Q_{ij} \theta_{ij}(p)  \left(\log\frac{p_i}{\pi_i}-\log\frac{p_j}{\pi_j}\right)^2.
\]
The dynamics evolve according to
\begin{equation}\label{eq:fisher-general}
\left\{
\begin{aligned}
    &\frac{\mathrm{d}p_i(t)}{\mathrm{dt}} + \sum_{j \neq i} \pi_i Q_{ij} \theta_{ij}(p(t)) (\psi_j(t) - \psi_i(t)) = 0, \\
    &\frac{\mathrm{d}\psi_i(t)}{\mathrm{d}t} + \gamma(t) \psi_i(t) + \frac{1}{2} \sum_{j \neq i} \pi_i Q_{ij}\frac{\partial \theta_{ij}(p(t))}{\partial p_i}  (\psi_i(t) - \psi_j(t))^2 + \frac{\partial \mathrm{I}(p(t) \| \pi)}{\partial p_i} = 0\,.
\end{aligned}
\right.
\end{equation}  
The first equation is a discrete-state continuity equation that describes the evolution of the probability function \(p\). The second equation is a discrete Hamilton–Jacobi equation, which governs the dynamics of the momentum variable \( \psi \), incorporating both dissipative effects through the term \( \gamma(t) \psi_i \) and information-theoretic feedback via \( \frac{\partial \mathrm{I}(p \| \pi)}{\partial p_i} \). Here, \(\gamma(t)\) is a user-specified damping parameter. \((\theta_{ij}(p(t)))\) is a mobility weight matrix that may depend on $p(t)$. A particularly important choice is the \textit{logarithmic mean} or \textit{KL divergence mean}, given by
    \begin{equation}\label{eq:log-mean}
    \theta_{ij}(p) = \frac{ \frac{p_i}{\pi_i}-\frac{p_j}{\pi_j}}{\log \frac{p_i}{\pi_i} - \log \frac{p_j}{\pi_j}} = \frac{p_i}{\pi_i} \cdot \frac{1- \frac{\pi_i}{\pi_j}\frac{p_j}{p_i}}{\log \left( \frac{\pi_j}{\pi_i} \frac{p_i}{p_j} \right)}.
    \end{equation}

The proposed accelerated MCMC (aMCMC) dynamics in \eqref{eq:fisher-general} can be shown to guarantee the non-increasing behavior of its associated Hamiltonian function on a simplex set, defined as $\mathcal{H}(p, \psi) = \frac{1}{2} \psi \mathbb{K}(p) \psi^\top + \mathrm{I}(p \| \pi)$, and the strict positivity of $p(t)$. For numerical implementation, we employ a staggered scheme to discretize \eqref{eq:fisher-general} in time. Additionally, an interacting particle system on discrete states is introduced, with transition rates derived from a positive/negative-part decomposition of the gradient term \( \psi_j - \psi_i \). We refer the reader to \Cref{sec:alg} for a detailed discussion of the algorithmic design.

{Numerical experiments in \Cref{sec:num} evaluate the proposed methods in
several complementary regimes. Fixed-iteration experiments on small graphs
verify the accelerated transient behavior predicted by the continuous
dynamics, while a multimodal lattice example illustrates the practical use
of warm starts, adaptive time stepping, and restarts. Notably, the incorporation of score functions in the aMCMC algorithm yields a higher empirical accuracy with respect to target distributions, with error of order \( \mathcal{O}(\frac{1}{M}) \), where \( M \) denotes the number of swarm particles. This marks a substantial improvement over the \( \mathcal{O}(\frac{1}{\sqrt{M}}) \) accuracy of the MH algorithm. Under fixed \textit{wall-clock}
budgets, experiments under current implementations on image-derived lattice targets and on a
one-dimensional Ising model show that the \textit{multi-chain} \texttt{log-Fisher}
achieves smaller $\ell^2$ error for empirical distributional and more accurate estimation for the partition function than the \textit{single-chain MH} benchmark, using aggregated-sample and effective-sample protocols. These experiments demonstrate an important advantage of \texttt{log-Fisher} method, that is, our method is able to produce same amount of samples with higher quality for Monte-Carlo estimation within the same computational time. 
}

This paper is organized as follows: \Cref{sub:MH} reviews the classical MCMC with the MH update. \Cref{sub:Maas} revisits the gradient flow formulation of the forward master equation, while \Cref{sub:Nesterov} discusses gradient and accelerated flows in Euclidean space \(\mathbb{R}^d\). Building on both perspectives, we introduce a class of aMCMC sampling algorithms on discrete-state spaces in \Cref{sec:acceleration}, with particular instances listed in \Cref{tab:AMH}. The properties of proposed dynamics are discussed in \Cref{sec:property}. \Cref{sec:alg} describes the numerical scheme and some practical techniques including initialization and restart. \Cref{sec:num} presents our numerical results. Lastly, we discuss our current limitations and future work in \Cref{sec:dis}.

\begin{table}[htpb!]
    \centering
    \begin{tabular}{c|c|c|c|c}
        Methods & $\theta(\frac{p_i}{\pi_i},\frac{p_j}{\pi_j})$ & potential $\mathcal{U}(p)$ & w/o $Z$ & strict positivity \\  \hline 
        \texttt{Chi-squared} & 1 &  $\frac{1}{2}\sum_{i=1}^n\frac{(p_i-\pi_i)^2}{\pi_i}$ & No & No \\
        \texttt{KL} & log-mean \eqref{eq:log-mean}  & $\sum_{i=1}^n p_i\log\frac{p_i}{\pi_i}.$ & Yes & No \\
        \texttt{log-Fisher} & log-mean \eqref{eq:log-mean}  & $\frac{1}{4}\sum_{i,j=1}^n \omega_{ij}(\log\frac{\pi_j}{\pi_i}\frac{p_i}{p_j})(\frac{p_i}{\pi_i}-\frac{p_j}{\pi_j}) $ & Yes & Yes \\
        \texttt{con-Fisher} & $\theta_{ij}$ & $\frac{1}{4}\sum_{i, j=1}^n  \omega_{ij}\theta_{ij} (\log\frac{\pi_j}{\pi_i}\frac{p_i}{p_j})^2 $ & Yes & Yes \\
    \end{tabular}
    \caption{Examples of aMCMC dynamics. The 2nd column is the mobility weight matrix $(\theta_{ij}(p))$. The 3rd column is the potential function $\mathcal{U}(p)$. The 4th column indicates whether the method can be implemented without knowing the normalizing constant $Z$ a prior. The 5th column indicates whether the method ensures that $p$ remains strictly positive.}
    \label{tab:AMH}
\end{table}

\noindent\textbf{Notations}

The probability simplex $\mathbb{P}(V)$ supported on $V=\left\{1,2,\ldots,n\right\}$ is defined as 
    \begin{equation*}
      \P(V) = \left\{ p = [p_1, \dots, p_n] \in \mathbb{R}^n ~ \Big| ~ p_i \geqslant 0 ~ \textrm{for any } i, ~ \textrm{and} ~  \sum_{i=1}^n p_i = 1  \right\}.
    \end{equation*}
We adopt the convention that vectors are represented as row vectors. Denote $T_p\mathbb{P}(V)=\{v=[v_1,\ldots,v_n] \in \mathbb{R}^n\ | \  v_1 + \ldots + v_n = 0 \}$ by the tangent space of $\mathbb{P}$ at $p$. The target probability is denoted by $\pi=[\pi_1,\ldots,\pi_n]$, and is typically the stationary probability of the proposed dynamics governing the state variables $p(t)=[p_1(t),\ldots,p_n(t)]$. For simplicity, we assume that $\pi$ is strictly positive ($\pi_i>0$ for all $i$). The transition probability matrix is denoted by $P=(P_{ij})$, where $P_{ij}=\P(X^{(k+1)}=j \mid X^{(k)}=i)$ represents the probability of transitioning from state $i$ to state $j$. We adopt the row-sum-one convention, meaning $\sum_{j=1}^nP_{ij}=1$ for each $i$. Additionally, we define the associated transition rate matrix as $Q=(Q_{ij})$, where $Q_{ij}\geqslant 0$ for non-diagonal entries, and $\sum_{j=1}^n Q_{ij} =0$ for each $i$. Consequently, $Q_{ii}=-\sum_{j\neq i}Q_{ij}$. Let $\bar{A}^r$ denote the row-sum-zero projection of any matrix $A$ such that
$\left\{
\begin{aligned}
    &\bar{A}^r_{ij}=A_{ij},\quad\textrm{for } i \neq j;\\
    &\bar{A}^r_{ii}= -\sum_{j\neq i} A_{ij}.
\end{aligned}
\right.$ Thus for any transition rate matrix $Q$, $\bar{Q}^r=Q$. We occasionally use the notation $\bar{Q}^r$ in place of $Q$ to highlight that $Q$ is a row-sum-zero matrix. The stationary probability function $\pi$ can be defined to satisfy that either $\pi P=\pi$ or $\pi Q=0$.

Given a $Q$-matrix, it induces a directed weighted graph $G=(V,E,\omega)$ by: 1) $e_{ij}\in E$ if $Q_{ij}\neq 0$; 2) $\omega_{ij}=\pi_i Q_{ij}$. Under this choice, $\omega$ is row-sum-zero as well. Furthermore, if $Q$-matrix is reversible w.r.t its stationary probability $\pi$, i.e, $\pi_i Q_{ij}=\pi_j Q_{ji}$, then we say the associated Markov chain is \textit{reversible} and we obtain an undirected graph $G$ and a symmetric matrix $\omega$.

The identity matrix in $\mathbb{R}^{n\times n}$ is denoted by $I_n$. A row vector of size $n$ (a matrix of size $m\times n$) whose entries are ones is denoted by $\Id_n$ ($\Id_{m\times n}$, respectively). Similarly, $0_{n\times n}$ denotes the $n\times n$ zero matrix. $\R_+,\R_{\geqslant 0}$ denote the sets of positive and nonnegative real numbers, respectively. The orthogonal complement of the $\textrm{span}\{\Id_n\}$ is $\Id_n^{\perp}=\left\{v\in \R^n\mid \Id_n\cdot v =0\right\}$. For any square matrix $A\in\mathbb{R}^{n\times n}$, we denote $\mathrm{Ker}(A)=\{x=[x_1,\ldots,x_n]\ | \ xA=0 \}.$

\section{Preliminaries}

\subsection{Metropolis-Hastings as the Forward Master Equation\label{sub:MH} }

Consider a \textit{continuous-time reversible} Markov chain on a finite state space $V$. Let $Q$ be a transition-rate matrix of the Markov chain, and $p(0)=[p_1(0),\ldots,p_n(0)]$ be a probability function on $V$. The evolution dynamics of the probability function $p(t)$ with initial data $p(0)$ is described by the forward master equation:
    \begin{equation}\label{eq:Kolmogorov}
     \frac{\mathrm{d}}{\mathrm{d}t} p_i= \sum_{j=1}^n p_j Q_{ji}=\sum_{j\neq i} p_j Q_{ji}-p_i Q_{ij},
    \end{equation}
whose stationary probability function is denoted by $\pi=[\pi_1,\ldots,\pi_n]$. The forward master equation can be represented in a matrix form as $\frac{\mathrm{d}}{\mathrm{d}t}p = p Q $, by using $Q_{ii}=-\sum_{j\neq i}Q_{ij}$. From the fundamental theorem of Markov chains, there is a unique, strictly positive stationary distribution $\pi$ if we further assume the Markov chain is ergodic.

One can discretize \eqref{eq:Kolmogorov} using the forward difference scheme. Given a step size $\Delta t>0$, the discrete-time update satisfies 
    \begin{equation}\label{eq:jump}
       p^{(k+1)}= p^{(k)} P, \quad P=I_n+ Q\Delta t\in\mathbb{R}^{n\times n},
    \end{equation}
where $I_n$ is the identity matrix. When $\Delta t$ is chosen sufficiently small, the matrix $P$ is a valid transition probability matrix so that \eqref{eq:jump} is a discrete-time jump process.

In a sampling algorithm, $Q$ must be designed to ensure that the probability distribution $p(t)$ converges to the given target distribution $\pi$. The Metropolis-Hastings algorithm provides a generic way to construct a \textit{discrete-time} Markov chain on $V$, that is ergodic (i.e., irreducible and aperiodic) and stationary with respect to $\pi$. Given a user-specified conditional probability $q_{ij}=\P(Y=j\mid X=i)$, also known as the \textit{candidate kernel}, MH designs an acceptance-rejection matrix:
    \[A_{ij} := A(X=i,Y=j)=\left\{
    \begin{aligned}
    &\min\left\{\frac{\pi_j q_{ji}}{\pi_i q_{ij}},1\right\},\qquad &\pi_i q_{ij}>0;\\
    &1,\qquad &\pi_i q_{ij}=0.
    \end{aligned}
    \right.
    \]
A key advantage is that the acceptance-rejection probability depends on the ratio of $\frac{\pi_j}{\pi_i}$, eliminating the need to know the normalizing constant for $\pi$ \citep{Ottobre2016MCMC}. In this work we use the standard random walk MH scheme for $(q_{ij})$ as a benchmark for comparisons. The Metropolis-Hastings algorithm is presented in \Cref{alg:MH}. 

\begin{algorithm}[hp!tb]\caption{Metropolis-Hastings algorithm  \label{alg:MH}}
\SetKwInOut{Input}{Input}\SetKwInOut{Output}{Output}
    \Input{Initial data $X^{(0)}$, target distribution $\pi$, candidate kernel $q$, total iterations $N$.}
    \Output{The Markov chain $(X^{(i)})_{i=1}^N$.}
    \BlankLine
    \For{$i\leftarrow 0$ \KwTo $N$}{
    Propose a candidate $Y^{(i)} \sim q(y\mid X^{(i)})$\;
    Generate $u^{(i)} \sim \mathrm{uniform}[0,1]$\;
    Set $X^{(i+1)} = Y^{(i)} $ if $u \leqslant A(X^{(i)},Y^{(i)})$; otherwise $X^{(i+1)}=X^{(i)}$.
    }
\end{algorithm}

The acceptance-rejection matrix induces a transition rate matrix for \eqref{eq:Kolmogorov}, satisfying the detailed balance condition:
    \begin{equation}\label{eq:Q-MH}
        Q_{ij}^{\textrm{MH}}:= q_{ij}A_{ij}=\min \left\{\frac{\pi_j}{\pi_i} q_{ji},q_{ij}\right\}.
    \end{equation}

In summary, the MH algorithm yields a reversible and ergodic Markov chain on a discrete state space $V$. Starting from an arbitrary initial distribution $p(0)$, the time evolution of the probability distribution $p(t)$ under the MH dynamics is governed by \eqref{eq:Kolmogorov}, alternatively, the jump process $p^{(k)}$ arising from MH can be described by \eqref{eq:jump}, specifically using the transition rate matrix $Q^{\textrm{MH}}$. 

One advantage of $Q^{\textrm{MH}}$ is its time-homogeneity, which simplifies implementation in numerical experiments. Additionally, $Q^{\textrm{MH}}$ is dependent of the normalizing constant $Z$ for $\pi$. However, it may suffer from \textit{pseudo-convergence} in multimodal distribution \citep[see][Chapter 1.11.2]{Geyer2011Introduction}. In practice, MH sampling can proceed either by running a single long chain or by averaging over multiple shorter chains.

\citet{Schnakenberg1976Network} shows that given any transition rate probability matrix $Q$ and the target distribution $\pi$, there is a weighted directed graph representation $G=(V,E,\omega)$ of the forward master equation, where weight matrix $\omega$ is given by $(\omega_{ij})=(\pi_i Q_{ij})$. In the specific case where $Q=Q^{\mathrm{MH}}$, the detailed balance condition implies that $\omega_{ij}=\omega_{ji}$, rendering $G$ an undirected graph that permits self-loops but excludes multiple edges. Our primary focus in this work is sampling on the corresponding undirected graph $G=(V,E,\omega)$ with $Q=Q^{\textrm{MH}}$.

\subsection{Reversible Markov Chain as a Wasserstein Gradient Flow}\label{sub:Maas}

A series of seminal works \citep{Maas2011gradient, Erbar2012Ricci,Chow2012Fokker, Mielke2013Geodesic,Erbar2016Gradient,Karatzas2021Trajectorial} interpret the reversible Markov chain on a discrete state space as the ``gradient flow" of the entropy functional over the probability manifold $\P(V)$, by introducing a graphical version of the Wasserstein metric, and exploiting the geodesic convexity of the entropy w.r.t this metric. In this subsection, we provide a brief review with a focus on the Wasserstein gradient flow and postpone the discussion on convexity to \Cref{sub:damping} and \Cref{app:Hessian}.

We refer interested readers to \citep{Otto2001Geometry, Ambrosio2008Gradient, Villani2009OT} for Wasserstein gradient flows (WGF) on continuous state spaces; and to \citep{Maas2011gradient, Chow2012Fokker} on discrete state spaces. In these studies, the probability manifold is endowed with a Riemannian metric known as the Wasserstein metric. The so-called \textit{Otto calculus} \citep{Otto2001Geometry} is then used to introduce a gradient flow interpretation on the probability manifold for a family of evolution equations.

We first recall approaches in \citep{Mielke2013Geodesic,Gao2024Master} to treat the gradient structure w.r.t the Riemannian metric. In optimization, it is equivalent to apply a preconditioning matrix---commonly known as \textit{Onsager's response} matrix $\mathbb{K}$---onto the flat gradient. Given a convex function $f:\mathbb{R}_{\geqslant 0} \mapsto \mathbb{R}$ that satisfies $f(1)=0$, $f(0)=\displaystyle\lim_{x\to 0+}f(x)$, and $f''(x)>0$ for $x>0$, we define the $f$-divergence on $\P(V)$ as
    \begin{equation}\label{eq:phi-divergence}
      \mathrm{D}_f(p\|\pi) = \sum_{i=1}^n f(\frac{p_i}{\pi_i})\pi_i.
    \end{equation}
By Jensen's inequality, $\D_{f}(p\|\pi)\geqslant f(1)=0$. A common choice of $f$ is $f(x) = x\log x $. This yields the Kullback–Leibler divergence $\mathrm{D}_{\textrm{KL}}(p\|\pi) = \sum_{i=1}^n p_i \log(\frac{p_i}{\pi_i})$.

Due to the detailed balance $\omega_{ij}=\pi_i Q_{ij}=\pi_j Q_{ji}=\omega_{ji}$, we can rewrite the  forward master equation \eqref{eq:Kolmogorov} as
    \begin{equation}\label{eq:Kol-in-theta}
        \begin{aligned}
        \frac{\mathrm{d}}{\mathrm{d}t}p_i&=\sum_{j\neq i} p_j Q_{ji}-p_i Q_{ij} = \sum_{j\neq i} \frac{p_j}{\pi_j} \pi_j Q_{ji} - \frac{p_i}{\pi_i} \pi_i Q_{ij}=\sum_{j\neq i} \omega_{ij}(\frac{p_j}{\pi_j}-\frac{p_i}{\pi_i})\\
        &=\sum_{j\neq i} \omega_{ij}\theta_{ij}(p)(f^{\prime}(\frac{p_j}{\pi_j})-f^{\prime}(\frac{p_i}{\pi_i}))\,,
        \end{aligned}
    \end{equation}
where we introduce a mobility function $\theta(x,y)$ (also known as an activation function , see \cite{Gao2024Master}) as
    \begin{equation}\label{eq:phi-average}
    \theta (x, y) =\left\{
    \begin{aligned}
      &\frac{x - y}{f'(x) - f'(y)}\qquad &x\neq y;\\
      &\frac{1}{f''(x)}\qquad &x=y,
    \end{aligned}
    \right.
    \end{equation}
and use the abbreviation $\theta_{ij}(p)=\theta (\frac{p_i}{\pi_i},\frac{p_j}{\pi_j})$. 

Now we define the Onsager's response matrix $\mathbb{K}(p)$ as      
    \begin{equation}\label{def:Onsager_response}
      \mathbb{K}_{ij}(p)=-\omega_{ij}\theta_{ij}(p)\qquad\textrm{and}\qquad \mathbb{K}_{ii}(p)=-\sum_{j\neq i}\mathbb{K}_{ij}(p)\,.
    \end{equation}
$\mathbb{K}(p)$ is row-sum-zero and symmetric. One can verify that $\mathbb{K}(p)$ is positive semi-definite with $\mathrm{Ker}(\mathbb{K}(p)) = \mathrm{span}(\Id_n)$. We refer the readers to  \Cref{app:derivation_wg} for detailed discussion. In fact, $\mathbb{K}(p)$ is the graph Laplacian of a weighted undirected graph with a weight matrix $(\omega_{ij}\theta_{ij})$.
    
Since $\partial_{p_j} \D_{f}(p \| \pi)=f'(\frac{p_j}{\pi_j})$, \eqref{eq:Kol-in-theta} can be rewritten as $\displaystyle\frac{\mathrm{d}}{\mathrm{d}t}p_i = -\sum_{j=1}^n \mathbb{K}_{ij}(p) \partial_{p_j}\mathrm{D}_f(p \| \pi)$. Denote by
$\nabla_p \mathrm{D}_f(p\|\pi) := [\partial_{p_1}\mathrm{D}_f (p\|\pi), \dots, \partial_{p_n}\mathrm{D}_f (p\|\pi)]$, we have 
    \begin{equation}\label{eq:Wgf_precond}
    \frac{\mathrm{d}}{\mathrm{d}t}p =- \nabla_{p} \mathrm{D}_f(p \| \pi) \mathbb{K}(p),\qquad\textrm{or equivalently}\quad  \frac{\mathrm{d}}{\mathrm{d}t}p \mathbb{K}^{\dagger}(p)=-\nabla_{p} \mathrm{D}_f(p \| \pi),
    \end{equation}
where $\mathbb{K}^{\dagger}(p)$ is the Moore-Penrose inverse of $\mathbb{K}(p)$. Thus \eqref{eq:Wgf_precond} is employed with a preconditioning matrix $\mathbb{K}(p)$ on the gradient $\nabla_p \mathrm{D}(p\|\pi)$ from the view of optimization.

Consider a curve $p(t)$ on $\P(V)$ parametrized by a time variable $t$, we denote its tangent vector by $\dot{p}(t)$. Define the \textit{graphical Wasserstein} metric tensor $g_W(\cdot,\cdot)$ as 
    \begin{equation}\label{def:W_metric}
      g_{W}(\dot{p}(t), \dot{p}(t)) = \frac{1}{2} \sum_{i=1}^n \sum_{j\neq i} \omega_{ij} \theta_{ij}(p(t))(\psi_i(t) - \psi_j(t))^2 =\psi(t)\mathbb{K}(p(t))(\psi(t))^{\top}, 
    \end{equation}
where $\psi(t) = [\psi_1(t), \ldots, \psi_n(t)]$ denotes the momentum (cotangent) vector associated with the tangent vector $\dot{p}(t)$. We have the following graphical continuity equation:
\begin{equation} \label{eq:graph_continuity_eq}
  \dot{p}_i(t) + \sum_{j\neq i} \omega_{ij} \theta_{ij}(p(t)) (\psi_j (t) - \psi_i (t)) = 0,\qquad\textrm{that is,}\qquad \dot{p}-\psi \mathbb{K}(p)=0.
\end{equation}
Thus the inner product \eqref{def:W_metric} is expressed as 
    \begin{equation}\label{eq:W_metric}
    g_W(\dot{p}(t), \dot{p}(t)) = \dot{p}(t) \mathbb{K}^\dagger(p) \dot{p}(t)^\top.
    \end{equation}
One can interpret $g_W(\cdot, \cdot)$ as a Riemannian metric on $\mathbb{P}(V)$ with a metric tensor $\mathbb{K}^\dagger(p)$ 

It is worth noting that  $\mathrm{Ker}(\mathbb{K}(p)^\dagger) = \mathrm{Ker}(\mathbb{K}(p)) = \mathrm{span}(\Id_n)$. One can verify that $\mathbb{K}(p)$ is positive definite on $\Id_n^\perp$, and so is $\mathbb{K}(p)^\dagger$. Since the tangent space $T_p \mathbb{P}(V)$ at $p$ corresponds to $\Id_n^\perp$, $\mathbb{K}(p)^\dagger$ is positive definite on $T_p \mathbb{P}(V)$ for all $p \in \mathbb{P}(V)$. Consequently, one can treat $g_W(\cdot, \cdot)$ as a Riemannian metric on $\mathbb{P}(V)$. Such a metric also determines the geodesic and the corresponding distance function on $(\mathbb{P}(V), g_W)$. One can introduce the \textit{kinetic energy} on the graph from the Riemannian metric by 
    \begin{equation*}
       \frac12 g_W(\dot p(t), \dot p(t)) = \frac{1}{2} \psi(t) \mathbb{K}(p(t)) \psi(t)^\top = \frac{1}{4}\sum_{i=1}^n \sum_{j\neq i} \omega_{ij} \theta_{ij}(p(t)) (\psi_i(t)-\psi_j(t))^2.   \label{def:graph kinetic energy}
    \end{equation*}

Furthermore, analogous to the Benamou-Brenier theorem \citep{Benamou2000Computational} in continuous space, the graphical version of the Wasserstein metric between two probabilities $p_0$ and $p_1$  on the probability manifold $\P(V)$ is defined as follows: 
    \begin{equation}\label{def: Graph BB}
      W_2^2(p_0, p_1 ) = \inf_{p,\psi} \left\{  \frac12 \int_0^1 \sum_{i=1}^n \sum_{j\neq i}^n \omega_{ij}\theta_{ij}(p(t))(\psi_i(t) - \psi_j(t))^2 \mathrm{d}t    \right\},
    \end{equation}
where the infimum is taken over all sufficiently regular curves $p:[0,1]\mapsto \P(V)$ and $\psi:[0,1]\mapsto \R^{n}$ that satisfy the graphical continuity equation \eqref{eq:graph_continuity_eq} with boundary conditions $p(0)=p_0$, $p(1)=p_1$. And the minimizer $p(t)$ is the geodesic from $p_0$ to $p_1$ on the Riemannian manifold $(\P(V),g_W)$. Consequently, \eqref{eq:Wgf_precond} is the \textit{Wasserstein gradient flow} of $\mathrm{D}_f(\cdot\|\pi)$ on $(\mathbb{P}(V), g_W)$:
    \begin{equation*} 
      \frac{\mathrm{d} p}{\mathrm{d} t} = - \mathrm{grad} \mathrm{D}_f(p \| \pi) = -\nabla_p \mathrm{D}_f(p\|\pi) \mathbb{K}(p).
    \end{equation*}
The detailed derivation for the Wasserstein gradient $\mathrm{grad} \mathrm{D}_f(p \| \pi)$ can be founded in \Cref{app:derivation_wg}.

{\begin{remark}
The mobility function and the corresponding Onsager response matrix are central to non-equilibrium statistical physics. Here, the mobility determines the directions, while the Onsager's matrix characterizes the linear response around equilibrium. Together they specify the dissipation mechanism: from the positive definiteness of Onsager's matrix, one has the free energy dissipation of the master equations. In the context of statistical physics, there are several well-known gradient flow structures from the Onsager matrix, including gradient drift Fokker-Planck equations, Allen-Cahn equations, etc. These studies are known as dynamical density functional theory \cite{teVrugt2020DDFT}, which applies to both continuous and discrete-state gradient flows. 
\end{remark}
}

\subsection{Gradient Flow and Accelerated Flow in \texorpdfstring{$\mathbb{R}^n$}{Rn}}\label{sub:Nesterov}
In this subsection, we review some basic concepts and results from classical optimization in the Euclidean space $\R^d$. This perspective will help clarify how an accelerated framework can emerge naturally from gradient flows.

Let $U:\mathbb{R}^d \to \mathbb{R}$ be a $\mathcal{C}^1$ convex function. The classical gradient descent method for finding the global minimum of $U(x)$ consists of iteratively moving in the negative gradient direction at a small step size $\Delta t$
    \[x^{(k+1)} = x^{(k)} - \Delta t \nabla U(x^{(k)}).\] 
Convergence of the gradient descent algorithm can be guaranteed when $U(x)$ is $L$-smooth and $0<\Delta t< 1/L$. Then one has $U(x^{(k)})-U(x^*) = \mathcal{O}(k^{-1})$, where $U(x^*)$ is the minimum value. If $U(x)$ is further assumed to be $\lambda$-strongly convex, then with the choice of $\Delta t = 1/L$, one has linear convergence: $U(x^{(k)})-U(x^*) = \mathcal{O}\big((1-\lambda/L)^k\big)$.

The gradient flow can be viewed as the limit of the gradient descent as $\Delta t \to 0$. We consider instead the continuous trajectory $x(t)$ which follows the first-order ODE: 
\[\dot x(t) = -\nabla U(x(t))\] 
for $t>0$ with the initial condition $x(0) = x_0$. It is well known that the convergence rate of gradient flow is $\mathcal{O}(t^{-1})$ for convex functions. If $f$ is $\lambda$-strongly convex for some $\lambda >0$, then gradient flow converges at a rate of $\mathcal{O}(e^{-\lambda t})$.

The seminal paper \citet{Nesterov1983Accelerating} proposed the following iteration scheme to minimize a convex function $U(x)$ with some initial guess $x_0=y_0$:
    \begin{equation}\label{eq:nag_convex}
    \left\{
        \begin{aligned}
            x^{(k)} &= y^{(k-1)} - \Delta t \nabla U(y^{(k-1)}), \\
            y^{(k)} &= x^{(k)} + \frac{k-1}{k+2}(x^{(k)}-x^{(k-1)})\,. 
        \end{aligned}
    \right.
    \end{equation}
With the choice of $0<\Delta t \leqslant 1/L$, this scheme can be shown to converge at a much faster rate than the gradient descent. Specifically, $U(x^{(k)})-U(x^*) = \mathcal{O}(\frac{1}{\Delta t k^2})$. As a result, Nesterov's method is also referred as the \textit{accelerated gradient} method.

For $\lambda$-strongly convex and $L$-smooth functions, Nesterov's iterations are
    \begin{equation}\label{eq:nag_sc}
    \left\{
        \begin{aligned}
            x^{(k)} &= y^{(k-1)} - \Delta t \nabla U(y^{(k-1)}), \\
            y^{(k)} &= x^{(k)} + \frac{\sqrt{\kappa}-1}{\sqrt{\kappa}+1}(x^{(k)}-x^{(k-1)})\,. 
        \end{aligned}
    \right.
    \end{equation}
Here $\kappa = \frac{\lambda}{L}$ is the condition number. With the choice of $\Delta t = 1/L$, one can show that $U(x^{(k)})-U(x^*) = \mathcal{O}\big(e^{-k/\sqrt{\kappa}}\big)$. 

In contrast to the gradient flow, the continuous version of the Nesterov acceleration is a second order ODE, first proposed and studied by \citet{Su2016DE}:
\begin{equation}\label{eq:su_ODE} 
    \Ddot{x} + \gamma(t) \dot x + \nabla U(x) = 0 \,,
\end{equation}
with initial condition $x(0) = x_0$, $\dot x(0) = 0$. Here $\gamma(t)$ is a user-specified damping parameter that can be time-homogeneous or time-inhomogeneous. When $U(x)$ is convex, the original Nesterov's iterations \eqref{eq:nag_convex} corresponds to the choice $\gamma(t) = \frac{3}{t}$. When $U(x)$ is further assumed to be $\lambda$-strongly convex, \eqref{eq:nag_sc} corresponds to the choice $\gamma(t) = 2\sqrt{\lambda}$. 

An important observation by \citet{maddison2018hamiltonian} is that the second-order ODE \eqref{eq:su_ODE} can be formulated as 
    \begin{equation}\label{eq:NG-matrix}
        \begin{bmatrix}
            \dot x \\
            \dot v
        \end{bmatrix} + \begin{bmatrix}
            0 \\ \gamma(t) v 
        \end{bmatrix} - \begin{bmatrix}
            0 & I \\
            -I & 0
        \end{bmatrix} \begin{bmatrix}
            \partial_x \mathcal{H}(x,v) \\
            \partial_v \mathcal{H}(x,v)
        \end{bmatrix} = 0\,,
    \end{equation}
where $x$ is the state variable and $v$ is the momentum variable. $\mathcal{H}(x,v) \defeq \|v\|^2/2 + U(x)$ is the Hamiltonian function. This formulation lifts the second-order ODE to a system of ODEs, and shows that the accelerated gradient flow can be viewed as a damped Hamiltonian flow.

\citet{zhang2018towards} proposed the Riemannian version of Nesterov's method for
geodesically smooth and strongly convex problems. Borrowing ideas from \citet{Su2016DE}, \citet{Alimisis2020Continuous} generalized Nesterov's ODE to a Riemannian manifold $(\mathcal{M},g)$: 
\begin{equation*}
    \ddot X + \gamma(t) \dot X + \mathrm{grad}_{\mathcal{M}} U(X)= 0\,,
\end{equation*}
where $X\in \mathcal{M}$ and $\mathrm{grad}_{\mathcal{M}}U(X)$ is the Riemannian gradient of $U$ at $X$. In particular, when $U$ is geodesically convex, \citet{Alimisis2020Continuous} proved an $\mathcal{O}(t^{-2})$ convergence rate, matching the convergence rate in Euclidean space. The choice of the damping parameter $\gamma(t)$ also depends on the lower bound of the sectional curvature of the manifold $\mathcal{M}$. 

\section{Nesterov Acceleration on Probability Manifolds}\label{sec:acceleration}

Given a target probability $\pi$ and an irreducible $Q$-matrix that is reversible w.r.t $\pi$, they induce a graph $G=(V,E,\omega)$ with a weight matrix $(\omega_{ij})=(\pi_i Q_{ij})$. For the probability manifold $\P(V)$, selecting an eligible activation function $\theta_{ij}$ in the form of \eqref{eq:phi-average} determines the Onsager's response matrix $\mathbb{K}(p)$ via \eqref{def:Onsager_response}, which induces a Riemannian metric on $\P(V)$. Within this framework, we construct a damped Hamiltonian dynamics for a pair of state and momentum variables $(p(t),\psi(t))$, 
    \begin{equation}\label{eq:AMH-matrix}
        \begin{bmatrix}
            \dot{p}(t)\\
            \dot{\psi}(t)
        \end{bmatrix}=\begin{bmatrix}
            0\\
            -\gamma(t) \psi(t)
        \end{bmatrix}+\begin{bmatrix}
        0 & I\\
        -I & 0
        \end{bmatrix}
        \begin{bmatrix}
            \partial_p \mathcal{H}(p(t), \psi(t)) \\
            \partial_\psi \mathcal{H}(p(t), \psi(t)) 
        \end{bmatrix},
    \end{equation}
where $\gamma(t)$ is a user-specified positive damping parameter. \eqref{eq:AMH-matrix} is a variant of \eqref{eq:NG-matrix}, inspired by previous works on the Euclidean space \citep{Su2016DE, maddison2018hamiltonian} and on the probability manifold \citep{Wang2022Accelerated, Shi2025Accelerating}.

The Hamiltonian $\mathcal{H}(p,\psi)$ is defined as follows
    \begin{equation}\label{eq: Ham-func}
        \mathcal{H}(p, \psi)  = \frac12 \psi \mathbb{K}(p)\psi^\top + \mathcal{U}(p) 
        = \frac{1}{4}\sum_{i=1}^n \sum_{j\neq i} \omega_{ij} \theta_{ij}\left(p\right) (\psi_i - \psi_j)^2  + \mathcal{U}(p), 
    \end{equation}
where $\mathcal{U}:\mathbb{P}(V) \rightarrow \mathbb{R}$ is a function, whose unique critical point is the target probability $\pi$. The first term in \eqref{eq: Ham-func} can be interpreted as the graphical analogue of kinetic energy. The second term can be viewed as the potential energy in the classical Hamiltonian system. 

We expand the matrix form \eqref{eq:AMH-matrix} into the template ODEs by plugging \eqref{eq: Ham-func}, 
    \begin{subequations}\label{eq:NODE}
    \begin{empheq}[left=\empheqlbrace]{align}
       & \frac{\mathrm{d} p_i}{\mathrm{d} t} + \sum_{j\neq i} \omega_{ij} \theta_{ij}(p) (\psi_j-\psi_i) = 0, \label{eq:NODE_1}\\
         & \frac{\mathrm{d}\psi_i}{\mathrm{d} t} + \gamma(t) \psi_i + \frac{1}{2}\sum_{j\neq i} \omega_{ij} \frac{\partial \theta_{ij}(p)}{\partial p_i} (\psi_i - \psi_j)^2   + \frac{\partial \mathcal{U}(p)}{\partial p_i}=   0 ,\label{eq:NODE_2}
    \end{empheq}
    \end{subequations}
where \eqref{eq:NODE_1} can be written compactly as $\frac{\mathrm{d}}{\mathrm{d}t}p=\psi \mathbb{K}(p)$ while the second equation is in general nonlinear. Users may specify the weight matrix $\omega$, the activation function $\theta(p)$, the potential $\mathcal{U}(p)$ and the damping parameter $\gamma(t)$ in this system. We will discuss examples shown in \Cref{tab:AMH} in the following subsections. 

In general, \eqref{eq:NODE} has these basic properties:
    \begin{itemize}
        \item $\mathcal{H}(p(t),\psi(t))$ is nonincreasing. Since 
        \begin{equation}\label{eq:Decay-Hamiltonian}
        \begin{split}
            \frac{\mathrm{d}\mathcal{H}(t)}{\mathrm{d}t}&=\partial_p \mathcal{H}\cdot \frac{\mathrm{d}p(t)}{\mathrm{d}t} + \partial_{\psi}\mathcal{H}\cdot \frac{\mathrm{d}\psi(t)}{\mathrm{d}t}=\partial_p\mathcal{H}\cdot \partial_{\psi}\mathcal{H}+\partial_{\psi}\mathcal{H}\cdot (-\gamma(t) \psi(t)-\partial_p\mathcal{H})\\
            &=-\gamma(t) \partial_{\psi}\mathcal{H}\cdot \psi(t)=-\gamma(t) (\psi(t)\mathbb{K}(p))\cdot \psi(t) \leqslant 0.
        \end{split}
        \end{equation}
        \item $p(t)\in \P(V)$ is provided with $p(0)\in \P(V)$. Since
            \[\dot{p}(t)\Id_n^\top = \partial_{\psi}\mathcal{H} \Id_n^\top = \psi(t)\mathbb{K}(p(t)) \Id_n^\top=0\,.\]
        \item $p(t)$ driven by \eqref{eq:NODE} must converge to $\pi$. From \eqref{eq:NODE_1}, for any stationary point $(p^*,\psi^*)$ to \eqref{eq:NODE}, it is necessary that $\psi_i^*=\psi_j^*=c$ for some constant $c$. Consequently, $\frac{\partial \mathcal U(p^*)}{\partial p_i}=0$ from \eqref{eq:NODE_2}. Note that $\pi$ is the unique critical point to $\mathcal{U}(p)$.
    \end{itemize}

To construct a sampling algorithm (a ``jump process'' of particles on $V$), we rewrite \eqref{eq:NODE_1} in the form of the forward master equation
    \begin{equation}\label{eq:Kolmogorov-Ham}
    \frac{\mathrm{d}}{\mathrm{d}t}p=p\bar{Q}^r_{\psi},
    \end{equation}
and discretize it in time to obtain the jump process. The derivation can be found in \Cref{app:derivation}.  Here, the notation $\bar{Q}^r_{\psi}$ is to resemble the $Q$-matrix used in \Cref{sub:MH} and is row-sum-zero, given by 
    \begin{equation}\label{eq:AMH-jump}
    \begin{aligned}
    \begin{bmatrix}
        -\displaystyle\sum_{j\neq 1}\dfrac{\omega_{1j} \theta_{1j}(p) (\psi_1-\psi_j)_{-}}{p_1} & \dfrac{\omega_{12}\theta_{12}(p) (\psi_2-\psi_1)_{+}}{p_1} & \cdots & \dfrac{\omega_{1n}\theta_{1n}(p) (\psi_n-\psi_1)_{+}}{p_1}\\
        \dfrac{\omega_{21}\theta_{21}(p) (\psi_1-\psi_2)_{+}}{p_2} & -\displaystyle \sum_{j\neq 2}\dfrac{\omega_{2j}\theta_{2j}(p) (\psi_2-\psi_j)_{-}}{p_2} & \cdots & \dfrac{\omega_{2n}\theta_{2n}(p) (\psi_n-\psi_2)_+}{p_2}\\
        \vdots & \cdots & \ddots &\vdots\\
        \dfrac{\omega_{n1}\theta_{n1}(p) (\psi_1-\psi_n)_+}{p_n} & \dfrac{\omega_{n2}\theta_{n2}(p) (\psi_2-\psi_n)_+}{p_n} & \cdots & -\displaystyle\sum_{j\neq n}\dfrac{\omega_{nj}\theta_{nj}(p) (\psi_n-\psi_j)_{-}}{p_n}
    \end{bmatrix}
    \end{aligned}
    \end{equation}
where $(\psi_i-\psi_j)_+:=\max\{\psi_i-\psi_j, 0\}$ and $(\psi_i-\psi_j)_{-}:=-\min\{\psi_i-\psi_j, 0\}$. Notably, if $p_i(t)=0$, then the $i$-th row is undefined. In order to use $P=I_n+\bar{Q}^r_{\psi}\Delta t$ as the transition probability matrix in the jump process, it is crucial to ensure the \textit{strict positivity} of $p_i(t)$ for any $i$ and any $t>0$. 

\begin{remark}
The reformulation of our dynamics in terms of the forward master equation \eqref{eq:Kolmogorov-Ham} enables a formal comparison with the MH method reviewed in \Cref{sub:MH}. Using the same candidate kernel $q_{ij}$ as in the acceptance-rejection matrix $A_{ij}$ in \eqref{eq:Q-MH} for the MH method, the corresponding acceptance-rejection matrix of our jump process takes the form
    \begin{equation}\label{eq:new-accept}
(\bar{A}_{\psi})_{ij} := \max \left\{\frac{\omega_{ij}\theta_{ij}(p)(\psi_j-\psi_i)}{p_i q_{ij}},0\right\}=\max \left\{\frac{\pi_i \theta_{ij}(p)}{p_i}(\psi_j-\psi_i),0 \right\}A_{ij}.
    \end{equation}
It is also possible to derive an equivalent form of the forward master equation under a different $Q$-matrix, resulting in a different acceptance-rejection matrix, see one example for the \texttt{Chi-squared} method in \Cref{sub:chi}. As noted in \cite{Le2026Swarm}, which is motivated by considerations analogous to those underlying our dynamics in the setting of interacting particle systems (i.e., swarm dynamics), although \eqref{eq:AMH-jump} admits a Markovian interpretation, the resulting process is nonlinear in the sense that it interacts with its law. Consequently, it requires the simulation of multiple chains or, equivalently, the evolution of empirical distributions. By contrast, the Metropolis-Hastings algorithm presented in \Cref{alg:MH} can run in a single long chain, even though it may suffer from poor initialization, have long burn-in time or show pseudo-convergence \citep{Geyer2011Introduction}. This limitation has motivated the development of the multiple-chain variants of MH algorithms; see, e.g., \citet{Gelman1992Multiple,  Jasra2007Population, Craiu2009Neighbor, Jacob2011Parallel, Calderhead2014Parallel}.
\end{remark}

In the following subsections, we present some examples of activation and potential functions.
In this paper, we focus mainly on the activation function $\theta_{ij}(p)=\dfrac{\frac{p_i}{\pi_i}-\frac{p_j}{\pi_j}}{f'(\frac{p_i}{\pi_i})-f'(\frac{p_j}{\pi_j})}$ motivated by Wasserstein gradient flows \eqref{eq:Wgf_precond}, and the corresponding potential $\mathcal{U}(p)$ either as $f$-divergences or relative Fisher information functions. 

For the former potential by $f$-divergences, from the perspectives outlined in \Cref{sub:Maas} and \Cref{sub:Nesterov}, the corresponding forward master equation \eqref{eq:Kolmogorov} can be interpreted as the gradient flow of the $f$-divergence in the probability manifold equipped with the graphical Wasserstein metric \eqref{def: Graph BB}. Consequently the damped Hamiltonian dynamics \eqref{eq:AMH-matrix} driven with $f$-divergence as its potential, can be regarded as the Nesterov accelerated flow in the probability manifold. We also note that the Hamiltonian flow with the relative Fisher information function in the Wasserstein space has been studied for Schr{\"o}dinger equations; 
see related motivations in \citep{PhysRev.150.1079,vonRenesse2012OTSchroedinger, CHOW20192440}. In this paper, we design a damped Hamiltonian flow on the simplex set, which uses the relative Fisher information functional to maintain the strict positivity of the probability function.

We present the explicit form of \eqref{eq:NODE} under the selected activation and potential functions in each subsequent subsections. The detailed derivations for \Cref{sub:chi} - \Cref{sub:Fisher} can be founded in \Cref{app:derivation}.

\subsection{\texorpdfstring{$\chi^2$}{} Divergence and Uniform Activation}\label{sub:chi}

In this subsection, we select $\theta_{ij}=1$ and $\chi^2$ divergence $\mathcal{U}(p)=\frac{1}{2}\sum_{i=1}^n \frac{(p_i-\pi_i)^2}{\pi_i}$. It is clear that $\mathcal{U}(p)\geqslant 0$ and $\mathcal{U}(\pi)=0$. Now \eqref{eq:NODE} reads
    \begin{equation}
    \label{eq:chi-1}
        \left\{\begin{aligned}
       & \frac{\mathrm{d} p_i}{\mathrm{d} t} =  \psi_i \sum_{j\neq i} \omega_{ij}-\sum_{j\neq i} \psi_j \omega_{ji},\\
         & \frac{\mathrm{d}\psi_i}{\mathrm{d}t} + \gamma(t) \psi_i + \frac{p_i}{\pi_i}-1=   0 . 
    \end{aligned}\right.
    \end{equation}
Denote by $[p,\psi]=[p_1,\ldots,p_n,\psi_1,\ldots,\psi_n]$, we rewrite \eqref{eq:chi-1} into a matrix equation:
\[\frac{\mathrm{d}}{\mathrm{d}t}[p,\psi]=[p,\psi]L+[0_{1\times n},\Id_{1\times n}],\]
where the matrix $L$ is given by
    \begin{equation}\label{eq:chi-matrix}
    L\defeq\begin{bmatrix}
    0_{n\times n} & -\textrm{diag}(\frac{1}{\pi})\\
    K & -\gamma(t) I_n
    \end{bmatrix},
    \end{equation}
where $\textrm{diag}(\frac{1}{\pi})$ represents a diagonal matrix with diagonal entries $(\frac{1}{\pi_1},\ldots,\frac{1}{\pi_n})$. Choosing $\theta_{ij}=1$ results in a constant matrix $K=(-\omega_{ij})=-\mathrm{diag}(\pi)Q$ in \eqref{def:Onsager_response}, which corresponds to a flat metric of $\P(V)$. We will refer to this method as the \texttt{Chi-squared} method. 

The vector form of \eqref{eq:chi-1} can be reformulated as the second-order ODE $\frac{\mathrm{d}^2 p}{\mathrm{d}t^2}+\gamma(t) \frac{\mathrm{d}p}{\mathrm{d}t} - pQ=0$, which is the Telegrapher's equation \citep{Archer2009Dynamical}. In coordinates, this equation reads
    \begin{equation}\label{eq:telegrapher}
    \frac{\mathrm{d}^2 p_i}{\mathrm{d}t^2}+\gamma(t) \frac{\mathrm{d} p_i}{\mathrm{d}t} - \sum_{j\neq i}(p_j Q_{ji}- p_i Q_{ij})=0.
    \end{equation}
 
\begin{remark}
From the time discretizations of the equation on the state variables \eqref{eq:chi-1} and of 
of the above Telegrapher's equation \eqref{eq:telegrapher}, we obtain a non-Markovian process (a jump process for every two steps, i.e., from $p^{(k)}$ to $p^{(k+2)}$),
    \begin{multline*}
    p_i^{(k+2)}=p_i^{(k)}+ (\Delta t)^2 \sum_{j\neq i}\left[Q_{ji} + \frac{2-\gamma(t) \Delta t}{\Delta t}\frac{\pi_j Q_{ji} (\psi_i^{(k)}-\psi_j^{(k)})_+}{p_j^{(k)}}\right]p_j^{(k)}\\ 
    + (\Delta t)^2 \left[-Q_{ij} - \sum_{j\neq i}\frac{\pi_i Q_{ij}(\psi_i^{(k)}-\psi_j^{(k)})_-}{p_i^{(k)}}\right] p_i^{(k)},
    \end{multline*}
see \Cref{app:derivation} for details. Thus the acceptance-rejection matrix can be formally defined as
    \begin{equation}\label{eq:chi-accept}(\bar{A}_{\psi})_{ij}=\left(1+\frac{2-\gamma(t) \Delta t}{\Delta t}\frac{\pi_j}{p_j}(\psi_i-\psi_j)_+ \right)A_{ij},
    \end{equation}
where $A_{ij}=\frac{Q_{ij}}{q_{ij}}$ is the acceptance-rejection matrix from MH from \eqref{eq:Q-MH}. Notably, \eqref{eq:chi-accept} is not the special case of \eqref{eq:new-accept} after plugging $\theta_{ij}=1$. The design of an optimal step size $\Delta t$ for the corresponding new numerical scheme is left for future work.

\end{remark}

\subsection{KL Divergence and Logarithmic Mean}\label{sub:KL}

In this subsection, we choose the logarithmic mean $\theta_{ij} (p)=\dfrac{\frac{p_j}{\pi_j}-\frac{p_i}{\pi_i}}{\log\frac{\pi_i p_j}{\pi_j p_i}}$ in \eqref{eq:log-mean} and the \textit{Kullback-Leibler} (KL) divergence $\mathcal{U}(p)=\sum_{i=1}^n p_i\log\frac{p_i}{\pi_i}$. It is clear that $\mathcal{U}(p)\geqslant 0$ and $\mathcal{U}(p)=0$ if and only if $p=\pi$. Under these choices, \eqref{eq:NODE} becomes
    \begin{equation}\label{eq:KL}
        \left\{
        \begin{aligned}
            &\frac{\mathrm{d}p_i}{\mathrm{dt}}=\sum_{j\neq i}\pi_i Q_{ij} \frac{\frac{p_i}{\pi_i}-\frac{p_j}{\pi_j}}{\log\frac{\pi_j p_i}{\pi_i p_j}} (\psi_i-\psi_j),  \\
            &\frac{\mathrm{d}\psi_i}{\mathrm{d}t}+\gamma(t) \psi_i +\log\frac{p_i}{\pi_i}+\frac{1}{2}\sum_{j\neq i}Q_{ij}\frac{\log\left(\frac{\pi_j p_i}{\pi_i p_j}\right)-1+ \left(\frac{\pi_j p_i}{\pi_i p_j}\right)^{-1}}{\log^2\left(\frac{\pi_j p_i}{\pi_i p_j}\right)} (\psi_i-\psi_j)^2 =0.
        \end{aligned}
        \right.
    \end{equation}
We will refer this method as the \texttt{KL} method in subsequent sections. The choice of $\theta$ and $\mathcal{U}(p)$ in the \texttt{KL} method is natural, since it can be regarded as the Nesterov accelerated flow for the KL divergence in the probability manifold $(\P(V),g_W)$.  However, the positivity of density $p(t)$ might not be preserved under the functional $\mathcal{U}(p)=\sum p_i\log\frac{p_i}{\pi_i}$. One has to select a careful numerical scheme in order to do the jump process proposed in the beginning of this section.  Therefore we propose the relative Fisher information in the following subsections to tackle this issue. 

\subsection{The Relative Fisher Information and Logarithmic Mean}\label{sub:Fisher}

We still select the logarithmic mean while choose the potential energy $\mathcal{U}(p)=\I(p \| \pi)=\frac{1}{4}\sum_{i=1}^n \sum_{j\neq i} \omega_{ij}(\log \frac{p_i}{\pi_i}-\log\frac{p_j}{\pi_j})(\frac{p_i}{\pi_i}-\frac{p_j}{\pi_j})$ as the relative Fisher information $\I(p\| \pi)$. Notably, $\I(p \| \pi)\geqslant 0$, with equality if and only if $\frac{p_i}{\pi_i}=\frac{1}{Z}$ for all $i$. The following lemma shows that the relative Fisher information has a unique critical point.

\begin{lemma}\label[lemma]{lem:unique}
Given a strictly positive target probability $\pi$ and an irreducible transition-rate matrix $Q$, then $p>0$ is a critical point of $\mathrm{I}(p\|\pi)$ in $\mathbb{P}(V)$, i.e., $\nabla_p \mathrm{I}(p\|\pi)= 0$, if and only if $p=\pi$.
\end{lemma}

\eqref{eq:NODE} takes the form
    \begin{equation}\label{eq:log-Fisher}
        \left\{
        \begin{aligned}
            &\frac{\mathrm{d}p_i}{\mathrm{dt}}=\sum_{j\neq i}\pi_i Q_{ij}\dfrac{\frac{p_j}{\pi_j}-\frac{p_i}{\pi_i}}{\log \frac{\pi_i p_j}{\pi_j p_i}}(\psi_i-\psi_j),\\
            &\frac{\mathrm{d}\psi_i}{\mathrm{d}t}+\gamma(t) \psi_i+\frac{1}{2}\sum_{j\neq i}Q_{ij}[\log\left(\frac{\pi_j p_i}{\pi_i p_j}\right)+1-\left(\frac{\pi_j p_i}{\pi_i p_j}\right)^{-1}]\\
            &+\frac{1}{2}\sum_{j\neq i}Q_{ij} \frac{\log\left(\frac{\pi_j p_i}{\pi_i p_j}\right)-1+ \left(\frac{\pi_j p_i}{\pi_i p_j}\right)^{-1}}{\log^2\left(\frac{\pi_j p_i}{\pi_i p_j}\right)}(\psi_i-\psi_j)^2=0.
        \end{aligned}
        \right.
    \end{equation}
We will refer to this approach as the \texttt{log-Fisher} method. One advantage of this method is that it preserves the positivity of $p(t)$ along the damped Hamiltonian flow, whereas the \texttt{KL} method does not.

\subsection{The Relative Fisher Information and Constant Activation}\label{sub:Fisher+1}

Now we consider a constant mobility weight matrix $(\theta_{ij})$ that does not depend on $p(t)$, whose entries are positive. Selecting $\mathcal{U}(p)=\frac{1}{4}\sum_{i,j=1}^n  \omega_{ij}\theta_{ij}(\log \frac{p_i}{\pi_i}-\log \frac{p_j}{\pi_j})^2 $, \eqref{eq:NODE} reads
    \begin{equation*}
        \left\{
        \begin{aligned}
            & \frac{\mathrm{d} p_i}{\mathrm{d} t} = \sum_{j\neq i} \pi_i Q_{ij}\theta_{ij} (\psi_i-\psi_j), \\
            & \frac{\mathrm{d}\psi_i}{\mathrm{d} t} + \gamma(t) \psi_i + \frac{\pi_i}{p_i}\sum_{j\neq i} Q_{ij}\theta_{ij}\log \left(\frac{\pi_j p_i}{\pi_i p_j}\right)=0.
        \end{aligned}
        \right.
    \end{equation*}
We refer this as the \texttt{con-Fisher} method. This potential $\mathcal{U}(p)$ inherits the advantages discussed in the \Cref{sub:Fisher} such as the unique critical point and preserving strict positivity, while simplifying $\psi$-equation \eqref{eq:NODE_2}. Additionally, the Hessian of $\mathcal{U}(p)$ takes a simpler form compared to the one in \texttt{log-Fisher}. A detailed analysis will be presented in \Cref{sub:damping}.

\section{Properties of aMCMC}\label{sec:property}
In this section we discuss key properties of aMCMC and compare advantages and disadvantages across different methods.

\subsection{Normalizing Constant}\label{sub:normalizing}

To design a damped Hamiltonian system \eqref{eq:AMH-matrix} that does not require the normalizing constant $Z$, it is necessary to verify the following conditions:
    \begin{itemize}
        \item The Onsager's response matrix $\mathbb{K}(p)$ defined in \eqref{def:Onsager_response} is independent of $Z$.
        \item $\pi$ is the unique critical point to $\mathcal{U}(p)$ in $\P(V)$. 
        \item The damping parameter $\gamma(t)$ does not rely on the knowledge of $Z$.
    \end{itemize}

In \texttt{Chi-Squared} and \texttt{con-Fisher} methods, the associated Onsager response matrix $(K_{ij})=(-\omega_{ij}\theta_{ij})$ is constant, where parameters are commonly selected as $\omega_{ij}=\pi_i Q^{\textrm{MH}}_{ij}$ and $\theta_{ij}=1$. Though $Q^{\text{MH}}$ itself is independent of $Z$, the resulting $K$ depends on $Z$. In contrast, \texttt{KL} and \texttt{log-Fisher} methods take the Onsager response matrix in the form of $K_{ij}(p)=-\pi_i Q^{\textrm{MH}}_{ij}\dfrac{\frac{p_i}{\pi_i}-\frac{p_j}{\pi_j}}{\log \frac{\pi_j}{\pi_i}\frac{p_i}{p_j}}=-Q_{ij}^{\text{MH}}\dfrac{p_i-p_j\frac{\pi_i}{\pi_j}}{\log \frac{\pi_j}{\pi_i}\frac{p_i}{p_j}} $, which does not depend on $Z$.

\subsection{Positivity}\label{sub:positivity}

The $\chi^2$ divergence $\frac{1}{2}\sum_{i=1}^n \frac{(p_i-\pi_i)^2}{\pi_i}$ and the KL divergence $\sum_{i=1}^n p_i\log\frac{p_i}{\pi_i}$ remain well-defined even if some state $p_i=0$. However, neither divergence explicitly prevents $p_i(t)$ from approaching zero during evolution. In the following theorem, we propose a general condition on the function $\mathcal{U}(p)$ to ensure $p_i(t)$ away from zero. The potentials in \texttt{log-Fisher} and \texttt{con-Fisher} are two examples.

\begin{theorem}\label{thm:positivity}
Given a strictly positive target probability $\pi$ and an irreducible transition-rate matrix $Q$, let $(p(t),\psi(t))$ be the solution to the damped Hamiltonian dynamics \eqref{eq:NODE} on a graph $G=(V,E,\omega)$, with the potential function $\mathcal{U}(p)$ in the form of 
\[\mathcal{U}(p)=\frac{1}{4}\sum_{i,j=1}^n F_{\frac{p_i}{\pi_i}}(\frac{p_j}{\pi_j})\omega_{ij},\] where $F_{r_0}(r):\R_+\mapsto \R$ is a family of functions with a parameter $r_0\in \R_{+}$, satisfying that
    \begin{enumerate}[label=\arabic*)]
        \item $F_{r_0}(r)\geqslant 0$ with the equality if and only if $r=r_0$.
        \item $\displaystyle\lim_{r\to 0+} F_{r_0}(r)=\infty$.
        \item $F_{r_0}(r)$ is strictly decreasing on $(0,r_0]$.
        \item Fix $r\in (0,r_0]$, if $r_1>r_0$, then $F_{r_1}(r)> F_{r_0}(r)$.
    \end{enumerate}

Suppose $\mathcal{H}(p(0),\psi(0))$ is bounded and $p(0)>0$, then there exists a positive constant $\varepsilon>0$, such that for any $t$ and for any $i$, $p_i(t)>\varepsilon$.
\end{theorem}

\begin{example}
In \Cref{sub:Fisher}, $\mathcal{U}(p)=\frac{1}{4}\sum_{i,j=1}^n (\log \frac{p_i}{\pi_i}-\log \frac{p_j}{\pi_j})(\frac{p_i}{\pi_i}-\frac{p_j}{\pi_j})\omega_{ij}$ corresponds to $F_{r_i}(r_j)=(\log r_i-\log r_j)(r_i-r_j)$ and $r_i=\frac{p_i}{\pi_i}$. 
\end{example}

\begin{example}
In \Cref{sub:Fisher+1}, $\mathcal{U}(p)=\frac{1}{4}\sum_{i,j=1}^n (\log \frac{p_i}{\pi_i}-\log \frac{p_j}{\pi_j})^2 \theta_{ij}\omega_{ij}$ chooses $F_{r_i}(r_j)=(\log r_i-\log r_j)^2 \theta_{ij}$ for some positive constant matrix $(\theta_{ij})$ and $r_i=\frac{p_i}{\pi_i}$.
\end{example}

One can verify conditions in \Cref{thm:positivity} are satisfied in both examples. The proof can be found in \Cref{app:proof}.

\subsection{Damping Parameter and Convergence Analysis}\label{sub:convergence}

With the Onsager's response matrix $\mathbb{K}(p)$ in \eqref{def:Onsager_response}, our proposed dynamics \eqref{eq:NODE} become
    \[
    \left\{
    \begin{aligned}
    \dot p(t) &= \psi(t) \mathbb{K}(p(t)) \,,  \\
    \dot \psi(t) &= -\frac{1}{2}\nabla_p \lan \psi(t), \psi(t) \mathbb{K}(p(t))  \ran - \gamma(t) \psi(t) - \nabla_p \mathcal{U}(p(t)) \,.
    \end{aligned}\right.
    \]
Recall the Hamiltonian is $\mathcal{H}(p,\psi) = \frac{1}{2} \psi \mathbb{K}(p) \psi^{\top} + \mathcal{U}(p).$

In this subsection, we establish the exponential convergence of the state variables for an explicitly chosen damping parameter by employing the Lyapunov analysis. Previous work designs the damping parameter $\gamma(t)$ using the minimum Hessian eigenvalue in Euclidean settings \citep{Nesterov1983Accelerating,Nesterov2013Introductory} or sectional curvatures on Riemannian manifolds \citep{Alimisis2020Continuous}. When the Onsager response matrix $\mathbb{K}(p)$ depends on $p$, analyzing sectional curvatures on the probability manifold $\P(V)$ endowed with the graphical Wasserstein metric $g_W$ becomes essential \citep{li2025geometriccalculationsprobabilitymanifolds}.

Our derivation relies on the assumptions that the Onsager's response matrix is constant (i.e, independent of $p(t)$), and the potential function $\mathcal{U}(p)$ is geodescially $\lambda$-strongly convex. Let the Onsager's response matrix $\mathbb{K}(p)=K$ for some positive semi-definite $K$. The corresponding Riemannian metric is flat. A function $\mathcal{U}$ on $\P(V)$ equipped with such a flat metric is geodesically $\lambda$-strongly convex for some constant $\lambda >0$, if
    \begin{equation}\label{eq:lambda-convex-flat}\mathcal{U}(\hat{p})\geqslant \mathcal{U}(p) + (\hat{p}-p)\nabla_p \mathcal{U}(p)^{\top} + \frac{\lambda}{2} \norm{\hat{p}-p}^2_{K^{\dagger}}.
    \end{equation}
Now, we define a Lyapunov function by
    \begin{equation}\label{eq:lyapunov_const_K}
    \mathcal{L}(t)=\frac{1}{2}e^{\sqrt{\lambda}t}\norm{\sqrt{\lambda}(p(t)-\pi)+ \psi K }^2_{K^{\dagger}} + e^{\sqrt{\lambda}t}(\mathcal{U}(p(t))-\mathcal{U}(\pi))\,.
    \end{equation}

\begin{theorem}\label{thm:convergence_rate}
Consider the probability manifold $\P(V)$ on a graph $G=(V,E,\omega)$ equipped with a flat metric \eqref{def:W_metric} or \eqref{eq:W_metric}, induced from a constant Onsager's response matrix \eqref{def:Onsager_response}. Given a function $\mathcal{U}$ that is geodesically $\lambda$-strongly convex for $\lambda >0$, select a time-homogeneous damping parameter $\gamma(t)=2\sqrt{\lambda}$, and let $(p(t),\psi(t))$ be the solution to the proposed damped Hamiltonian dynamics \eqref{eq:AMH-matrix}. Then the Lyapunov function \eqref{eq:lyapunov_const_K} is non-increasing. Furthermore, $\mathcal{U}(p(t))-\mathcal{U}(\pi)\leqslant \mathcal{O}(e^{-\sqrt{\lambda}t})$.
\end{theorem}

\begin{proof}
\begin{align}\label{eq:decay_lyapunov_const_K}
   e^{-\sqrt{\lambda}t} \frac{\mrd \mathcal{L}}{\mrd t} &= \frac{1}{2}\sqrt{\lambda}\left[\lambda \norm{p-\pi}^2_{K^{\dagger}} + \norm{\psi}^2_{K} + 2\sqrt{\lambda}\psi (p-\pi)^\top \right] + \sqrt{\lambda}(\mathcal{U}(p)-\mathcal{U}(\pi)) \nonumber \\
   &\qquad + \nabla_p \psi K \mathcal{U}(p)^\top- (\sqrt{\lambda} \psi K + \nabla_p \mathcal{U}(p) K)K^{\dagger}(\sqrt{\lambda}(p-\pi)+\psi K)\nonumber \\
   &= \frac{1}{2}\lambda^{\frac{3}{2}} \norm{p-\pi}^2_{K^{\dagger}} + \frac{\sqrt{\lambda}}{2}\norm{\psi}^2_{K} + \lambda \psi (p-\pi)^\top +  \sqrt{\lambda}(\mathcal{U}(p)-\mathcal{U}(\pi))+ \psi K \nabla_p \mathcal{U}(p)^\top\nonumber \\
   &\qquad -\lambda \psi (p-\pi)^\top - \sqrt{\lambda}\norm{\psi}^2_{K}-\sqrt{\lambda}(p-\pi) \nabla_p \mathcal{U}(p)^\top - \psi K \nabla_p \mathcal{U}(p)^\top\nonumber \\
   &= -\frac{1}{2}\sqrt{\lambda}\norm{\psi}^2_{K} + \sqrt{\lambda}\left[\mathcal{U}(p)-\mathcal{U}(\pi)+(\pi-p)\nabla_p \mathcal{U}(p)^\top + \frac{\lambda}{2} \norm{\pi-p}_{K^{\dagger}}^2\right]\,.
\end{align}
Observe that the first term in \eqref{eq:decay_lyapunov_const_K} is non-positive. The second term in \eqref{eq:decay_lyapunov_const_K} is also non-positive by the geodesically $\lambda$-strongly convex assumption on $\mathcal{U}$. Thus, we conclude that $
\frac{\mrd \mathcal{L}}{\mrd t} \leqslant 0\,$, 
which implies 
    \begin{equation*}
        \mathcal{U}(p(t))-\mathcal{U}(\pi) \leqslant e^{-\sqrt{\lambda}t} \mathcal{L}(t) \leqslant e^{-\sqrt{\lambda}t} \mathcal{L}(0) = \mathcal{O}(e^{-\sqrt{\lambda}t}) \,.
    \end{equation*}

\end{proof}

\subsection{Geodesic Convexity, Hessian and Eigenvalues}\label{sub:damping}

According to \Cref{thm:convergence_rate}, under the flat metric $K$, the damping parameter $\gamma(t)$ can be designed based on the constant $\lambda>0$ that characterizes the convexity of $\mathcal U$. \Cref{app:Hessian} reviews the relation between the Hessian and geodesic convexity of a potential $\mathcal{U}$, following \cite{Mielke2013Geodesic}. It suffices to verify condition \eqref{eq:lambda-convex-flat} or to establish that $\mathrm{D}^2\mathcal{U}\succeq\lambda K^{\dagger}$ for some constant $\lambda>0$. In this section, we study how to find the constant $\lambda$.

\subsubsection{\texttt{Chi-squared} method}

To check if $\mathcal{U}(p)=\frac{1}{2}\sum_{i=1}^n \frac{(p_i-\pi_i)^2}{\pi_i}$ is geodesically $\lambda$-strongly convex, we need to find a positive constant $\lambda>0$ such that
    \[
    \begin{aligned}
    &\mathcal{U}(\hat{p})-\mathcal{U}(p)-(\hat{p}-p) (\nabla_p \mathcal{U}(p))^{\top}=\sum_{i=1}^n (\hat{p}_i-p_i)\frac{\hat{p}_i+p_i -2\pi_i}{2\pi_i}-\sum_{i=1}^n (\hat{p}_i -p_i)\frac{p_i-\pi_i}{\pi_i}\\
    = & \sum_{i=1}^n (\hat{p}_i - p_i)\frac{\hat{p}_i-p_i}{2\pi_i}=(\hat{p}-p)\textrm{diag}(\frac{1}{2\pi})(\hat{p}-p)^{\top}\geqslant \frac{\lambda}{2}(\hat{p}-p)K^{\dagger}(\hat{p}-p)^{\top},
    \end{aligned}
    \]
where $K=(-\omega_{ij})=-\mathrm{diag}(\pi)Q$. A sufficient condition is that $0<\lambda\leqslant \frac{\min \frac{1}{\pi_i}}{\lambda_{\textrm{max}}((-\omega)^{\dagger})}$. Given that $\pi$ is assumed to be strictly positive and $\omega$ is the weight matrix associated to a reversible and irreducible Markov chain, we know $\lambda$ exists.

Given a transition rate matrix 
$Q$, for instance $Q^{\textrm{MH}}$ in MH, we compare the convergence of the classical MCMC method and its accelerated counterpart driven by the $\chi^2$ divergence, using spectral analysis. In particular, when a constant damping parameter is applied, the matrix $L$ defined in \eqref{eq:chi-matrix} is a constant matrix, whose eigenvalues can be computed directly. In Lemma \ref{lem: eigen} we derive an explicit relationship between the spectrum of 
$L$ and
$Q$, which enables a quantitative comparison of convergence rates in terms of the spectral gap in \Cref{thm:chi}. Background material on the spectral properties of general irreducible and reversible transition rate matrices is provided in \Cref{app:Q-matrix}.

\begin{lemma}\label{lem: eigen}
Let $\gamma(t)=d$ for some $d >0 $.
    \begin{itemize}
        \item If $\mu$ is a real eigenvalue of $L$, then $\alpha = \mu (d+ \mu)$ is a real eigenvalue of $Q$. 
        \item If $\alpha$ is a real eigenvalue of $Q$, then there exist (possibly complex) eigenvalues $\mu$ of $L$ such that $\mu(d+\mu)=\alpha$.
    \end{itemize} 
\end{lemma}

\begin{proof}
Given square matrices $A,B,C,D$,  if $C,D$ are commutative, then we have the simple determinant formula
    \[\textrm{det}\begin{pmatrix}
        A & B\\
        C & D
    \end{pmatrix}=\textrm{det}(AD-BC).\]
Take $A=0_{n\times n}$, $B=-\textrm{diag}(\frac{1}{\pi})$, $C=-\textrm{diag}(\pi)Q$, and $D=-d I_n$ from $L$ matrix in \eqref{eq:chi-matrix}, then 
    \begin{equation}
    \label{eq:eig-relation}
    \textrm{det}(L-\mu I_{2n})=\textrm{det}((A-\mu I_{n})(D-\mu I_{n})-BC)=\textrm{det}(\mu(d +\mu) I_{n}-Q).
    \end{equation}
Obviously, if $\mu$ is a real eigenvalue of $L$, then  $\alpha=\mu(d +\mu)$ is a real eigenvalue of $Q$. On the other hand, if $\alpha$ is a real eigenvalue of $Q$, there exists (possibly complex) $\mu$ such that $\mu(d+\mu)=\alpha$, and
    \[0=\textrm{det}(\alpha I_{n} -Q)=\textrm{det}(\mu(d+\mu)I_{n}-Q)=\textrm{det}(L-\mu I_{2n }).\]
\end{proof}
In particular, $0$ is an eigenvalue of both $Q$ and $L$.

\begin{theorem}\label{thm:chi}
Let $\alpha_*$ be the largest negative eigenvalue (i.e., the spectral gap) of $Q$,
\[\alpha_*=\max\left\{\alpha< 0\mid (Q-\alpha I)\textrm{~is not injective~}\right\}.\]
If $\abs{\alpha_*}<1$, then there exists damping parameter $\gamma(t)=d\in [2\sqrt{\abs{\alpha_*}},\abs{\alpha_*}+1)$, such that the largest negative eigenvalue $\mu_*$ of $L$ satisfies $\mu_*<\alpha_*$.
\end{theorem}
\begin{proof}
Let $\alpha$ be an eigenvalue of $Q$, then from \eqref{eq:eig-relation} there exists $\mu_{1}(\alpha)=\frac{-d+ \sqrt{d^2+4\alpha}}{2}, \mu_{2}(\alpha)=\frac{-d-\sqrt{d^2+4\alpha}}{2}$ which are (possibly complex) eigenvalues of $L$. In particular, $\mu_1(0)=0$ and $\mu_2(0)=-d$. Let $\alpha_{n-1}\leqslant \alpha_{n-2}\leqslant\cdots\leqslant \alpha_1<0=\alpha_0$ denote the ordered eigenvalues of $Q$. Then $\alpha_*=\alpha_1$. Fix $\alpha \in [\alpha_{n-1}, \alpha_1]$. If $d\geqslant 2\sqrt{\abs{\alpha}}$, then $\mu_2(\alpha)\leqslant -\frac{d}{2}\leqslant \mu_1(\alpha)<0$. If $d< 2\sqrt{\abs{\alpha}}$, $\textrm{Re}(\mu_1(\alpha))=\textrm{Re}(\mu_2(\alpha))=\frac{-d}{2}$. Since $\abs{\alpha_*}<1$, we can pick $d\in [2\sqrt{\abs{\alpha_*}},\abs{\alpha_*}+1)$. We consider two cases: 
\begin{itemize}
\item If there exists some $k$ such that $d\in [2\sqrt{\abs{\alpha_k}}, 2\sqrt{\abs{\alpha_{k+1}}})$, then \[\mu_1(\alpha_1)=\max \left\{\mu_1(\alpha_1),\mu_2(\alpha_1),\cdots,\mu_1(\alpha_k),\mu_2(\alpha_k)\right\}\geqslant -\frac{d}{2},\]
while $\textrm{Re}(\mu_1(\alpha_j))=\textrm{Re}(\mu_2(\alpha_j))=-\frac{d}{2}$ for $j\geqslant k+1$. Thus $\mu_*=\mu_1(\alpha_*)=\frac{-d+\sqrt{d^2+4\alpha_*}}{2}$.
\item If $2\sqrt{\abs{\alpha_{n-1}}}\leqslant d<\abs{\alpha_*}+1$, then $\mu_*=\mu_1(\alpha_*)=\frac{-d+\sqrt{d^2+4\alpha_*}}{2}$.
\end{itemize}

In either case, note that $\mu_*=\frac{-d+\sqrt{d^2+4\alpha_*}}{2}$ is increasing w.r.t $d$. Thus $d<\abs{\alpha_*}+1$ implies that $\mu_*<\alpha_*$. 
\end{proof}

\begin{remark}
The optimal choice is in the form of $d=2\sqrt{\abs{\alpha_*}}$, which resembles the choice of the damping parameter in the classical Nesterov's acceleration \eqref{eq:su_ODE}, and also our choice of $\gamma(t)$ in \Cref{thm:convergence_rate}. It is worth noting that $\alpha_*$ refers to the largest negative eigenvalue of $Q$ in \eqref{eq:Kolmogorov}, while $\lambda$ in the Nesterov's method as well as \Cref{thm:convergence_rate} refers to the $\lambda$-strong convexity of the function, satisfying $0<\lambda\leqslant \frac{\min_i \frac{1}{\pi_i}}{\lambda_{\mathrm{max}}((-\omega)^{\dagger})}$. 

With the optimal damping parameter, $\mu_*=-\sqrt{\abs{\alpha_*}}$ is the optimal convergence rate of \texttt{Chi-squared} method in terms of spectral analysis.

\end{remark}

\subsubsection{\texttt{con-Fisher} method}

To check if $\mathcal{U}(p)=\I(p\|\pi)=\frac{1}{4}\sum_{i=1}^n \sum_{j\neq i}\omega_{ij}\theta_{ij}(\log \frac{p_i}{\pi_i}-\log \frac{p_j}{\pi_j})^2$ is geodesically $\lambda$-strongly convex, we need to find a positive constant $\lambda$ such that
    \[
    \begin{aligned}
        &\I(\hat{p}\|\pi)-\I(p\|\pi)-(\hat{p}-p)\nabla_p \I(p\|\pi)\\
        =&\frac{1}{4}\sum_{i=1}^n \sum_{j\neq i}\omega_{ij}\theta_{ij}[\log \frac{\pi_j^2 p_i \hat{p}_i }{\pi_i^2 p_j \hat{p}_j}\log \frac{\hat{p}_i p_j}{p_i \hat{p}_j}]-\sum_{i=1}^n (\hat{p}_i-p_i)\frac{\pi_i}{p_i}\sum_{j\neq i}Q_{ij}\theta_{ij}\log \frac{\pi_j p_i }{\pi_i p_j }\\
        =&\sum_{i=1}^n\sum_{j\neq i}\omega_{ij}\theta_{ij}\left[\frac{1}{4}\log \left(\frac{\hat{p}_i p_j}{p_i \hat{p}_j }\right)^2 + \frac{1}{2}\log \left(\frac{\hat{p}_i p_j}{p_i \hat{p}_j }\right)\log \frac{\pi_j p_i}{\pi_i p_j}-(\frac{\hat{p}_i}{p_i}-1)\log \frac{\pi_j p_i}{\pi_i p_j}\right]\\
        \geqslant &\frac{\lambda}{2}(\hat{p}-p)K^{\dagger}(\hat{p}-p)^{\top},
    \end{aligned}
    \]
where $K=(-\omega_{ij}\theta_{ij})$ for some constant matrices $(\omega_{ij})$ and $(\theta_{ij})$. It appears nontrivial to us to find a sufficient condition. Thus we compute the Hessian $\mathrm{D}^2\mathcal{U}$ instead. \Cref{app:Hessian} provides detailed derivations of the explicit form of the Hessian and its related properties. 

The positive constant $\lambda$ can be obtained by studying the Rayleigh quotient problem
    \begin{equation}\label{eq:rayleigh}
    \min_{\psi \Id_n^{\top}=0} \dfrac{\psi K \mathrm{D}^2\mathcal{U}(p)K^{\top}\psi^{\top}}{\psi K\psi^{\top}},
    \end{equation}
which can be reformulated as an eigenvalue problem from a computational standpoint; see \Cref{app:Hessian} for details. 

\begin{lemma}\label[lemma]{lem:lambda-strong}
Given a strictly positive target probability $\pi$ and an irreducible transition-rate matrix $Q$, we consider the probability manifold $\P(V)$ on a graph $ G=(V,E,\omega)$, with the potential $\mathcal{U}(p)=\I(p\|\pi)=\frac{1}{4}\sum_{i=1}^n \sum_{j\neq i}\omega_{ij}\theta_{ij}(\log \frac{p_i}{\pi_i}-\log \frac{p_j}{\pi_j})^2$, where $(\theta_{ij})$ is a constant matrix whose entries are positive.

Let $K=(K_{ij})$ be the associated Onsager's response matrix, where $K_{ij}=-\omega_{ij}\theta_{ij}$ and $K_{ii}=-\sum_{j\neq i}K_{ij}$. Then there exists a positive constant $\lambda$ such that $\mathrm{D}^2 \mathcal{U}(p)\vert_{p=\pi}\geqslant \lambda K^{\dagger}$.
\end{lemma}
Consequently, we may use this $\lambda$ computed from \texttt{con-Fisher} at stationary distribution $p=\pi$ to construct the asymptotical limit of the damping parameter $\gamma(t)=2\sqrt{\lambda}$ in the \texttt{log-Fisher} method.

\section{Numerical Schemes}\label{sec:alg}

In this section, we outline numerical schemes employed in subsequent numerical experiments for the ODE integrator of \eqref{eq:NODE} and the jump process associated with \eqref{eq:Kolmogorov-Ham}, \eqref{eq:AMH-jump} in aMCMC.

\subsection{ODE Integrator}

The classical Hamiltonian Monte Carlo (HMC), requires specialized discretization schemes, such as the symplectic Euler scheme for volume-preserving or the Leapfrog scheme for quantity-conserving \citep[see][]{Neal2011HMC}. For our proposed damped Hamiltonian dynamics \eqref{eq:NODE}, we adopt the staggered scheme with splitting methods. Specifically, given a prescribed damping parameter $\gamma(t)$, we can express \eqref{eq:NODE} as:
    \[
    \left\{
    \begin{aligned}
        \frac{\mathrm{d}p(t)}{\mathrm{d}t}&=A(p(t),\psi(t)),\\
        \frac{\mathrm{d}\psi(t)}{\mathrm{d}t}&=B(p(t),\psi(t),\gamma(t)),
    \end{aligned}
    \right.
    \]
where $A(p(t),\psi(t))=\partial_{\psi} \mathcal{H}(p(t),\psi(t))$ and $B(p(t),\psi(t),\gamma(t))=-\gamma(t)\psi(t)-\partial_{p} \mathcal{H}(p(t),\psi(t))$. The associated staggered scheme is given by
    \begin{subequations}\label{eq:symplectic}
    \begin{empheq}[left=\empheqlbrace]{align}
    &\frac{p^{(k+1)}-p^{(k)}}{\Delta t} = A(p^{(k)},\psi^{(k)}),\label{eq:forward-master}\\
    &\frac{\psi^{(k+1)}-\psi^{(k)}}{\Delta t}  = B(p^{(k+1)},\psi^{(k)},\gamma^{(k)})\label{eq:psi}.
    \end{empheq}
    \end{subequations}
When $\mathbb{K}$ does not depend on $p$, the Hamiltonian $\mathcal{H}(p,\psi)$ is separable. Thus \eqref{eq:symplectic} is indeed the symplectic Euler scheme. Note that \eqref{eq:symplectic} involves designing an interacting particle system, which requires simultaneously integrating the coupled system of all state variables and momentum variables. In contrast, the ODE solver of the forward master equation \eqref{eq:Kolmogorov} in the MH can integrate each state variable independently.

Since the Hamiltonian $\mathcal{H}(p,\psi)$ is known to decay along the dynamics according to \eqref{eq:Decay-Hamiltonian}, it is therefore natural to employ a numerical scheme that preserves this dissipation property. In our numerical experiments, we select a sufficiently small step size to numerically approximate this continuous-time decay property.

\subsection{Jump Process}

{For the forward master equation \eqref{eq:NODE_1}, the equivalent formulation \eqref{eq:Kolmogorov-Ham} provides a jump-process representation for the state
variables. Consequently, both the classical MH update and the proposed aMCMC
admit a transition matrix of the form $P=I_n+Q\Delta t$, where MH takes $Q=Q^{\mathrm{MH}}$ in \eqref{eq:Q-MH} and aMCMC takes $Q=\bar{Q}^r_{\psi}$ in \eqref{eq:AMH-jump}. The step size $\Delta t$ is chosen sufficiently small to ensure that all entries of $P$ are nonnegative. 

The construction of $\bar{Q}^r_{\psi}$ requires components $p_i(t)$ to remain strictly positive. This property is guaranteed by \Cref{thm:positivity} for \texttt{log-Fisher} and \texttt{con-Fisher} methods. In practical computations, we additionally employ \textit{restart} techniques, detailed in Section 6, to correct potential degeneracies $p_i(t)=0$ due to finite-particle random sampling. 

\begin{algorithm}[hp!tb]
\caption{Metropolis--Hastings Sampling in Multiple Chains\label{alg:MH2}}
\SetKwInOut{Input}{Input}
\SetKwInOut{Output}{Output}

\Input{Initial distribution $p^{(0)}$; transition rate matrix $Q^{\mathrm{MH}}$; total number of states $n$; number of particles $M$; step size $\Delta t$; total iterations $N$.}
\Output{Terminal distribution $p^{(N)}$.}

\BlankLine
\textbf{Initialize:} Transition probability matrix
\[
P^{\mathrm{MH}} \gets I_n + Q^{\mathrm{MH}} \Delta t.
\]
\textbf{Sample:} Draw $M$ particles according to $p^{(0)}$, and record state counts:
\[
\texttt{bin}(i) \gets \#\{\text{particles in state } i\}, \quad i = 1, \dots, n.
\]

\BlankLine
\For{$k \gets 1$ \KwTo $N$}{
  \ForPar{$s \gets 1$ \KwTo $n$}{
    \tcp{Simulate transitions from state $s$}
    Draw $\texttt{bin}(s)$ samples from $\{1,\dots,n\}$ with probabilities $P^{\mathrm{MH}}(s,:)$\;
    Record the number of arrivals into each state $i$:
    \[
    \texttt{tmp}(s,i) \gets \#\{\text{particles jumping from } s \text{ to } i\}.
    \]
  }
  \tcp{Aggregate arrivals across all source states}
  \[
  \texttt{bin}(i) \gets \sum_{s=1}^n \texttt{tmp}(s,i), \quad i = 1, \dots, n.
  \]
}

\BlankLine
\textbf{Return:} Terminal distribution
\[
p^{(N)} \gets \frac{1}{M} \, \texttt{bin}.
\]
\end{algorithm}

In our numerical experiments, we consider two different comparison settings. For experiments under a fixed number of iterations, we implement the MH jump process using multiple chains, as described in \Cref{alg:MH2}. This provides a more direct and fair comparison with the particle implementation of aMCMC: both methods evolve an empirical distribution represented by a population of samples, and one iteration corresponds to one jump process following the transition matrix respectively. More importantly, the optimization viewpoint underlying our
acceleration theory concerns the convergence of the probability dynamics with
respect to evolution time, or equivalently the number of iterations. Thus, the multi-chain MH implementation is the natural baseline for examining the iteration-wise acceleration predicted by the theory. For the multimodal
targets considered below, the use of multiple MH chains also reduces the sensitivity to the initialization of a single long chain, as discussed in Remark~1.

In contrast, a single long MH chain is the standard implementation used in
many practical MCMC applications. We therefore also compare single-chain MH
with the $M$-particle \texttt{log-Fisher} method under a fixed wall-clock budget in \Cref{sub:image} and \Cref{sub:Ising}. This comparison accounts for the additional
computational cost of aMCMC, since aMCMC transition-rate matrix $\bar{Q}^r_{\psi}$ is time-inhomogeneous and depends on both the empirical distribution $p(t)$ and momentum variables $\psi(t)$, indicating that particles are interacting and exchanging information from the histograms. While the state variables evolve via a jump process, there is no corresponding jump-process formulation for the momentum variable, thus we employ the ODE integrator to solve \eqref{eq:psi}.
}

\begin{algorithm}[hp!]
\caption{Accelerated MCMC (aMCMC) Sampling\label{alg:aMCMC}}
\SetKwInOut{Input}{Input}
\SetKwInOut{Output}{Output}

\Input{
Initial distribution $p^{(0)}$; transition rate matrix $Q$;
total number of states $n$; number of particles $M$;
step size $\Delta t$;
warm-start iterations $L$; total iterations $N$.
}
\Output{Terminal distribution $p^{(N)}$.}

\textbf{Warm-start:}
Run $L$ iterations of Algorithm~\ref{alg:MH2} starting from $p^{(0)}$, and obtain $p^{(L)}$. Initialize $\psi^{(L)}$ using the strategy depending on the mobility weight matrix $(\theta_{ij})$.

\textbf{Initialize:} Time-inhomogeneous transition probability matrix $P \gets I_n + \bar{Q}^r_{\psi}\,\Delta t$, where the new transition rate matrix $\bar{Q}^r_{\psi}$ in \eqref{eq:AMH-jump} is constructed from $p^{(L)}$ and $\psi^{(L)}$.

\textbf{Sample:}
Draw $M$ particles according to $\rho^{(L)}$, and record state counts
\vspace{-0.5em}
\[
\texttt{bin}(i) \gets \#\{\text{particles in state } i\}, \quad i = 1,\dots,n.
\]

\For{$k \gets L+1$ \KwTo $N$}{
  \ForPar{$s \gets 1$ \KwTo $n$}{
    \tcp{Simulate transitions from state $s$}
    Draw $\texttt{bin}(s)$ samples from $\{1,\dots,n\}$ with probabilities $P(s,:)$\;
    Record the number of arrivals into each state $i$:
    \vspace{-0.5em}
    \[
    \texttt{tmp}(s,i) \gets \#\{\text{particles jumping from } s \text{ to } i\}.
    \]
  }
  \tcp{Aggregate arrivals across all source states}
  $\texttt{bin}(i) \gets \sum_{s=1}^n \texttt{tmp}(s,i), \quad i = 1,\dots,n.$ 
  
  \tcp{Restart mechanism}
  \eIf{restart condition is triggered {under a threshold}}{
    {Add particles to all $\texttt{bins}$ whose particles are less than a threshold}\;
    Update $p^{(k)}_i\gets \frac{1}{M_k}\,\texttt{bin}(i),$where $M_k$ denotes the current total number of particles at the $k$-th step\;
    Reinitialize all momentum variables $\psi^{(i)}$ according to the restart strategy.
  }{
  Update $p^{(k)}_i\gets \frac{1}{M_k}\,\texttt{bin}(i)$\;
  Update $\psi^{(k)}$ via its ODE solver.
  }
  Compute $\bar{Q}_{\psi}^r$ accordingly. 

  \tcp{Adaptive step-size control (local backtracking)}
  $\Delta t^{(k)} \gets \Delta t$\;
  \While{$P = I_n + \bar{Q}^r_{\psi}\,\Delta t^{(k)}$ is not a valid transition matrix}{
    $\Delta t^{(k)} \gets \Delta t^{(k)} / 10$\;
    $P \gets I_n + \bar{Q}^r_{\psi}\,\Delta t^{(k)}$\;
  }

}

\textbf{Return:}
Terminal distribution $p^{(N)}$.
\end{algorithm}

\subsection{Initialization and Restart}\label{sub:restart}

For the classical MCMC with the MH update, we initialize the state distribution $p(0)$ by a uniform distribution. {While it is well known that MCMC methods generally require a \textit{burn-in} period, where the first few iterations are discarded to mitigate the effects of a poor starting point \citep{Geyer2011Introduction}, we retain all iterations in our experiments reported in \Cref{sec:num}.} For aMCMC, in addition to initializing the state variables with a uniform distribution, the momentum variables must also be specified. We observe the following relationship: 
\begin{itemize}
\item $\theta_{ij}=1$ in \texttt{Chi-squared} and \texttt{con-Fisher} methods. Under the choice $\psi_j=-\frac{p_j}{\pi_j}$, \eqref{eq:NODE_1} becomes
    \[
    \frac{\mathrm{d}p_i}{\mathrm{d}t}=\psi_i\sum_{j\neq i}\omega_{ij} -\sum_{j\neq i}\psi_j \omega_{ji}=-\frac{p_i}{\pi_i}\sum_{j\neq i}\pi_i Q_{ij}+\sum_{j\neq i}\frac{p_j}{\pi_j}\pi_j Q_{ji}=\sum_{j\neq i}p_jQ_{ji}-p_i Q_{ij},
    \]
which is precisely the forward master equation for the classical MCMC. 

\item $\theta_{ij}=\frac{\frac{p_j}{\pi_j}-\frac{p_i}{\pi_i}}{\log \frac{\pi_i p_j}{\pi_j p_i}}$ in \texttt{KL} and \texttt{log-Fisher} methods. Under the choice $\psi_j=-\log\frac{p_j}{\pi_j}$, \eqref{eq:NODE_1} becomes
    \[
    \frac{\mathrm{d}p_i}{\mathrm{d}t}=\sum_{j\neq i}(-\log\frac{p_i}{\pi_i}+\log\frac{p_j}{\pi_j})\frac{\frac{p_i}{\pi_i}-\frac{p_j}{\pi_j}}{\log\frac{p_i}{\pi_i}-\log\frac{p_j}{\pi_j}}\pi_i Q_{ij}=\sum_{j\neq i}\frac{p_j}{\pi_j}\pi_j Q_{ji} - \frac{p_i}{\pi_i} \pi_i Q_{ij},
    \]
which recovers the forward master equation for the classical MCMC. 
\end{itemize}
Consequently, the state variables in \eqref{eq:NODE_1} evolve identically to classical MCMC, once employing the above assignments of momentum variables. This observation allows us to employ classical MCMC as a warm-start for aMCMC. Specifically, after a prescribed number of MH iterations, we initialize the momentum variables, by $\psi_j(0)=-\frac{p_j(0)}{\pi_j}$ for \texttt{Chi-squared} and \texttt{con-Fisher} methods, or by $\psi_j(0)=-\log\frac{p_j(0)}{\pi_j}$ for \texttt{KL} and \texttt{log-Fisher} methods. 

Given that the target distribution $\pi$ is strictly positive and the Markov chain is reversible and irreducible, \Cref{thm:positivity} ensures strict positivity of the state variables $p(t)$ for both the \texttt{log-Fisher} and \texttt{con-Fisher} methods. It suggests that when when $p_i(t)$ approaches zero, there is a sufficiently large momentum to pull $p_i(t)$ away from zero. Nonetheless, in numerical simulations, accumulation of truncation errors and large step sizes could lead to instances where the state variables become zero or even negative. 

For the ease of our discussion, denote by $p_{\textrm{i,ode}}^{(k)}$ the numerical value of $p_i(t)$ in \eqref{eq:NODE} at the $k$-th iteration using Euler scheme. Denote by  $p_{\textrm{i,jump}}^{(k)}$ the empirical probability density of state $i$ at the $k$-th iteration, which is calculated by the ratio between the number of particles in state $i$ at the $k$-th iteration and the total number of particles $M$. We first opt for adaptive step sizes (see line 20-23 in \Cref{alg:aMCMC}) to  prevent any state variables $p_{\textrm{i,ode/jump}}^{(k)}$ from becoming negative. In practice, this is usually sufficient when solving $p_{\textrm{i,ode}}^{(k)}$ in \eqref{eq:NODE} with \texttt{log-Fisher} and \texttt{con-Fisher} methods using Euler scheme. However, due to the random sampling following the transition probability matrix $P=I_n + \bar{Q}^r_{\psi}\Delta t$, a small sample size $M$ may cause the empirical particle density $p_{\textrm{i,jump}}^{(k)}$ to become zero even when $p_{\textrm{i,ode}}^{(k)}>0$.

Second, if $p_{\textrm{i,ode/jump}}^{(k)}=0$ or is under a threshold after some iterations, we employ a simple \textit{restart} mechanism by the MH update that works for any aMCMC method, described below:

\textit{Restart} (ODE solver): if $p_{\textrm{i,ode}}^{(k)}=0$ for some state $i$, for all state $j$ we set $\psi_{\textrm{j,ode}}^{(k)}=-\frac{p_{\textrm{j,ode}}^{(k)}}{\pi_j}$ in \texttt{Chi-squared} and \texttt{con-Fisher} methods, or $\psi_{\textrm{j,ode}}^{(k)}=-\log\frac{p_{\textrm{j,ode}}^{(k)}}{\pi_j}$ in \texttt{KL} and \texttt{log-Fisher} methods. This choice ensures the $(k+1)$-th iteration on the state variables effectively runs \textit{one} iteration of the MH update, thereby pulling $p_{\textrm{i,ode}}$ away from 0.

\textit{Restart} (jump process): if $p_{\textrm{i,jump}}^{(k)}=0$, i.e., the number of particles at state $i$ is zero after random sampling at the $k$-th iteration, then we add \textit{one} particle to state $i$. Note that this also increases the total number of particles by \textit{one}. The updated value of $p_{\textrm{i,jump}}^{(k)}$ is then set to $(M_{\mathrm{old}} + 1)^{-1}$, where $M_{\mathrm{old}}$ is the total amount of particles before adding the new particle. Then we set $\psi_{\textrm{j,jump}}^{(k)}=-\frac{p_{\textrm{j,jump}}^{(k)}}{\pi_j}$ in \texttt{Chi-squared} and \texttt{con-Fisher} methods, or $\psi_{\textrm{j,jump}}^{(k)}=-\log\frac{p_{\textrm{j,jump}}^{(k)}}{\pi_j}$ in \texttt{KL} and \texttt{log-Fisher} methods. In our experiments, the \texttt{log-Fisher} method barely requires restart.

Multiple iterations of Metropolis–Hastings restarts or the addition of multiple particles can be performed in experiments to ensure the strict positivity of state variables, albeit at the cost of deviating from the proposed damped Hamiltonian flow.

\section{Numerical Examples}\label{sec:num}

{This section evaluates the accelerated MCMC methods against Metropolis-Hastings baselines on several finite-state graphs. The experiments are presented in three stages. First, small-graph experiments validate the
accelerated dynamics under a fixed iteration budget. Second, sampling a multimodal
Gaussian-mixture target illustrates the warm-start, adaptive time-step, and restart mechanisms. Third, larger graph benchmarks compare the end-to-end implementations under a fixed wall-clock budget. Our code is available at \url{https://github.com/silentmovie/AMCMC}. Unless stated otherwise, all experiments were conducted in MATLAB on a MacBook Pro with an Apple M1 Max processor and 64~GB of memory. Additional experiments in Appendix are performed on the UCSB computing clusters (see \Cref{app:iter}) or Google Clound GPU (see \Cref{app:gpu})}

\begin{figure}[hp!tb]
    \centering
    \begin{subfigure}[t]{0.24\textwidth}
         \centering
         \includegraphics[width=\textwidth]{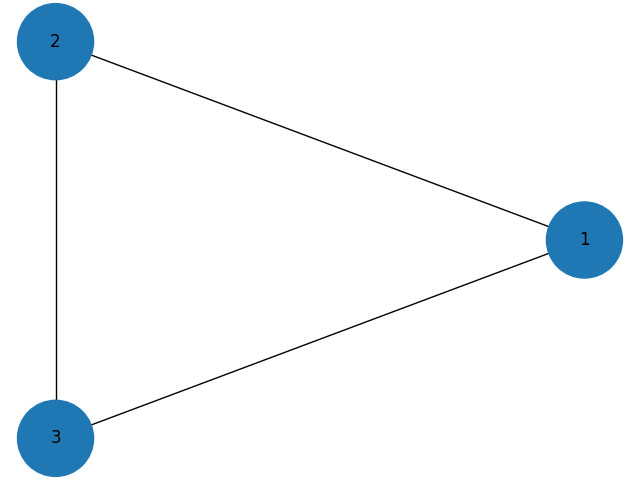}
         \caption{Circular graph.\label{fig:C3}}
     \end{subfigure}
     \hfill
     \begin{subfigure}[t]{0.24\textwidth}
     \centering
     \includegraphics[width=\textwidth]{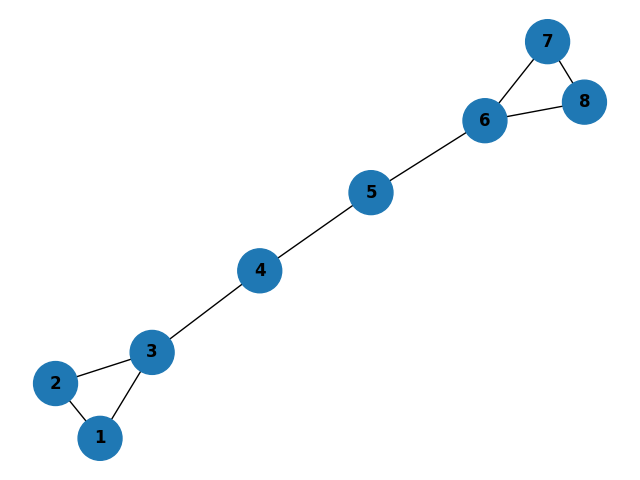}
     \caption{Two-loop graph.}\label{fig:2loop}
     \end{subfigure}
     \hfill
     \begin{subfigure}[t]{0.24\textwidth}
     \centering
     \includegraphics[width=\textwidth]{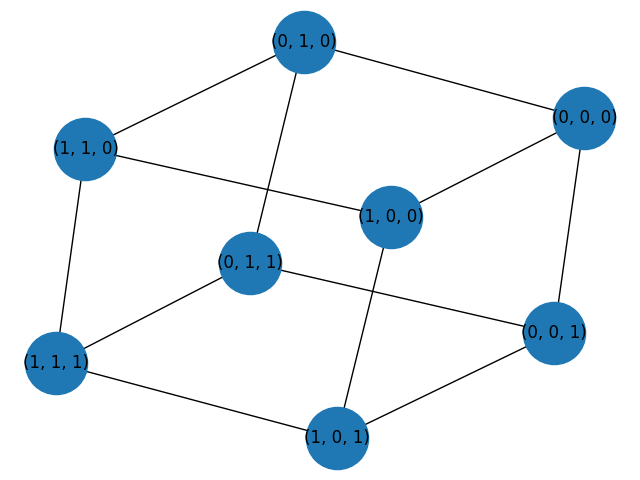}
     \caption{Hypercube.}\label{fig:hypercube}
     \end{subfigure}
     \hfill
     \begin{subfigure}[t]{0.24\textwidth}
     \centering
     \includegraphics[width=\textwidth]{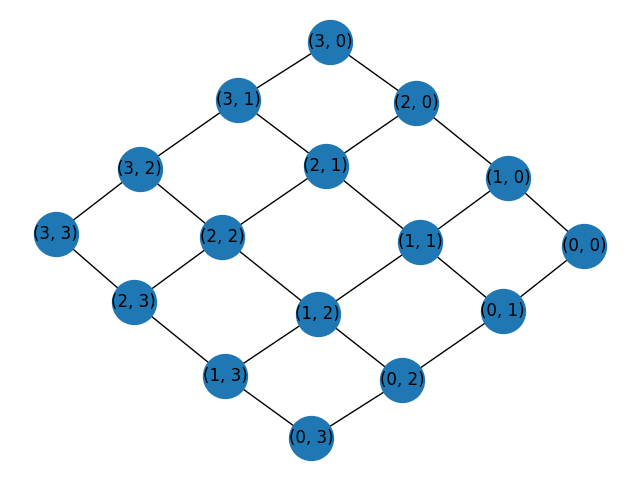}
     \caption{Lattice.}\label{fig:lattice}
     \end{subfigure} 
    \caption{Representative graph configurations employed in the numerical experiments.\label{fig:GraphConfig}}
\end{figure}

All experiments are conducted on a variety of connected and irreducible graphs, as illustrated in \Cref{fig:GraphConfig}. Given a graph $G=(V,E)$, we employ the simple random walk as the candidate kernel $
    q_{ij}=\left\{
    \begin{aligned}
    &\frac{1}{\textrm{deg}(i)}\qquad &e_{ij}\in E,\\
    &0\qquad &\textrm{otherwise}.
    \end{aligned}\right.
    $
The corresponding Metropolis–Hastings transition rate matrix is constructed as $Q^{\textrm{MH}}_{ij}=\min \left\{\frac{\pi_j}{\pi_i} q_{ji},q_{ij}\right\}$ and the associated weight matrix $\omega$ is given by $\omega_{ij}=\pi_i Q^{\textrm{MH}}_{ij}$, resulting in an undirected graph $G=(V,E,\omega)$. 

{First, we compare multi-chain MH with aMCMC on graph of less-than 10 nodes. Since the theoretical acceleration from gradient flow to accelerated flow is on the dynamic evolutional time or number of iterations, we select a relative small step size, implement both methods into multi-chain process described in \Cref{sec:num}, and use $\ell^2$ density norm as the comparison metric under a fixed iteration budget.

Given the superior property of the \texttt{log-Fisher} method (see \Cref{tab:AMH}), we mainly test this method and use the robust damping parameter suggested by the Rayleigh quotient. We test sampling a mixture of two Gaussians on a $25\times 25$ lattice and report technical details associated with \Cref{sub:restart}, along with the trajectory of sampling process. 

Finally, we compare single-chain MH with multi-chain aMCMC on graph of hundreds to thousands of nodes, sampling from image targets and the Ising model. To obtain a practical comparison under a \textit{fixed wall-clock} budget, we design two protocols to use comparable numbers of samples for the comparison, from the consideration of the central limit theorem.}

\subsection{Validation on Small Graphs}

In this subsection, we compare our proposed methods (\texttt{Chi-squared} and \texttt{log-Fisher}) with the classical MH in two settings: numerical integration of ODEs and multiple-chain jump-process on small graphs, under a \textit{fixed iteration} budget. To accurately approximate the proposed damped Hamiltonian dynamics, small time step sizes and large particle numbers are used. The warm-start, adaptive step sizes and restart mechanisms described in \Cref{sub:restart} for aMCMC are disabled in these tests. The goal of these experiments is to validate the convergence-rate and acceleration of aMCMC dynamics results under our designed damping parameter in \Cref{thm:convergence_rate} and \Cref{thm:chi} in terms of iterations.

\subsubsection{\texttt{Chi-squared} Method}

We start with a toy example on the circular graph $C_3$  (see \Cref{fig:C3}). The target distribution is $\pi=[0.9913,0.0044,0.0043]$ and the largest negative eigenvalue of $Q^{\textrm{MH}}$ is $\alpha_*\approx -0.5044$. We choose the damping parameter $ \gamma(t)=2\sqrt{\abs{\alpha_*}}=1.4220$. The corresponding largest negative eigenvalue of matrix $L$ in \eqref{eq:chi-matrix} is $\mu_*=-0.7102<\alpha_*$, thereby suggesting an accelerated convergence rate by \Cref{thm:chi}.

\begin{figure}[hp!tb]
\begin{minipage}[b]{0.5\textwidth}
  \begin{subfigure}[b]{\linewidth}
    \includegraphics[width=\linewidth, keepaspectratio]{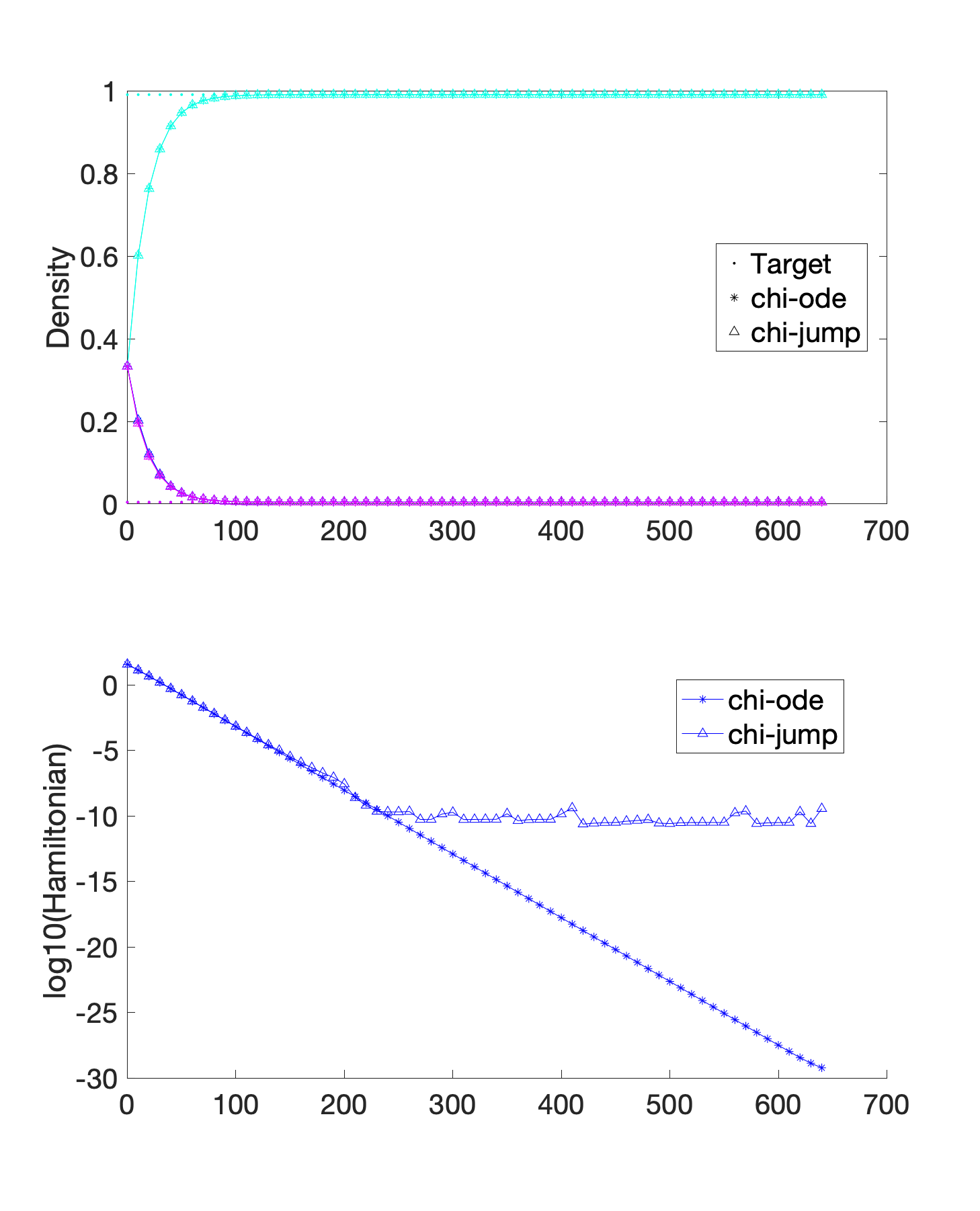}
    \caption{Validation of \texttt{Chi-squared}.\label{fig:C3-valid}}
  \end{subfigure}
\end{minipage}%
\hfill
\begin{minipage}[b]{0.5\textwidth}
  \begin{subfigure}[b]{\linewidth}
    \includegraphics[width=\linewidth, keepaspectratio]{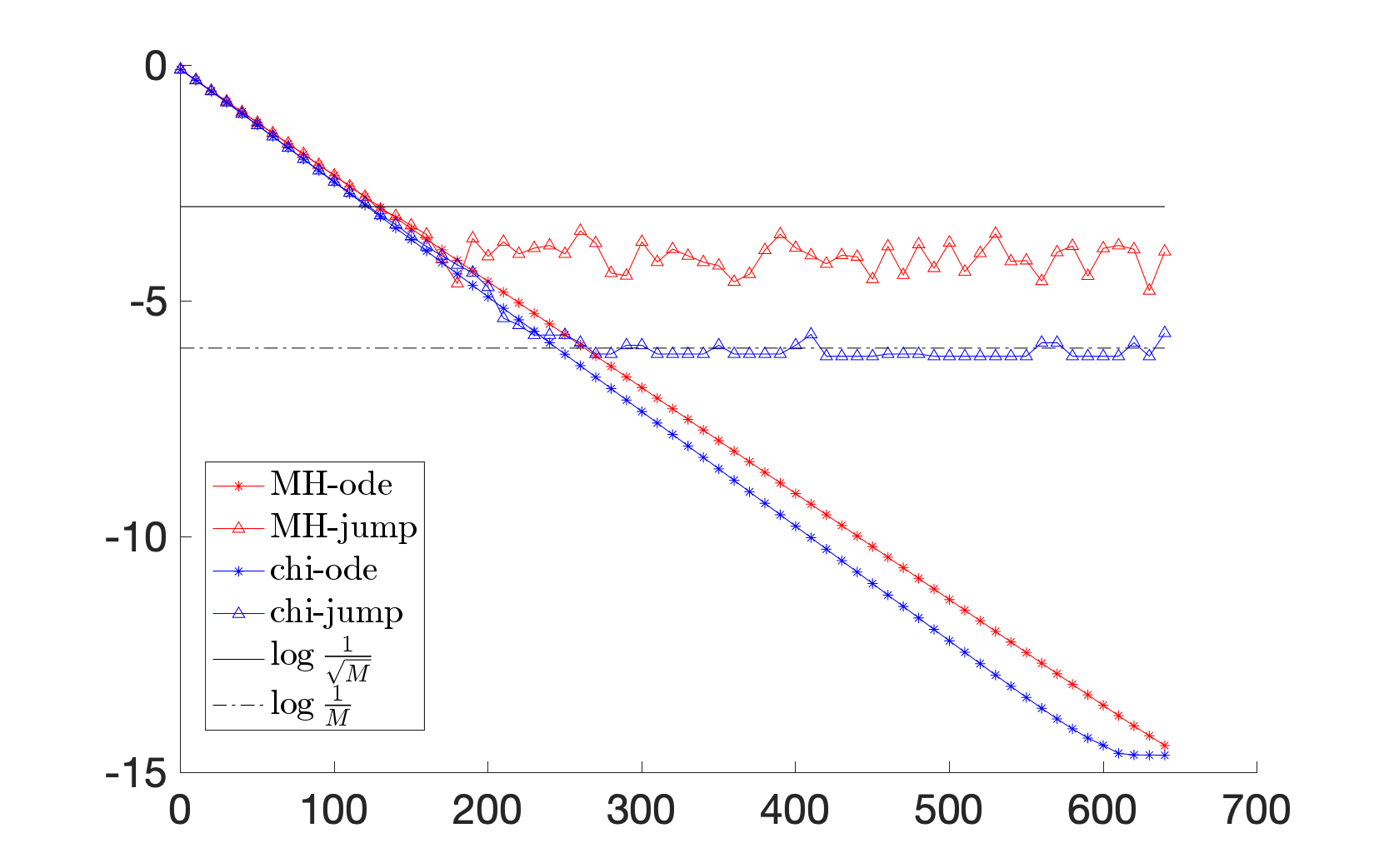}
    \caption{$\log_{10}\norm{p(t)-\pi}_2$ with $\Delta t=0.1$.\label{fig:C3-1}}
  \end{subfigure}\par
  \vfill
  \begin{subfigure}[b]{\linewidth}
    \includegraphics[width=\linewidth, keepaspectratio]{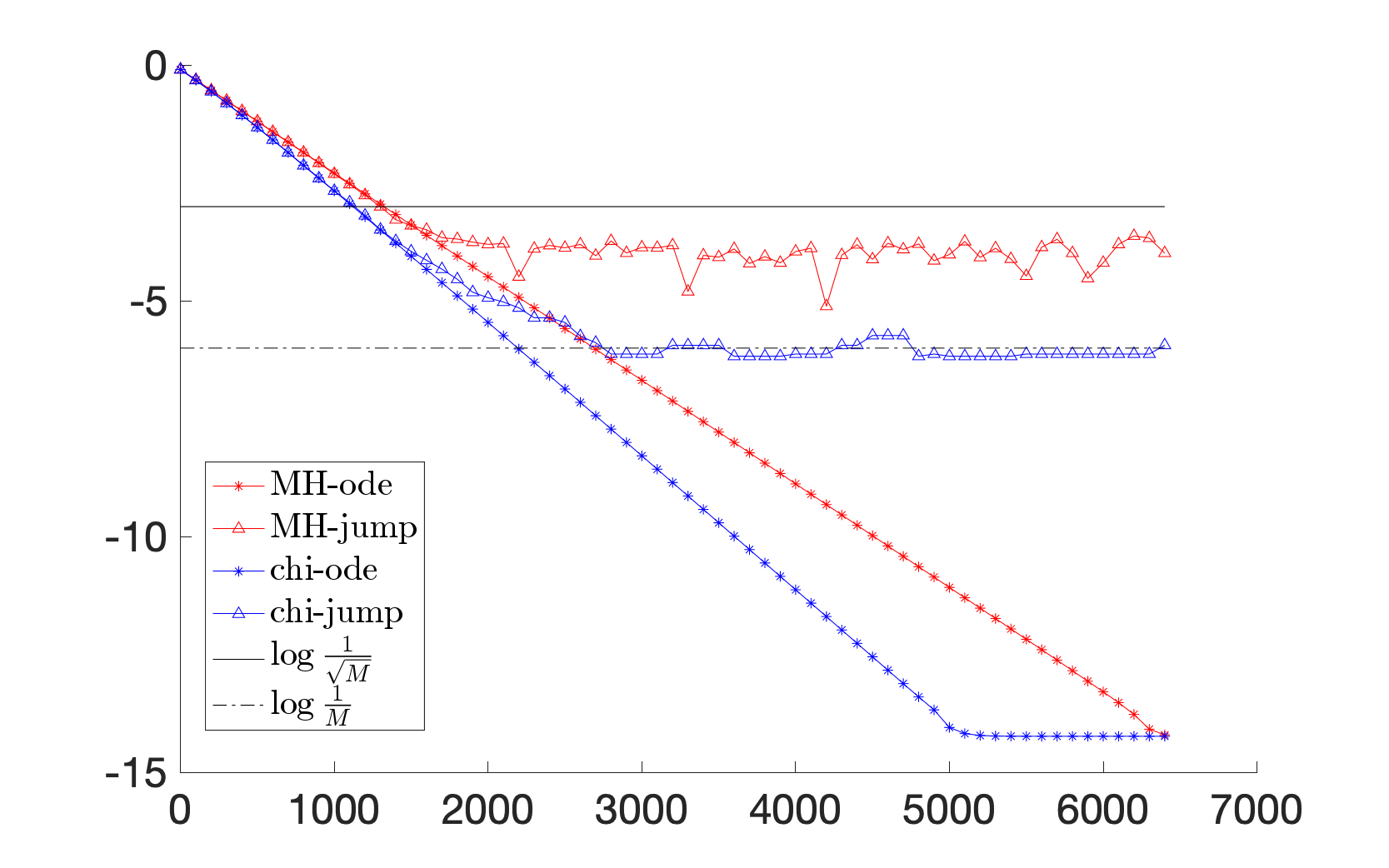}
    \caption{$\log_{10}\norm{p(t)-\pi}_2$ with $\Delta t=0.01$.\label{fig:C3-2}}
  \end{subfigure}
\end{minipage}
\caption{Accelerated sampling on $C_3$ graph via \texttt{Chi-squared} method. x-axes are in iterations and all results are over time span $[0,65]$ when both ODE solvers achieve machine precision. Panel (a) uses  step size $\Delta t=0.1$ and displays the convergence of empirical distribution and the decay of $\log_{10}(\mathcal{H}(p(t),\psi(t)))$. Panel (b) and (c) compare the decay of $\log_{10}\norm{p(t)-\pi}_2$ for step sizes $\Delta t=0.1$ and $\Delta t=0.01$ respectively. The jump process by \texttt{Chi-squared} achieves an error of order $\mathcal{O}(1/M)$. \label{fig:C3-performance}}
\end{figure}

The parameters for MH and \texttt{Chi-squared} methods are chosen identically:
    \[\textrm{particle size~}M=10^6,\quad\textrm{step size~}\Delta t= 0.1\textrm{~or~}0.01,\quad\textrm{total iterations~}N=650\textrm{~or~}6500,\]
yielding the same effective time span $[0,65]$ across every run shown in \Cref{fig:C3-performance}. We initialize $p(0)$ as the uniform distribution and $\psi_j(0)=-\frac{p_j(0)}{\pi_j}$.

\Cref{fig:C3-valid} validates the convergence of empirical distribution by \texttt{Chi-squared} and the decay of Hamiltonian \eqref{eq:Decay-Hamiltonian}.
\Cref{fig:C3-1} and \Cref{fig:C3-2} corroborate \Cref{thm:chi} and demonstrate that the \texttt{Chi-squared} dynamics achieve a faster convergence in $\ell_2$ error than the MH dynamics, in both the deterministic and particle implementations. Notably, jump process sampling achieves a higher order of accuracy $\mathcal{O}(\frac{1}{M})$. The smaller time step follows the continuous damped Hamiltonian dynamics more closely, further accelerating convergence.

\subsubsection{\texttt{log-Fisher} Method}

Certain geometric features of a Markov Chain influence its mixing time, with bottlenecks being a notable example \citep[See][]{Levin2009Markov}. We therefore a two-loop graph bridged by a thin bottleneck (see \Cref{fig:2loop}), whose target distribution is given by
$\pi=[\frac{4}{27},\frac{4}{27},\frac{4}{27},\frac{1}{18},\frac{1}{18},\frac{4}{27},\frac{4}{27},\frac{4}{27}]$. The largest negative eigenvalue to the associated $Q^{\textrm{MH}}$ is approximately $-3.79\times 10^{-2}$, reflecting the slow mode induced by the bottleneck. The parameters for MH and \texttt{log-Fisher} methods are chosen identically:
    \[\textrm{particle size~}M=10^4,\qquad\textrm{step size~}\Delta t= 0.1\,\qquad\textrm{total iterations~}N=1000.\]

\begin{figure}[hp!tb]
     \centering
     \begin{subfigure}[b]{0.45\textwidth}
         \centering
         \includegraphics[width=\textwidth]{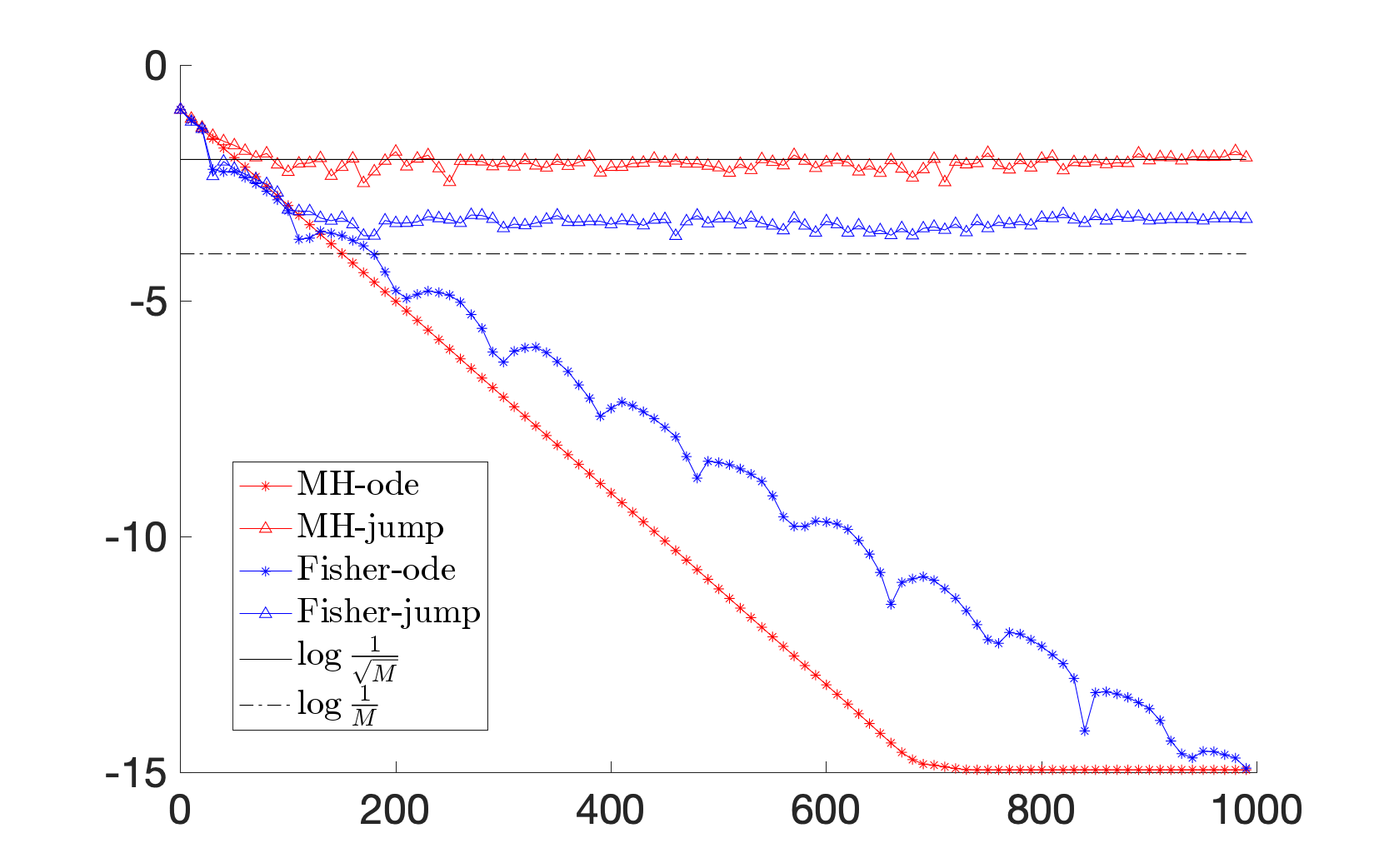}
     \end{subfigure}
     \hfill
     \begin{subfigure}[b]{0.45\textwidth}
         \centering
         \includegraphics[width=\textwidth]{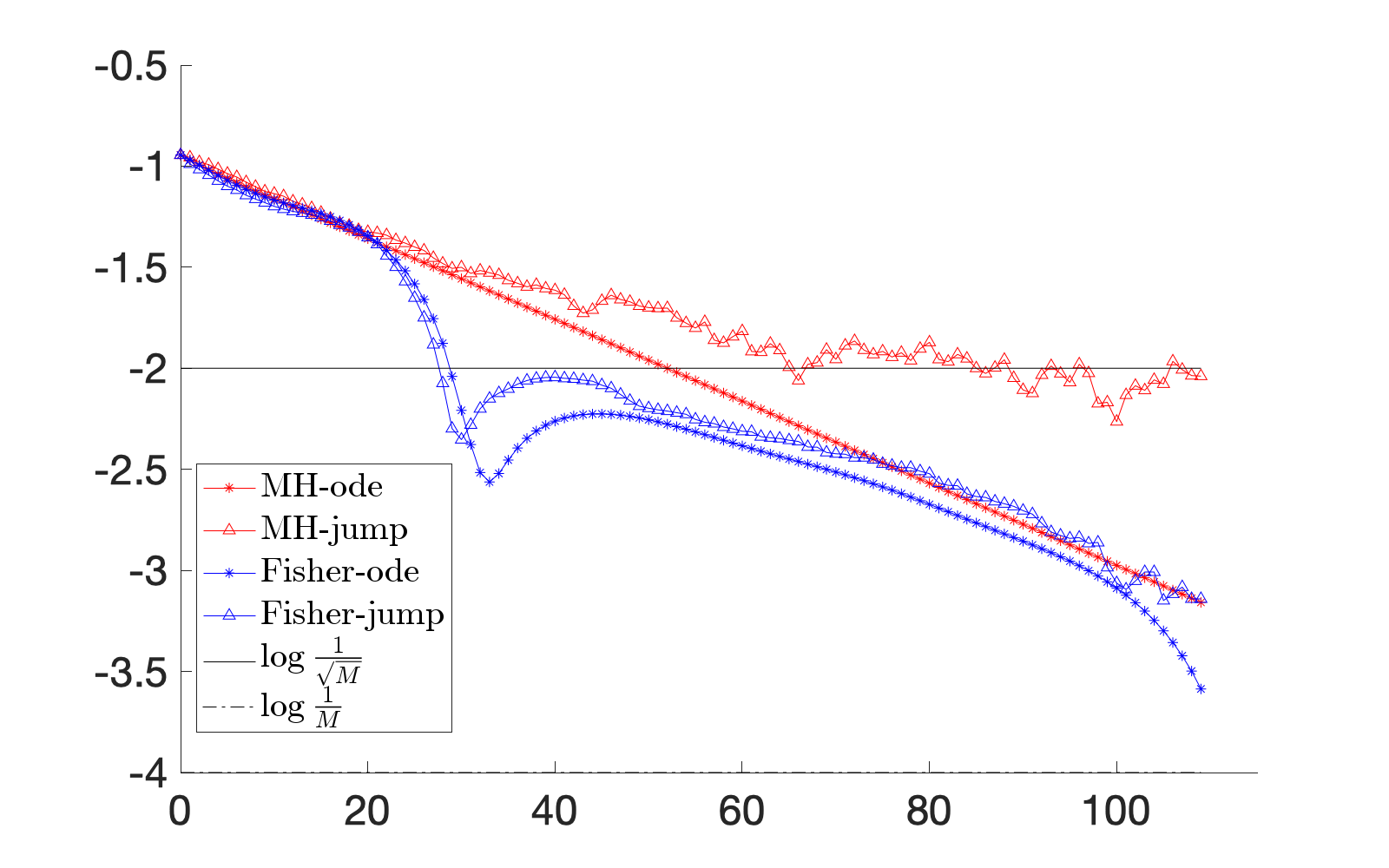}
     \end{subfigure}
     \caption{Accelerated sampling on a two-loop graph of node 8  (see \Cref{fig:2loop}) via \texttt{log-Fisher} method. x-axes are in iterations with a fixed step size $\Delta t=0.1$. y-axes shows $\log_{10}\norm{p(t)-\pi}_2$. The right panel zooms in the first 120 iterations of the left panel, when the MH-jump reaches its accuracy limit $\mathcal{O}(\frac{1}{\sqrt{M}})$. \label{fig:twoloop-performance}}
\end{figure}

For the \texttt{log-Fisher} method, we initialize $p(0)$ as the uniform distribution and {$\psi_j(0)=-\log\frac{p_j(0)}{\pi_j}$}. The damping parameter is
    \[
\gamma(t)=
\left\{\begin{aligned}
&0.5\qquad &t<3;\\
&\max\Big\{\frac{3}{t-2}, 0.6\Big\}\qquad &t\geqslant 3,
\end{aligned}\right.
    \]
which combines a Nesterov-type time-dependent term with the lower bound suggested by the \texttt{con-Fisher} approximation in \Cref{sub:damping}.

With the above choice of parameters, neither restarts nor adaptive step sizes are triggered in this experiment. Thus we compare the dynamics driven by the relative Fisher information with those by the MH update. As shown in \Cref{fig:twoloop-performance}, the jump process by the \texttt{log-Fisher} method exhibits faster convergence for fixed number of iterations, and achieves higher accuracy than its MH counterpart, particularly once MH sampling reaches its accuracy limit $\mathcal{O}(\frac{1}{\sqrt{M}})$. Though the \texttt{log-Fisher} ODE reaches machine precision later than the MH forward equation for this particular damping choice, its convergence rate can potentially be accelerated by appropriately tuning the damping parameter. 

We also note that, unlike the time-homogeneous transition-rate matrix used by MH, the transition-rate matrix \(Q\) in the \texttt{log-Fisher} method is time-inhomogeneous and must be updated along the dynamics. This leads to a larger computational cost per iteration.

\subsection{Warm Start, Adaptive Step Sizes and Restart}

We next examine the particle implementation on a $25\times 25$ lattice (see \Cref{fig:lattice}) for sampling a Gaussian mixture, which poses a challenging multimodal sampling problem. This subsection is to showcase the warm start, adaptive step sizes and restart mechanisms by the evolutional trajectory of the empirical distribution in the sampling procedure.

The target distribution is defined as a mixture of two Gaussian distributions on $[0,1]^2$, whose centers are $x_1=(0.25,0.25)$ and $x_2=(0.75,0.75)$, discretized a $25\times 25$ lattice:
    \begin{equation}\label{eq:Gaussian-mixture}
    \pi(x)=\frac{1}{Z}\left[\exp\left(-10\norm{x-x_1}_2^2\right)+\exp\left(-40\norm{x-x_2}_2^2\right)\right],
    \end{equation}
where normalizing constant $Z$ is unknown. The common parameters are specified as follows: 
    \[\textrm{particle size~}M=5\times 10^5,\quad\textrm{default step size~}\Delta t= 0.01\,\quad\textrm{total iterations~}N=1.5\times 10^5.\]
The \texttt{log-Fisher} method is warm-started with 2999 MH iterations (i.e., $t< 30$, accounting for 2.00\% of the total iterations). After this initialization phase, we employ a constant damping parameter $\gamma(t)=2\sqrt{\lambda_*}=0.0065$, where $\lambda_*$ is obtained from the Rayleigh quotient associated with the asymptotic \texttt{con-Fisher} approximation.

\begin{figure}[hp!tb]
     \centering
     \begin{subfigure}[b]{0.45\textwidth}
         \centering
         \includegraphics[width=\textwidth]{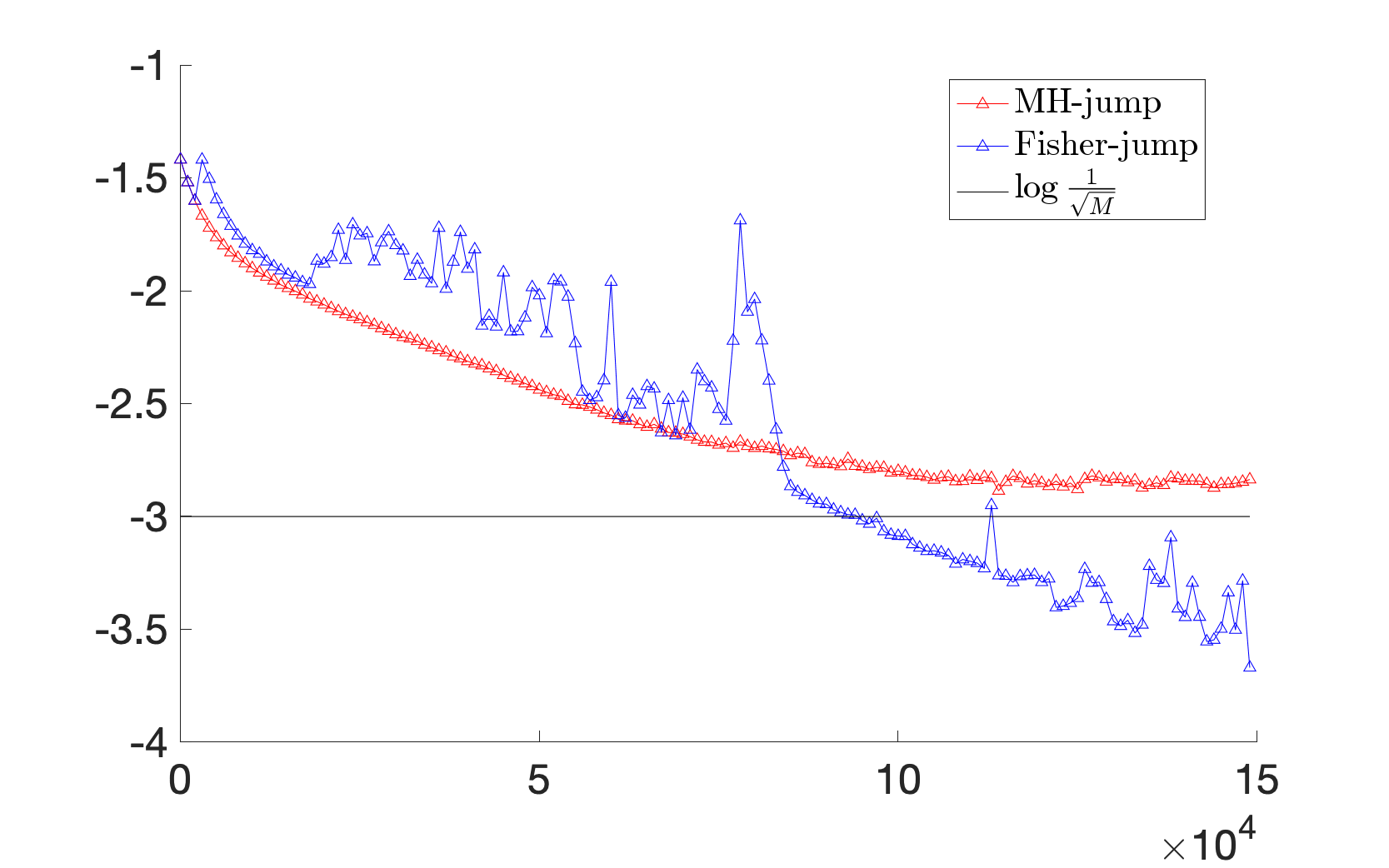}
     \end{subfigure}
     \hfill
     \begin{subfigure}[b]{0.45\textwidth}
         \centering
         \includegraphics[width=\textwidth]{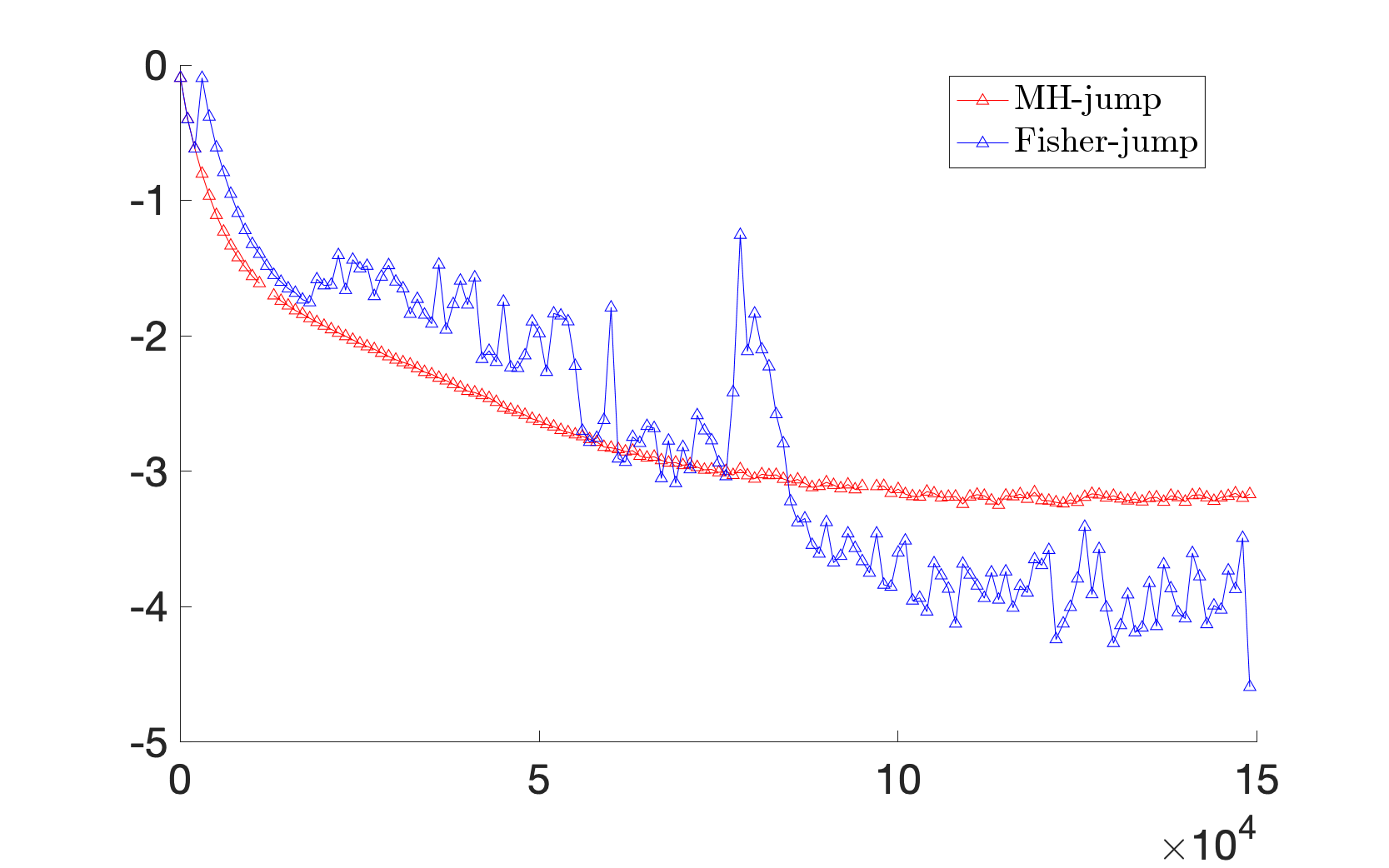}
     \end{subfigure}
     \caption{Sampling the Gaussian mixture \eqref{eq:Gaussian-mixture} on a $25\times 25$ lattice graph. x-axes are in iterations. The left panel shows the approximation error $\log_{10}\norm{p(t)-\pi}_2$ w.r.t the target distribution $\pi$. The right panel shows the approximation error $\log_{10}\abs{\sum_{i=1}^n p_i(t)\log \frac{p_i(t)}{Z\pi_i } - (-\log Z)}$ w.r.t the normalizing constant $Z$. \label{fig:lattice-performance}}
\end{figure}

\begin{figure}[hp!tb]
    \centering
    \includegraphics[width=0.85\linewidth]{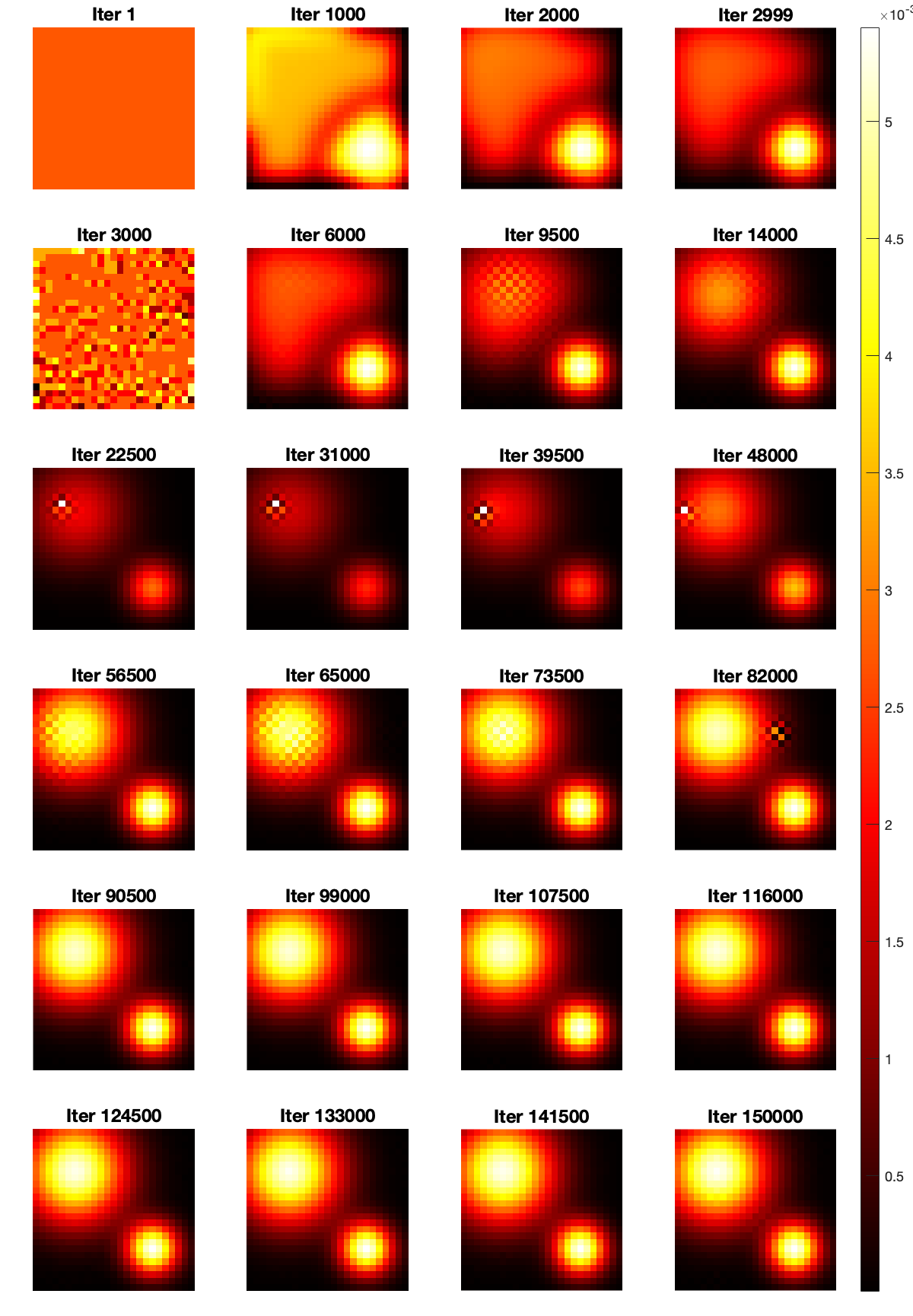}
    \caption{The \texttt{log-Fisher} sampling a Gaussian-mixture target on a $25 \times 25$ lattice. The transition between iterations 2999th and 3000th marks the end of the MH warm start and the beginning of the \texttt{log-Fisher} jump process. \label{fig:lattice-gaussian}}   
\end{figure}

Unlike the previous experiments, both adaptive step sizes and restart mechanism are activated for the \texttt{log-Fisher} method, to prevent $p^{(k)}_{\text{i,jump}}$ from being zero. The adaptive step-size scheme results in an effective evolutional time span of approximately $[0, 1449.9]$, corresponding to 96.66\% of the MH method's duration. The restart mechanism is activated whenever $p^{(k)}_{\text{i,jump}}=0$, in which case the damping effect is disabled in the next step and we add one particle to those empty states. During the entire jump process, the restart mechanism is activated $1243$ times (approximately $0.85\%$ of the total iterations). The frequency of restarts per $1000$ iterations ranges from $0$ to $14$.

As shown in \Cref{fig:lattice-performance}, the \texttt{log-Fisher} jump process achieves smaller approximation errors (smaller than $\mathcal{O}(\frac{1}{\sqrt{M}})$) to the target distribution and to the normalizing constant. \Cref{fig:lattice-gaussian} visualizes the evolution of the empirical distribution of the state variable. The abrupt change between iterations 2999 and 3000 coincides with the switch from the MH warm start to the accelerated dynamics. This transient indicates that a smoother initialization of the momentum, a more gradual damping schedule, or a smaller initial accelerated step may improve robustness. Nevertheless, the adaptive and restart mechanisms are triggered in a small fraction of iterations in this multimodal sampling.

{\subsection{Wall-clock benchmarks on large lattices}\label{sub:image}

We now compare the \texttt{log-Fisher }multiple-chain jump process with the single-chain MH (see \Cref{alg:MH}) under a  \textit{fixed wall-clock} budget. The purpose of these experiments is to evaluate the end-to-end computational performance on larger state spaces.

In the absence of prior graphical information, we simply read input images (`rose', `fractal tree' and `checkerboard') from real data used in \citep{Xu2023Normalizing} and discretize grayscale images onto a lattice of hundreds-to-thousands nodes. To ensure strict positivity of the target distribution, we add one tenth of the largest pixel value to each node, while the normalizing factor $Z$ is unknown in the numerical experiments.

\begin{figure}[hp!tb]
     \centering
     \begin{subfigure}{0.3\textwidth}
         \centering
         \includegraphics[width=\textwidth]{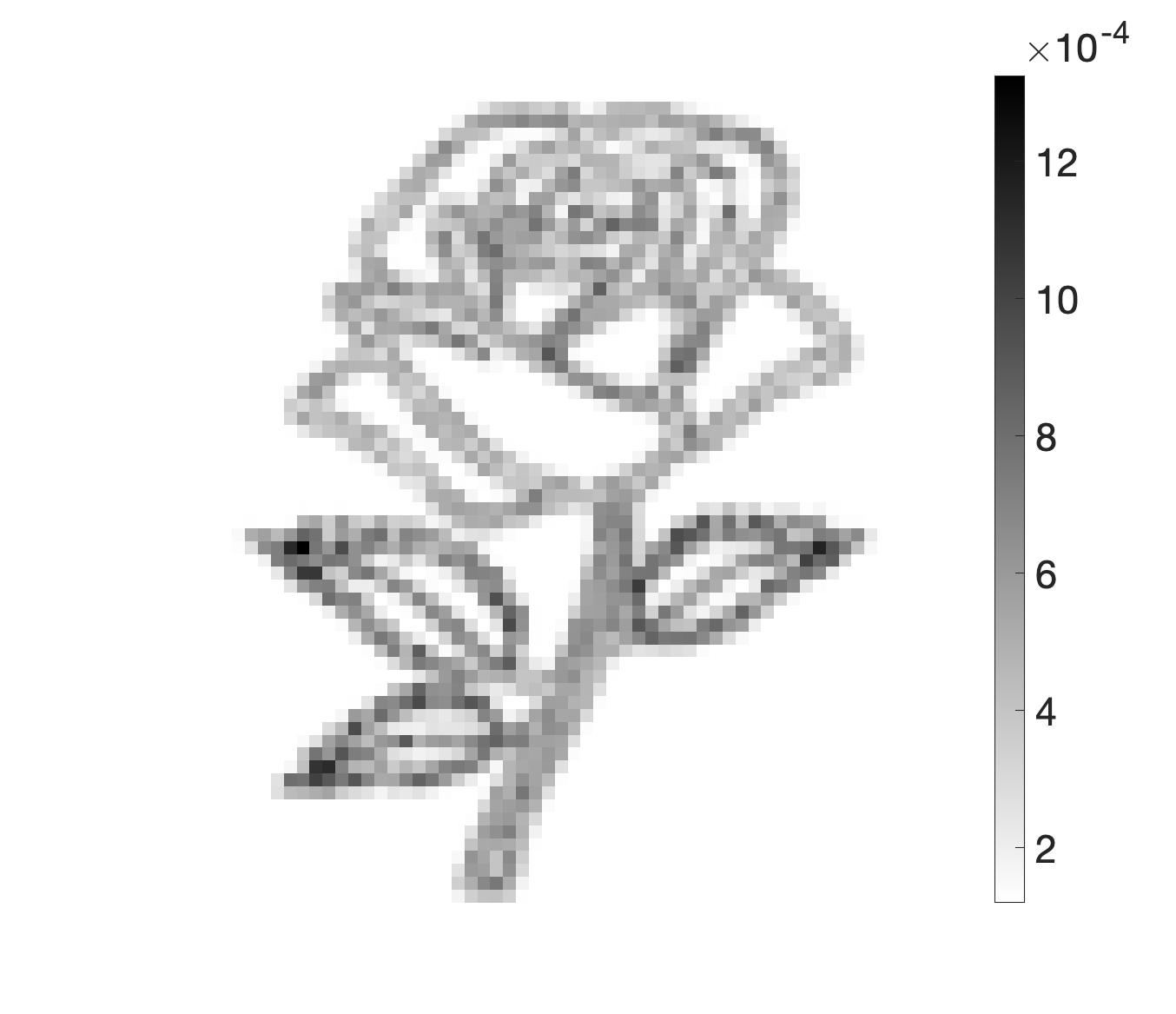}
     \end{subfigure}
    \begin{subfigure}{0.3\textwidth}
         \centering
         \includegraphics[width=\textwidth]{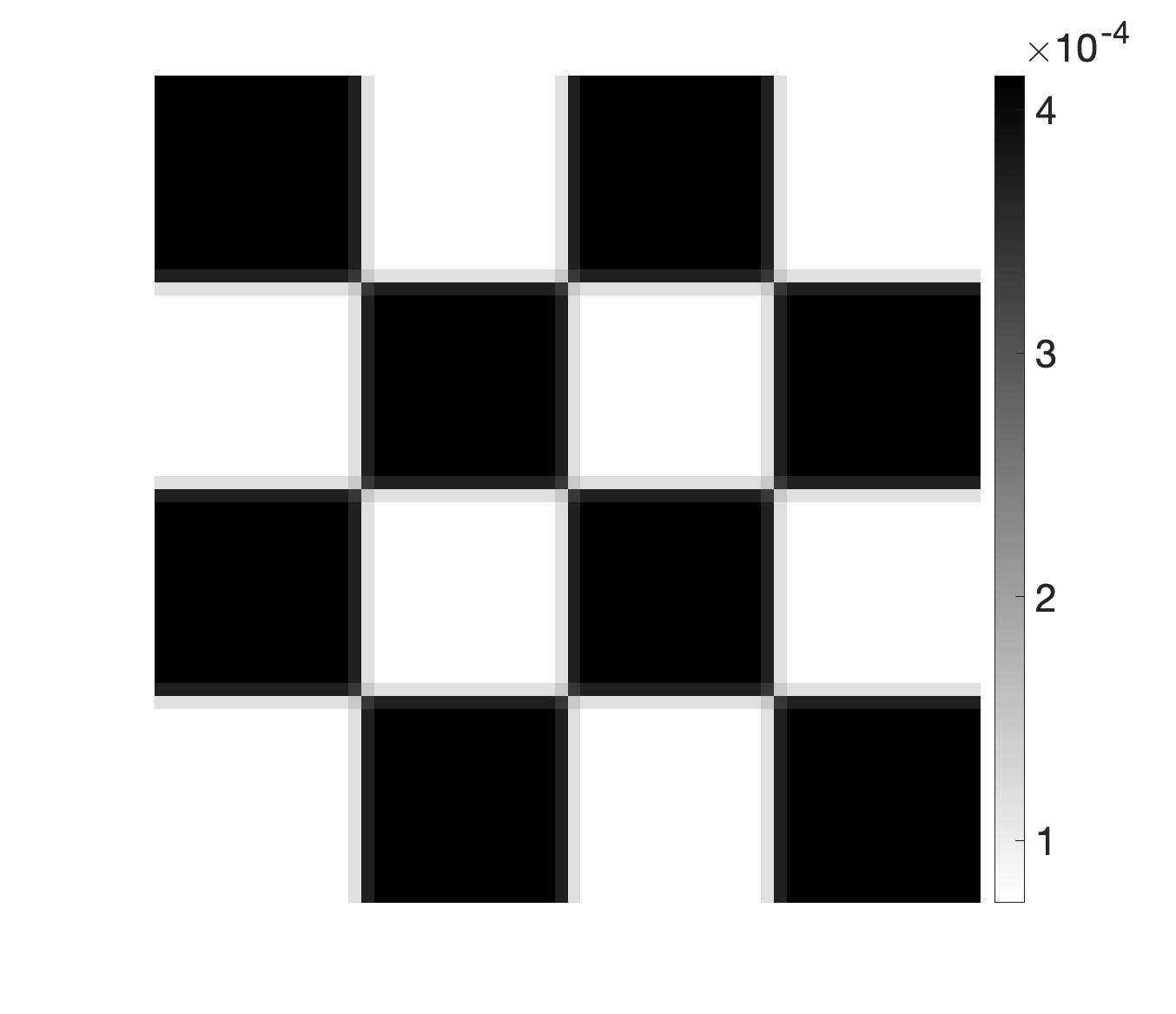}
     \end{subfigure}
    \begin{subfigure}{0.3\textwidth}
         \centering
         \includegraphics[width=\textwidth]{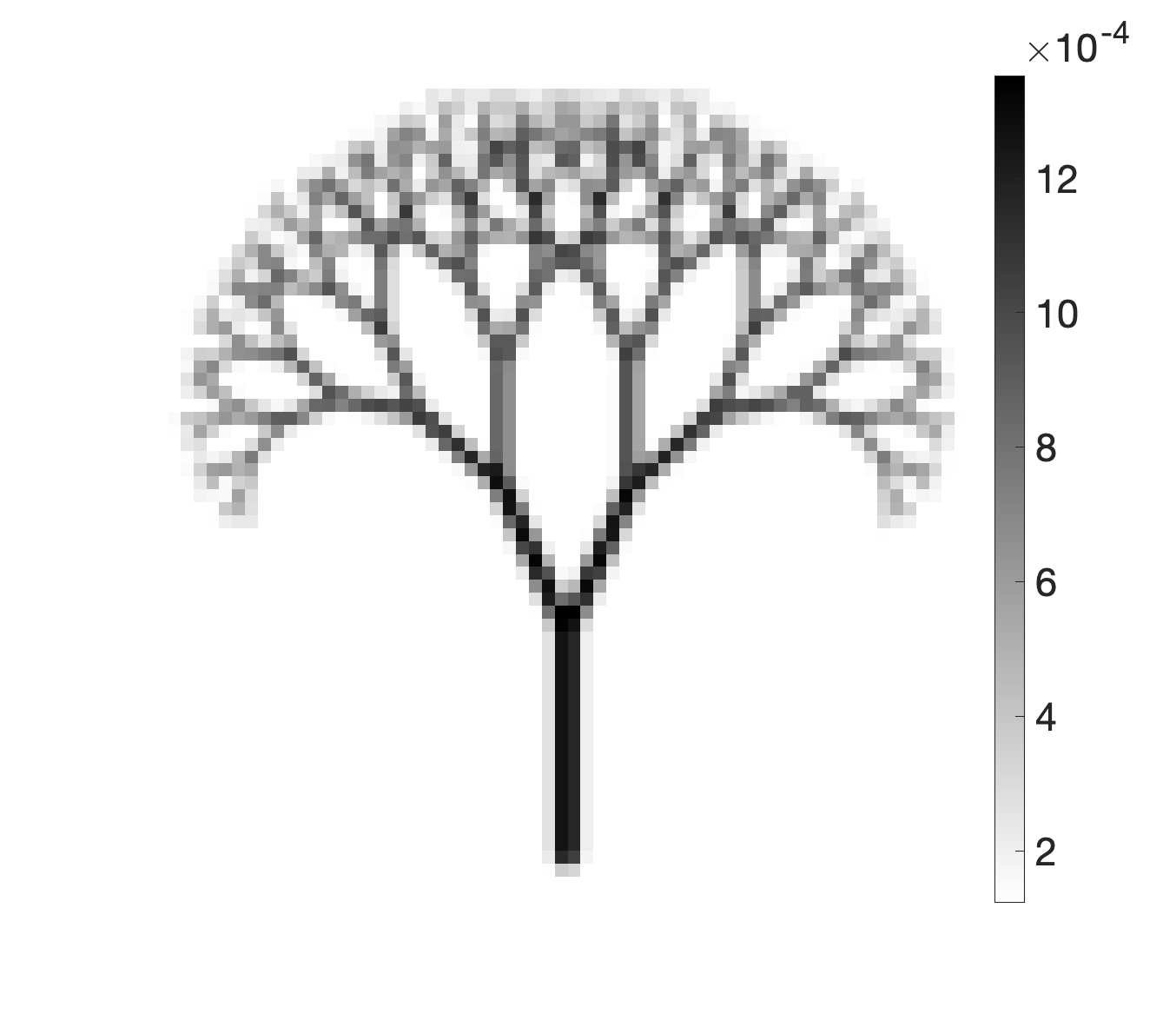}
     \end{subfigure}
\caption{Image-derived target distributions on a $64\times 64$ lattice: `rose', `checkerboard', and `fractal tree'. The normalizing constant $Z$ is unknown to algorithms.}
\end{figure}

The two methods produce samples in substantially different forms. A single
MH chain generates one sample per iteration, with successive samples generally
correlated, whereas the \texttt{log-Fisher} method evolves $M$ interacting
particles and therefore produces $M$ particle states per iteration. To account
for this difference, we consider the following two comparison protocols,
designed to be conservative with respect to the \texttt{log-Fisher} method,
under a fixed computational budget and with comparable numbers of samples.

\begin{enumerate}[label=\arabic*)]
    \item \textbf{Aggregated samples.}
    Suppose that the single-chain MH algorithm generates \(N_{\mathrm{MH}}(t)\)
    samples within a fixed computational time \(t\). For the \(M\)-particle
    \texttt{log-Fisher} jump process, we aggregate particles from the final $W(t) = \left\lfloor \frac{N_{\mathrm{MH}}(t)}{M} \right\rfloor$
    iterations within the same computational time, and then use \(M W(t)\approx N_{\mathrm{MH}}(t)\) particle states for estimating empirical distributions. This protocol therefore compares the
    accuracy of two methods with approximately the same number of recorded raw samples.

    \item \textbf{Effective samples.}
    Within a fixed computational time $t$, we collect the most recent $M$ effective samples from the single-chain MH sampler for estimating empirical distributions, the normalizing constant and the quantity of interest $\mathbb{E}[f(X)]$. 
    The effective samples are selected by a function-specified integrated autocorrelation time
    \(\tau_f\). The
    \texttt{log-Fisher} estimator is computed by the $M$ interacting particles at time $t$. This protocol is conservative with respect to the \texttt{log-Fisher} method in the sense that the effective sample size represented by its $M$ interacting particles cannot exceed $M$, whereas the MH samples are selected to represent approximately $M$ function-specific effective samples. 
\end{enumerate}

Both protocols are applied only at the data-process stage. The raw samples
generated by the two methods are obtained from the same experimental runs and
under the same computational settings described below.

For the single-chain MH sampler, the step size is set to be 1 and each iteration produces one sample. We do not apply burn-in strategy because the experiment is designed to compare the end-to-end evolution from the common initialization.

For the \texttt{log-Fisher} method, we implement \Cref{alg:aMCMC} and set the default step size to be 0.1. The adaptive step size is activated to ensure a valid transition matrix. The base particle number is $M=2^{14}\times 10$, and additional runs use particle numbers up to $16M$. Each \texttt{log-Fisher} run begins with 9 iterations of multi-chain MH updates, mainly for initializing the momentum $\psi$. The warm-start phase is included in the computational time budget. After this, we use the constant damping parameter $\gamma(t)=2\sqrt{\lambda_*}$ across all runs, where the Rayleigh quotient $\lambda_*$ is computed for each target. 

No attempt is made to optimize the implementation or tune
parameters separately for wall-clock performance. Thus the experiments
should be interpreted as comparisons of the present implementations under
a common computational environment. In particular, sparse linear algebra,
parallel implementations, computation platforms or GPU acceleration can substantially affect
the absolute computational cost. As the multi-chain MH and \texttt{log-Fisher} methods can be implemented in a closer way and produce samples in the same form. We also provide a comparison between multi-chain MH and \texttt{log-Fisher} methods under a \textit{fixed iteration} budget in \Cref{app:iter}.

\subsubsection{Aggregated Samples}

The comparison under the aggregated-sample protocol is in terms of the approximation error $\log_{10}\norm{p(t)-\pi}_2$, where $p(t)$ is the empirical distribution estimated by a comparable number of recorded samples within a fixed computational time.

For the single-chain MH sampler, all samples generated before time $t$ are used to estimate $p(t)$. A single MH jump is
substantially cheaper than one \texttt{log-Fisher} iteration. However, the total number of generated samples per second from \texttt{log-Fisher }$M$-chain jump process is much larger than MH. Given a fixed computational time where MH generates $N_{\textrm{MH}}(t)$ samples, we aggregate samples generated from the last $\left\lfloor \frac{N_{\mathrm{MH}}(t)}{M} \right\rfloor$ iterations of \texttt{log-Fisher} method, whose total sample size is close to, but does not exceed $N_{\textrm{MH}}(t)$. 

\begin{figure}[hp!tb]
     \centering
     \begin{subfigure}{0.48\textwidth}
         \centering
         \includegraphics[width=\textwidth]{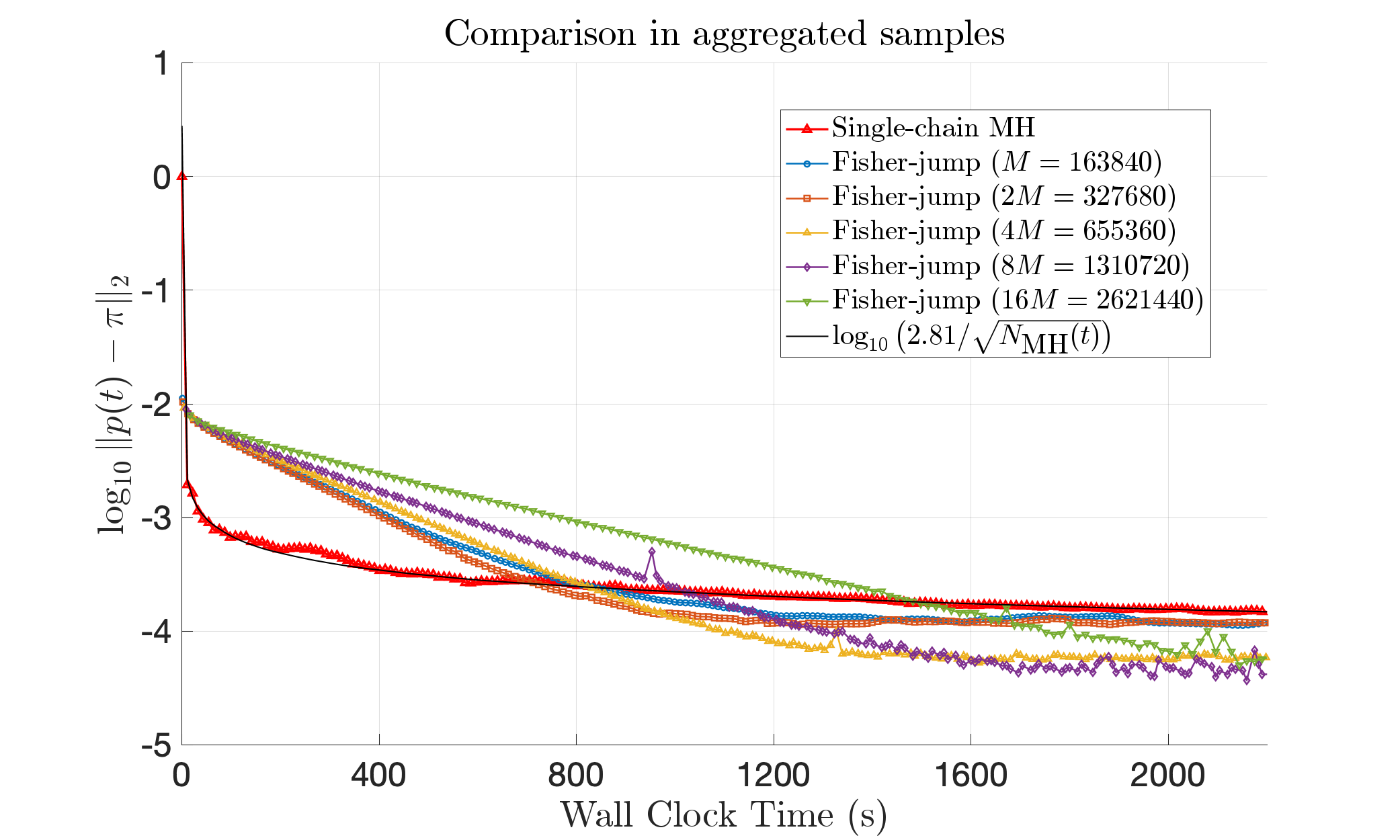}
     \end{subfigure}
    \begin{subfigure}{0.48\textwidth}
         \centering
         \includegraphics[width=\textwidth]{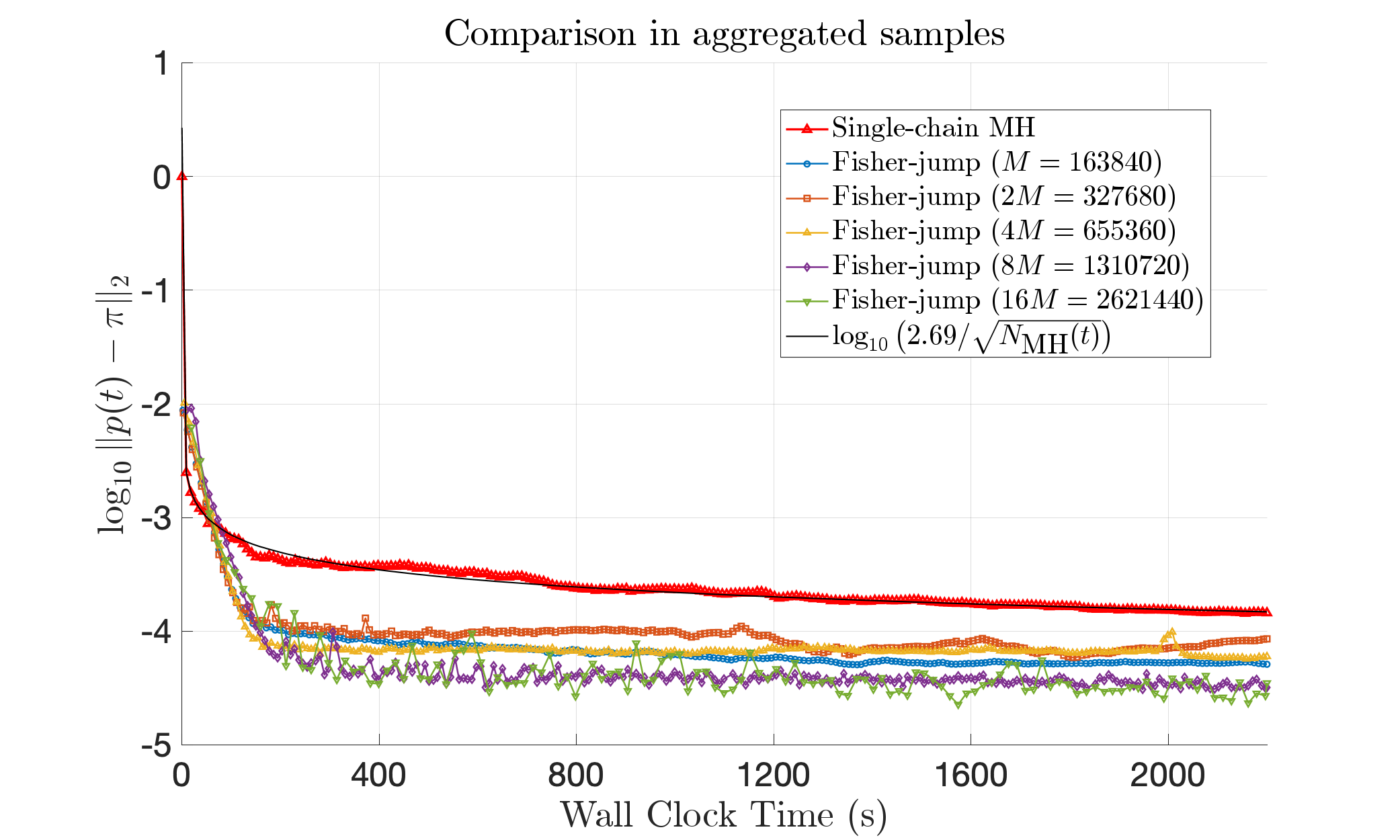}
     \end{subfigure}
\caption{Wall-clock comparisons on sampling image targets discretized on a $64\times 64$ lattice (left for `rose' and right for `checkerboard'). The empirical distributions use the aggregated-sample protocol. $x$-axes are in `wall-clock' time; $y$-axes are the approximation error $\log_{10} \norm{p(t)-\pi}_2$ w.r.t the target distribution $\pi$. \label{fig:matchedsamples}}
\end{figure}

\Cref{fig:matchedsamples} shows the corresponding $\ell^2$ error trajectories plots against the computational time. The
\texttt{log-Fisher} method reaches a smaller  error for each tested particle number across different targets in the long time, despite the damping parameter not being optimized. At later times, the trajectories are dominated by the standard error from the central limit theorem. We will show in the next comparison protocol that, a comparable empirical
accuracy by \texttt{log-Fisher} can often be achieved with substantially fewer samples.

The spikes in trajectories of the \texttt{log-Fisher} method coincide with restarts, which are triggered when the particle count at a node falls below the threshold of 10 used in these experiments. The adaptive step size and restart mechanisms are   in all
experiments; their activation rates remain small relative to the total number of iterations. We also observe that using more particles in parallel does not automatically lead to a better accuracy for the aggregated samples at a fixed wall-clock budget. 

\subsubsection{Effective samples}

Samples $X^{(1)},X^{(2)},\ldots$ generated from a long single-chain are not independent, thus autocorrelation analysis is introduced to quantify the effects of sampling error. More precisely, the sampling variance of $\mathbb{E}[f(X)]\approx \frac{1}{N} \sum_{i=1}^N f(X^{(i)})$ is given by
    \[
    \sigma^2 = \frac{\tau_f}{N} \textrm{Var}[f(X)],
    \]
where $\tau_f$ is the function-specific integrated autocorrelation time for the chain $f(X^{(n)})$. And $\frac{N}{\tau_f}$ is the effective number of samples.  We follow \citep{Goodman2010Ensemble} and \citep{EMCEE2013} to estimate the autocorrelation time $\tau_f$. 

For the \texttt{log-Fisher} method, $M$ particles are used to estimate the empirical distribution $p(t)$ at any computational time t. Given their interaction feature, $M$ particles may not be independent. For the MH method, we heuristically pick one effective sample every $\tau_f$ samples in the single-chain MH sampler, and collect $M$ effective samples for estimations. Since the total number of generated samples per second from \textrm{log-Fisher} $M$-chain jump process is much larger than the single-chain MH sampler, it takes some time to collect $M$ effective samples for the MH method. After that, for every computational time $t$, we collect the last window of $M$ effective samples for estimations.

Three metrics are used for comparison: $\log_{10} \norm{p(t)-\pi}_2$ measures the error in the empirical distribution with respect to the target distribution $\pi$; $\log_{10} \abs{\sum_{i=1}^n p_i(t)\log \frac{p_i(t)}{Z \pi_i}- (-\log Z)}$ measures the error in recovering the normalizing constant $Z$; the approximation error for the expectation $\mathbb{E}[f(X)]$.

In this subsection, we pick $f(x)=-\ln \pi(x)$ and evaluate the entropy on the graph $\mathbb{E}[f(X)]\approx -\frac{1}{N}\sum_{i=1}^N \ln \pi(X^{(i)})$, where $X^{(i)}\in V=\left\{1,2,\ldots,n\right\}$. When the empirical distribution $p=[p_1,p_2,\ldots,p_n]$ is available, it can be evaluated by $-\sum_{i=1}^n p_i \ln \pi_i$. Consequently, the third metric is given by $\log_{10}\abs{-\sum_{i=1}^n p_i \ln \pi_i + \sum_{i=1}^n \pi_i \ln \pi_i}$.

\begin{figure}[hp!tb]
     \centering
     \begin{subfigure}{0.32\textwidth}
         \centering
         \includegraphics[width=\textwidth]{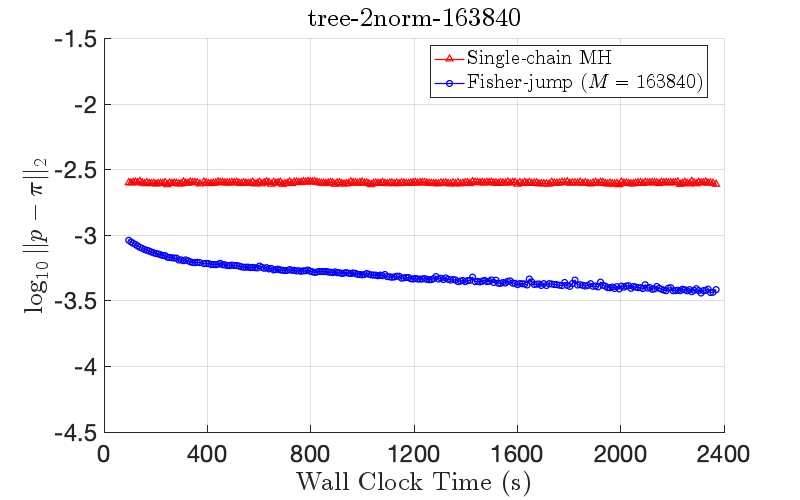}
     \end{subfigure}
    \begin{subfigure}{0.32\textwidth}
         \centering
         \includegraphics[width=\textwidth]{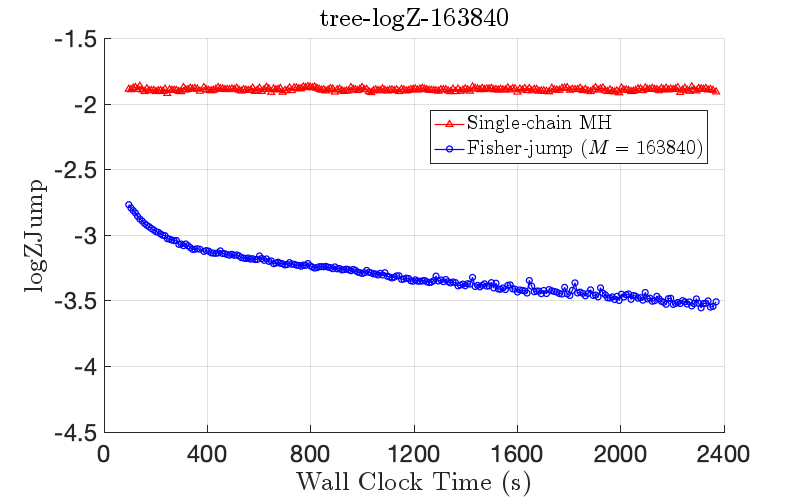}
     \end{subfigure}
     \begin{subfigure}{0.32\textwidth}
         \centering
         \includegraphics[width=\textwidth]{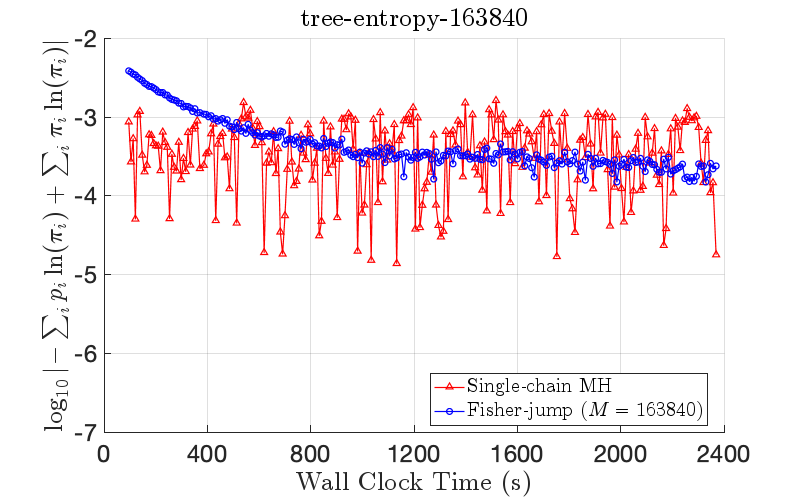}
     \end{subfigure}
    \begin{subfigure}{0.32\textwidth}
         \centering
         \includegraphics[width=\textwidth]{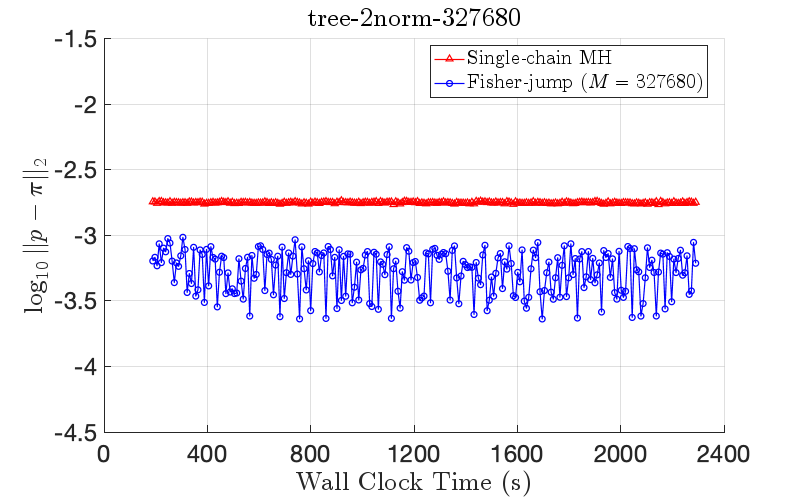}
     \end{subfigure}
    \begin{subfigure}{0.32\textwidth}
         \centering
         \includegraphics[width=\textwidth]{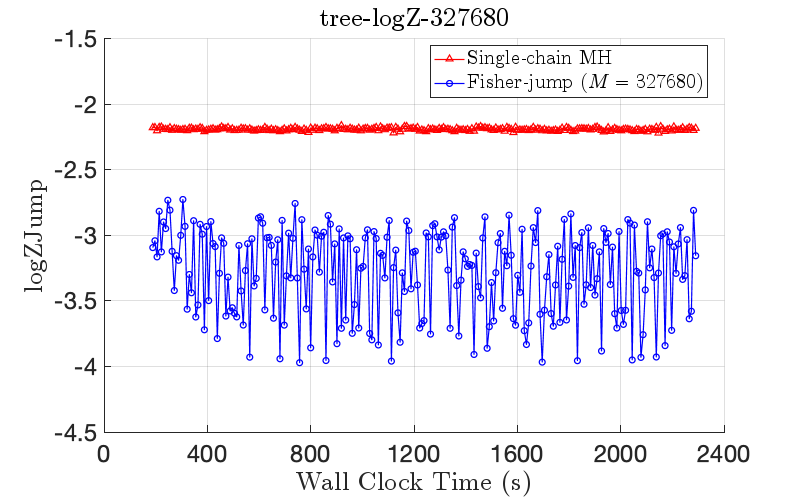}
     \end{subfigure}
     \begin{subfigure}{0.32\textwidth}
         \centering
         \includegraphics[width=\textwidth]{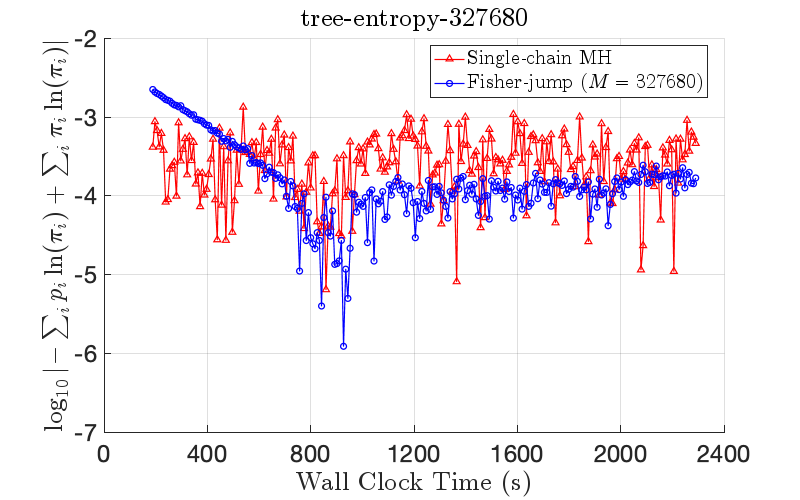}
     \end{subfigure}
          \begin{subfigure}{0.32\textwidth}
         \centering
         \includegraphics[width=\textwidth]{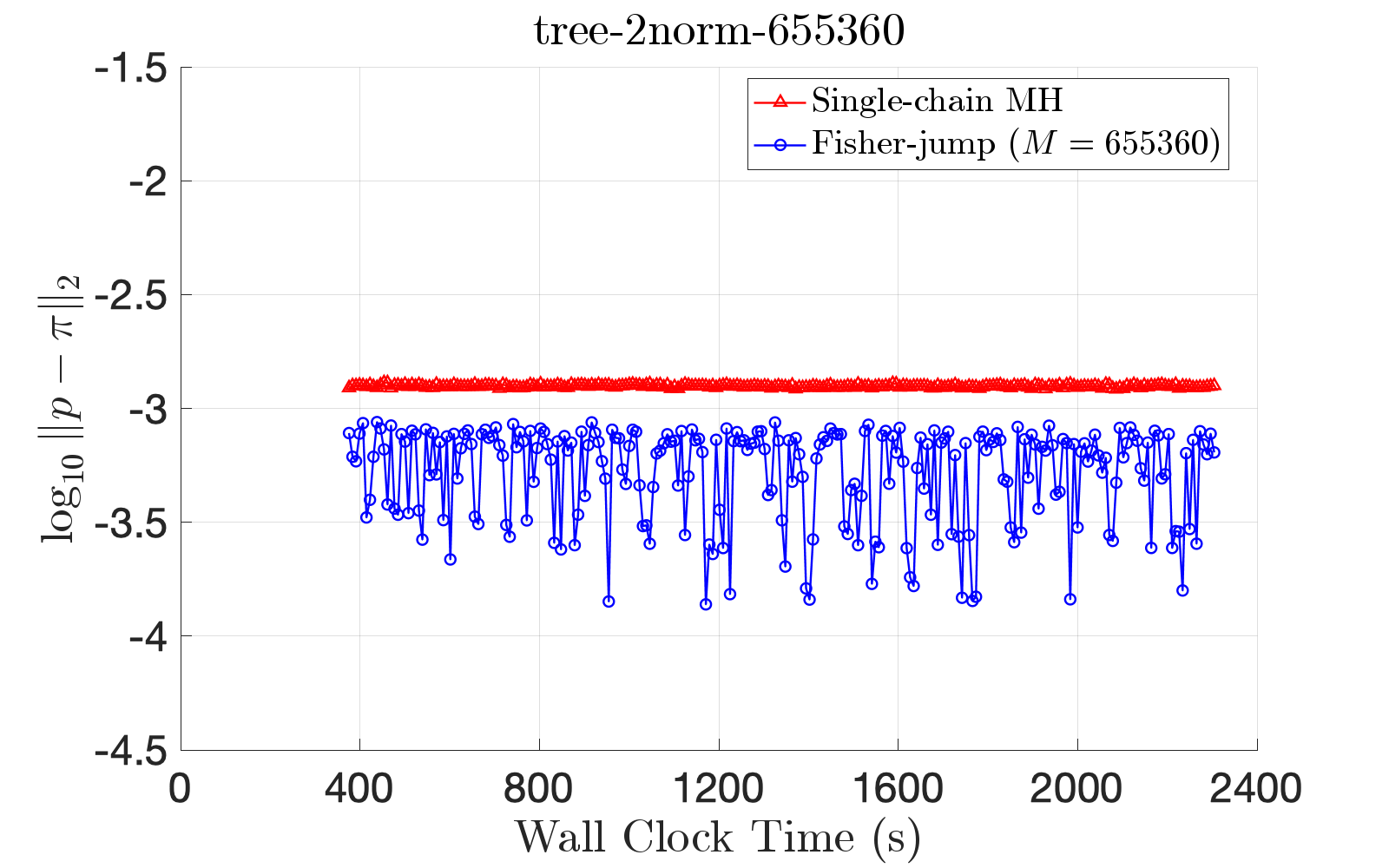}
     \end{subfigure}
    \begin{subfigure}{0.32\textwidth}
         \centering
         \includegraphics[width=\textwidth]{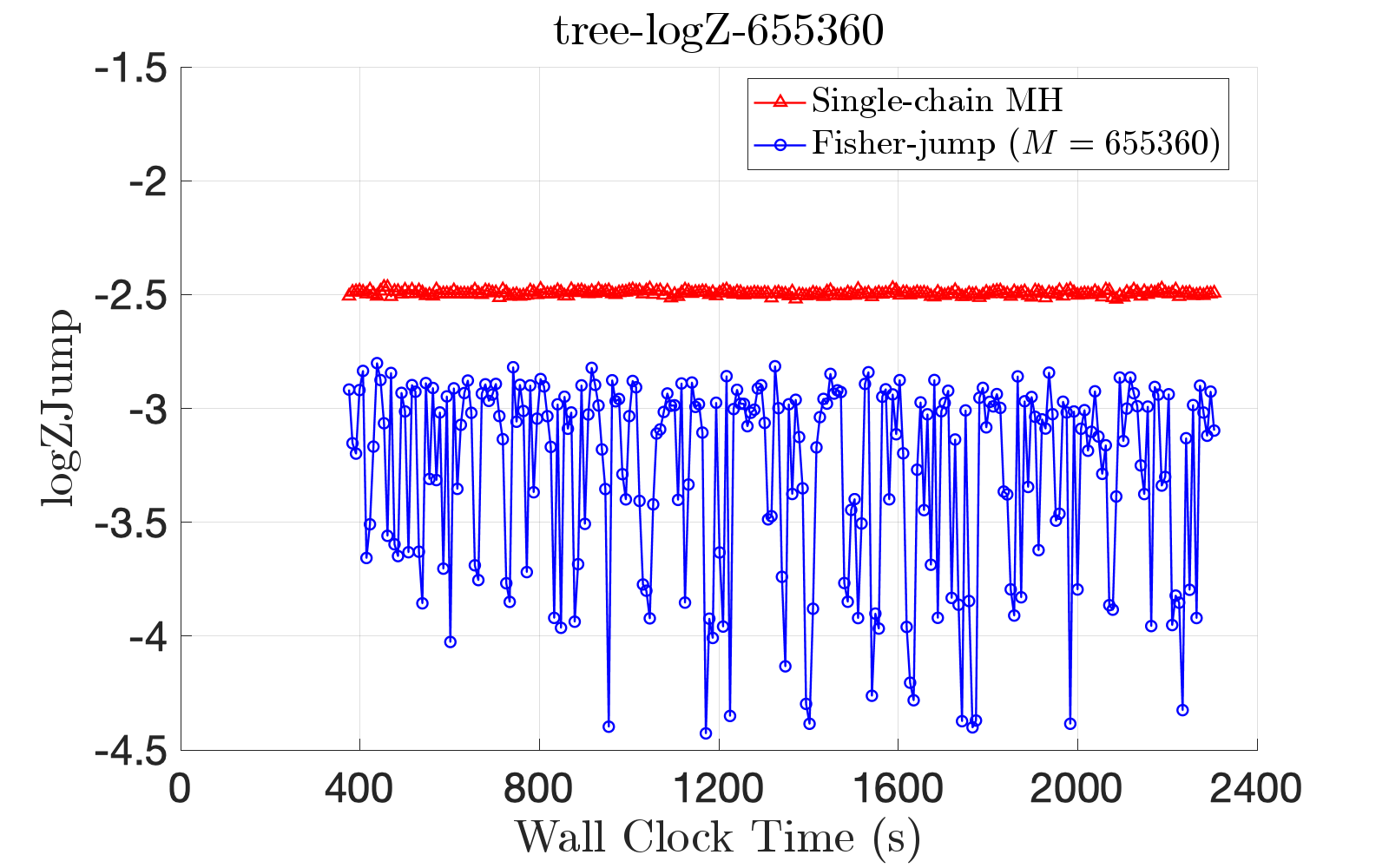}
     \end{subfigure}
     \begin{subfigure}{0.32\textwidth}
         \centering
         \includegraphics[width=\textwidth]{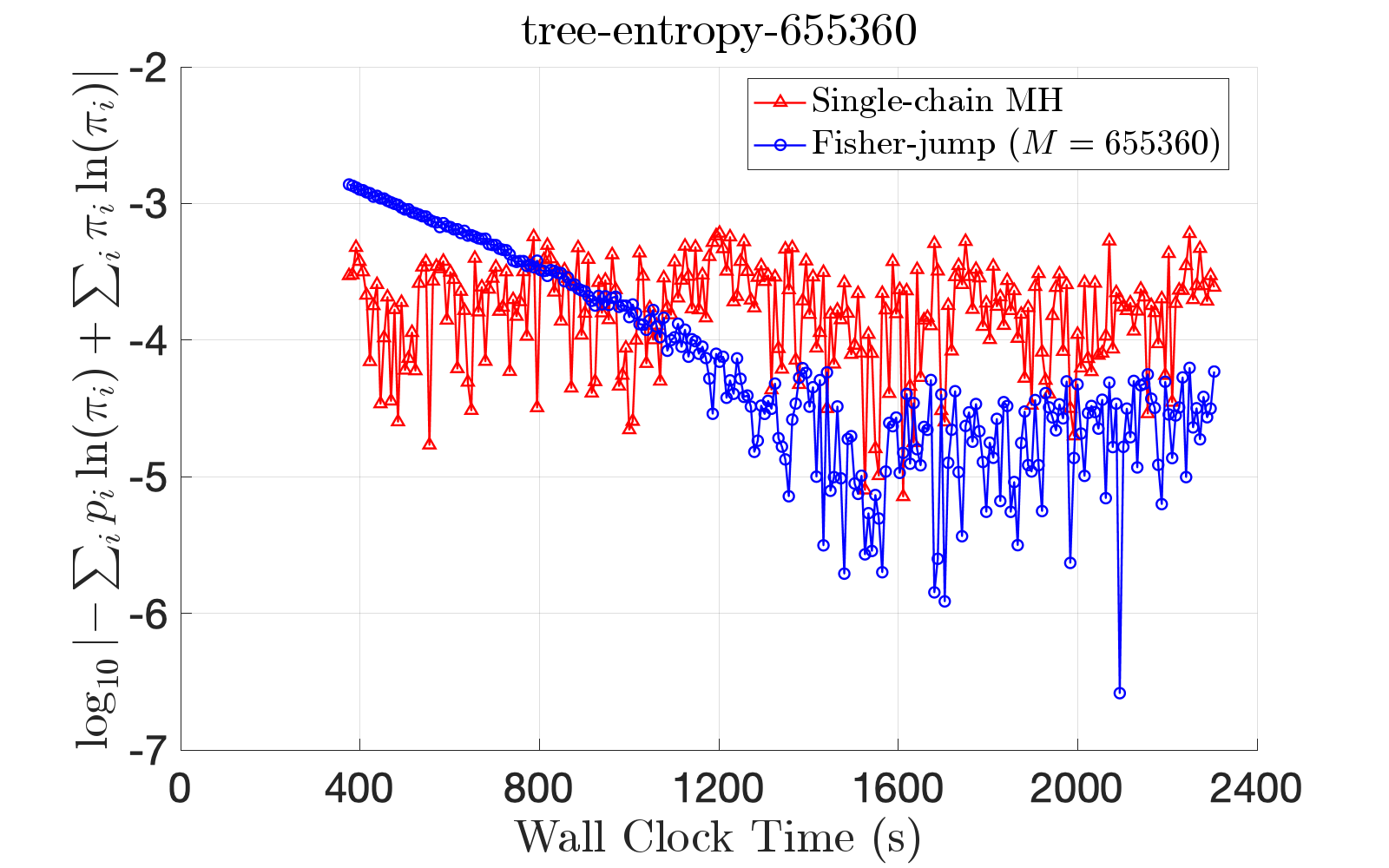}
     \end{subfigure}
\caption{Wall-clock comparisons on sampling image targets discretized on a $64\times 64$ lattice for the real dataset `fractal tree'. The empirical distribution of the MH method use the effective samples estimated from the integrated autocorrelation time for the entropy functional.
$x$-axes are in `wall-clock' time. The rows correspond to particle numbers $M=163840$, $2M$, and $3M$; the comparisons begin at 90.65, 188.58, and 375.45 seconds, respectively. The columns display error metrics $\log_{10}\norm{p(t)-\pi}_2$, $\log_{10} \abs{\sum_{i=1}^n p_i(t)\log \frac{p_i(t)}{Z \pi_i}- (-\log Z)}$ and $\log_{10}\abs{-\sum_{i=1}^n p_i \ln \pi_i + \sum_{i=1}^n \pi_i \ln \pi_i}$, respectively.  \label{fig:effectivesamples}}
\end{figure}

For the `fractal tree' target, the estimated integrated autocorrelation time is approximately $\tau_f=92$. In \Cref{fig:effectivesamples}, the displayed MH summaries retain one sample every 92 iterations. The \texttt{log-Fisher} has smaller $\ell_2$ and normalizing constant errors than MH among all tested particle numbers, while the two methods have comparable errors for the entropy functional. The first two columns should be regarded as supplementary
diagnostics because the autocorrelation time used here is associated with
the entropy. A more complete statistical comparison of estimator efficiency would
require for future study. The results also suggest that increasing the particle number in parallel can improve the \texttt{log-Fisher} estimate. Similar qualitative behavior is observed for the `rose' and `checkerboard' targets.

\subsection{Wall-clock benchmarks on the Ising Model}\label{sub:Ising}

Our final experiment examines our method on the Ising models in 1D or 2D with free boundary conditions. Let $L$ denote the total number of sites on a lattice, and let each spin $\sigma_i$ take values in $\{-1,1\}$. A configuration $\sigma=(\sigma_1,\ldots,\sigma_L)$ corresponds to a node in the $L$-dimensional hypercube (see \Cref{fig:hypercube}), and two configurations are adjacent when they differ by a single spin flip. The state space therefore contains $2^L$ states.

For nearest-neighbor interactions, the target distribution on the $L$-dimensional hypercube is
\[
\pi(\sigma)
=
\frac{1}{Z}
\exp\left(
\beta\sum_{(i,j)\in  E_{\mathrm{lat}}}J_{ij}\sigma_i\sigma_j
\right),\]
where $E_{\mathrm{lat}}$ is the edge set of the one-dimensional or two-dimensional lattice. For simplicity, we set $J_{ij}=1$ and use the critical inverse temperature $\beta=\frac{\ln (1+\sqrt{2})}{2}$. Away from the boundary, each site has two neighbors in one-dimensional lattice and four in two-dimensional lattice.

Due to the limitation of current multi-chain implementation of \texttt{log-Fisher} method, where we need density estimation and update the time-inhomogeneous transition-rate matrix $\bar{Q}^r_{\psi}$ in \eqref{eq:AMH-jump}, we currently cannot employ Ising model on our personal computer for very large $L$. The present experiment therefore focuses on a one-dimensional system with
$L=13$ sites. The GPU timing experiments in Apendix~\ref{app:gpu} investigates how GPU parallelization
can mitigate this computational bottleneck on two-dimensional Ising
targets with more sites.

For $L=13$, the configuration hypercube contains $2^{13}=8192$ nodes. For the single-chain MH sampler, the step size is set to be 1. For the \texttt{log-Fisher} method, we also use the default time step $1$, with adaptive
time stepping activated whenever necessary to preserve a valid transition
matrix. The base particle number $M=2^{26}=67108864$, and additional runs use particle numbers up to $3M$. Increasing the
particle number substantially reduces the frequency of restarts: in the
displayed runs, restarts occur in approximately $12.0\%$ of the
iterations for $M$ particles and approximately $0.1\%$ for both $2M$ and
$3M$ particles. Each \textrm{log-Fisher} run begins with one multi-chain MH update, after which we use the constant damping parameter $\gamma(t)=2\sqrt{\lambda_*}$.

\begin{figure}[hp!tb]
     \centering
     \begin{subfigure}{0.48\textwidth}
         \centering
         \includegraphics[width=\textwidth]{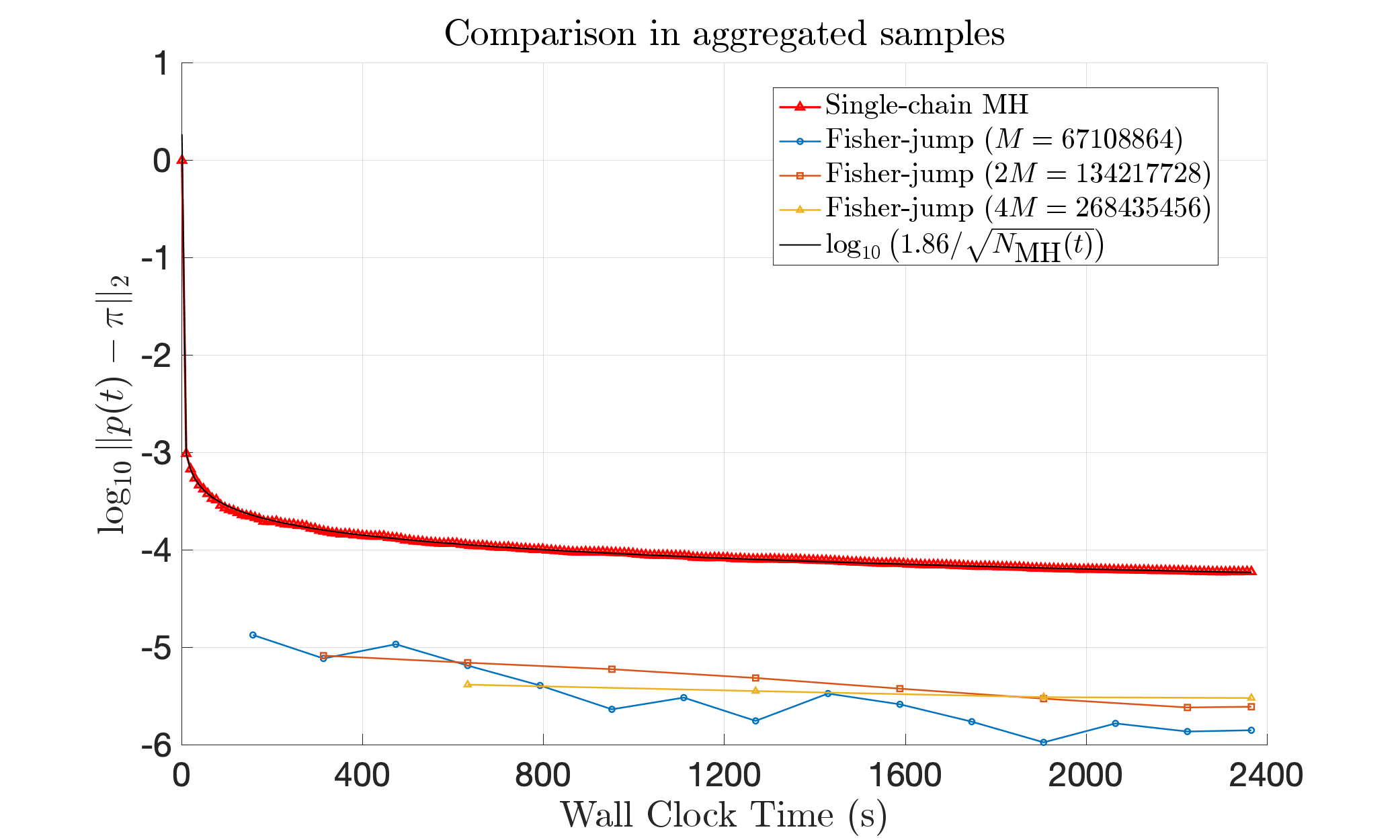}
     \end{subfigure}
    \begin{subfigure}{0.48\textwidth}
         \centering
         \includegraphics[width=\textwidth]{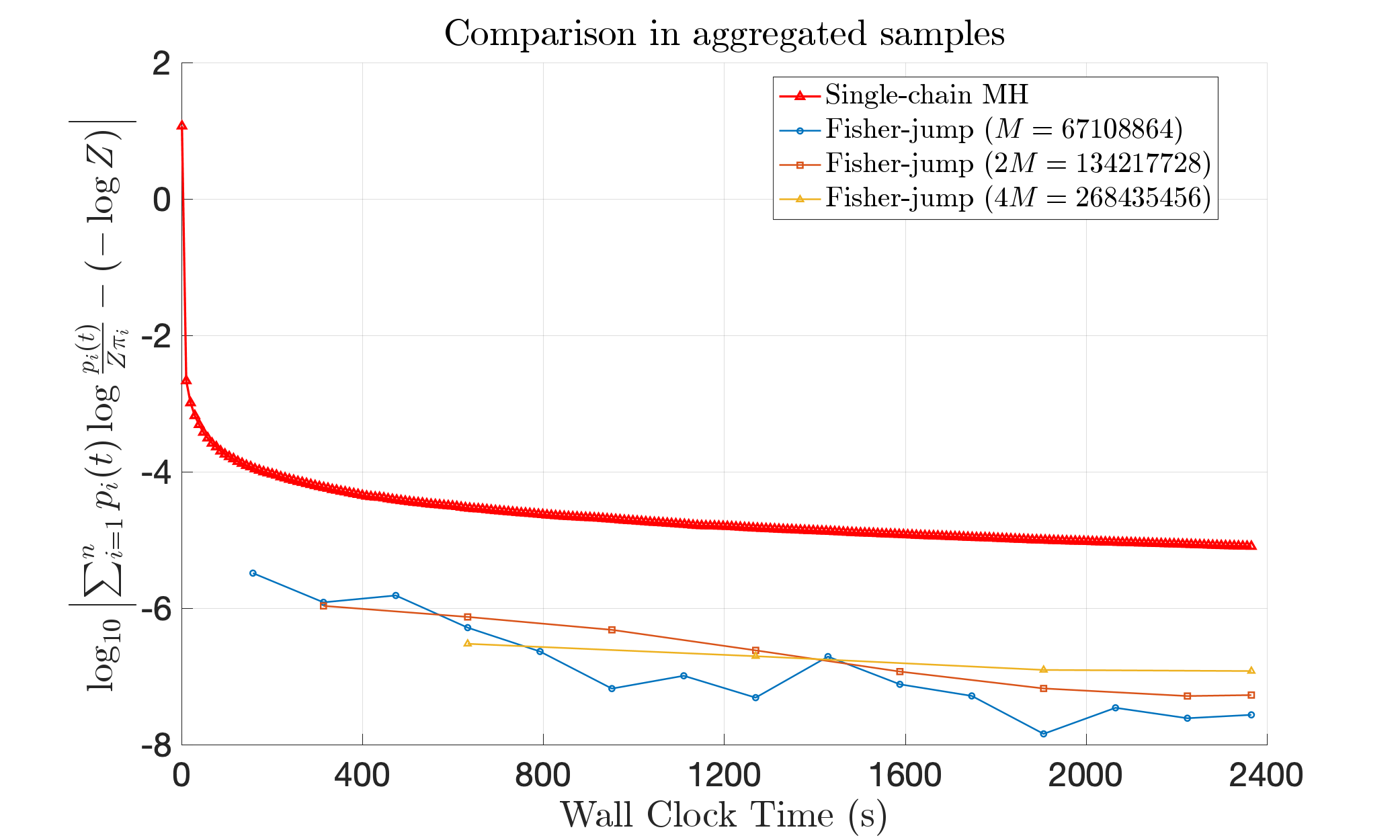}
     \end{subfigure}
\caption{Wall-clock comparisons on one-dimensional Ising model of 13 sites. The empirical distributions use the aggregated-sample protocol. $x$-axes are in `wall-clock' time. The left panel displays the approximation error $\log_{10} \norm{p(t)-\pi}_2$ w.r.t the target distribution $\pi$, the right panel displays the approximation error $\log_{10} \abs{\sum_{i=1}^n p_i(t)\log \frac{p_i(t)}{Z \pi_i}- (-\log Z)}$ w.r.t the normalizing constant $Z$. \label{fig:aggregatedsamples}}
\end{figure}

\Cref{fig:aggregatedsamples} uses the aggregated-sample protocol and compares accuracy in approximating the target distribution in $\ell^2$ norm and the normalizing constant under the same wall-clock budget. The MH method generates about $4.2\times 10^5$ samples per second thus it takes some time before the number of accumulated MH samples exceeds $M$. We observe that the \texttt{log-Fisher} method achieves higher accuracy in both metrics. In this experiment, we do not apply the effective-sample protocol,
since the MH method cannot accumulate $M$ approximately effective samples within
a comparable computational time.

}

\section{Discussions}\label{sec:dis}
In this paper, we design accelerated Markov Chain Monte Carlo (aMCMC) algorithms to sample target distributions on a discrete domain, by incorporating Nesterov's acceleration into the Metropolis-Hastings (MH) algorithm. The MCMC sampling algorithm can be formulated as an optimization problem in the space of probabilities. We study the MH algorithm as a gradient descent method for the KL divergence in the discrete Wasserstein-$2$ space. Our algorithm utilizes moment-based accelerated gradient methods in discrete Wasserstein-$2$ spaces and then develops an interacting Markov particle process in discrete domains. {Fixed-iteration tests validate accelerated convergence rate on small graphs, while the wall-clock benchmarks on image-derived and Ising model targets show smaller errors with respect to empirical distributional and
normalizing-constant than the single-chain MH benchmark, under two fair comparison protocols.}

In addition, the proposed accelerated sampling algorithm is related to the score functions on discrete domains. In \citet{lou2024discrete}, the discrete score function $\nabla \log p =\frac{\nabla p}{p}$ is approximated by $[\frac{p_j(t)-p_i(t)}{p_i(t)}]_{j\neq i}$. In this paper, we approximate the score function of a detailed-balanced MCMC algorithm by $\log\frac{p_j}{p_i}=\log p_j-\log p_i$ and the logarithmic mean \eqref{eq:log-mean}. In future work, we shall study these differences in the approximation of discrete score functions in terms of sampling complexities. 

Several open problems arising from the proposed accelerated MCMC framework merit further investigation. First, it is of interest to analyze the choice of the damping parameter $\gamma(t)$ with a given potential function. The choice of an optimal damping parameter is closely related to the geometric calculations on the generalized Wasserstein-$2$ manifold \citep{li2025geometriccalculationsprobabilitymanifolds}. Second, {a study of more sophisticated computational scheme and implementation, including the selection of the optimal step size and a structure-preserving discretization scheme, the GPU parallel computation, invite further investigation}. Third, it would be natural to incorporate additional acceleration techniques, such as irreversible jump samplers \citep{Ma2019Irreversible}, into the present framework, with particular emphasis on deriving single-chain approximations of the current swarm-based dynamics. Moreover, another interesting direction is to investigate the efficiency of score estimation in aMCMC. This estimation can be formulated as a variational framework based on time-reversible diffusion on discrete states.

\acks{S. Liu, B. Zhou and X. Zuo are supported by AFOSR YIP award No. FA9550-23-1-0087. S. Liu acknowledges support from the First Year Assistant Professor (FYAP) Award (ID: 38075) at Florida State University. W. Li is supported by AFOSR YIP award No. FA9550-23-1-0087, NSF RTG: 2038080, NSF FRG: 2245097, 
 NSF DMS: 2601952, and the McCausland Faculty
Fellowship at the University of South Carolina. Use was made of the computational facilities administered by the Center for Scientific Computing at the CNSI and MRL (an NSF MRSEC; DMR-2308708) and purchased through NSF CNS-1725797. All authors would like to thank the support of the AMS MRC conference: Ricci Curvatures of Graphs and Applications to Data Science in 2023. The discussion of this manuscript was initiated in the workshop. B. Zhou would like to extend appreciation to Andy Wan for a fruitful discussion. {We are also grateful to anonymous reviewers for the incisive comments and suggestions, which improved the quality of this work.}}

\appendix

\crefalias{section}{appendix}

\section{Graph Laplacian and detailed derivation of Wasserstein gradient}\label{app:derivation_wg}

\begin{lemma}\label{lem: grph_lap}
Suppose the Markov chain is irreducible and reversible with respect to some strictly positive distribution $\pi$. Assume $p\in\mathbb{P}(V)$ is strictly positive. Then, the graph Laplacian $\mathbb{K}(p)$ is positive semi-definite with $\mathrm{Ker}(\mathbb{K}(p)) = \mathrm{span}(\Id_n)$.
\end{lemma}
\begin{proof}
  Recall $\mathbb{K}(p)$ is the graph Laplacian of a weighted undirected graph with a weight matrix $(\omega_{ij} \theta_{ij}(p))$, for any $x=[x_1,\ldots,x_n]$, the quadratic form is
\[x\mathbb{K}(p)x^{\top} = -\sum_{1\leqslant i < j \leqslant n} \mathbb{K}_{ij}(p)(x_i-x_j)^2 = - \sum_{1 \leqslant i < j  \leqslant  n } \omega_{ij}\theta_{ij}(p)(x_i - x_j)^2 \leqslant 0.\]
The last inequality is due to $\omega_{ij}\geqslant 0$ and $\theta_{ij}(p) = \frac{\frac{p_i}{\pi_i} - \frac{p_j}{\pi_j}}{f'(\frac{p_i}{\pi_i}) - f'(\frac{p_j}{\pi_j})} > 0$. This verifies that $\mathbb{K}(p)$ is negative definite.

Furthermore, assume $y \in\mathrm{Ker}(\mathbb{K}(p))$, this leads to $y\mathbb{K}(p) y^\top = 0$, which further leads to $\mathbb{K}_{ij}(p)(y_i-y_j)^2 = 0$ for any $1 \leqslant i< j \leqslant n$. 

Recall that the Markov chain is reducible; for arbitrary $1\leqslant k < l \leqslant n$, we can find a path of nodes $k=i_1\to i_2\to \cdots \to i_m=l$ such that $\omega_{i_j\, i_{j+1}}>0$ for $j=1,\ldots,m-1$, which yields $\mathbb{K}_{i_j\, i_{j+1}}(p)>0$. This leads to $y_{k}=y_{i_1}=\cdots=y_{i_m}=y_l$. This verifies that $y\in\mathrm{span}(\Id_n)$, thus $\mathrm{Ker}(\mathbb{K}(p))\subset \mathrm{span}(\Id_n)$. It is clear that $\mathrm{span}(\Id_n)\subset \mathrm{Ker}(\mathbb{K}(p))$ since $\mathbb{K}(p)$ is row-sum-zero and symmetric. Thus we have proved $\mathrm{Ker}(\mathbb{K}(p)) = \mathrm{span}(\Id_n)$.
\end{proof}

\noindent
\textbf{Derivation of $\mathrm{grad}\mathrm{D}_f(p\|\pi)$:}
Denote the tangent space of $\mathbb{P}(V)$ at $p$ by $\mathrm{T}_p\mathbb{P}(V)$, and denote the Wasserstein gradient on $(\mathbb{P}(V), g_W)$ by $\mathrm{grad} \mathrm{D}_f(p\|\pi)\in\mathrm{T}_p \mathbb{P}(V)=\Id_n^\perp$. We consider arbitrary curve $\{p(t)\}$ passing through $p$ at $t=0$, by 
the definition of the gradient operator on Riemannian manifold \citep[see][Chapter~2]{Lee2018Introduction}
    \begin{equation*}
      \frac{\mathrm{d} }{\mathrm{d} t}\mathrm{D}_f(p_t\|\pi)\Big|_{t=0} = g_W(\mathrm{grad} \mathrm{D}_f(p\|\pi), \ \dot p(0)).
    \end{equation*}
Denote $\nabla_p \mathrm{D}_f(p\|\pi)$ as the flat gradient of $f$-divergence with respect to $p$, for any $\dot p(0)\in \Id_n^\perp$, we have
\begin{equation*}
     \nabla_p \mathrm{D}_f(p\|\pi) (\dot p(0))^{\top} = \mathrm{grad} \mathrm{D}_f(p\|\pi)\mathbb{K}(p)^\dagger (\dot p(0) )^\top.
\end{equation*}

Let $\Pi:\mathbb{R}^n \rightarrow \mathbb{R}^n$ be the orthogonal projection onto $\mathrm{T}_p \mathbb{P}(V) = \Id_n^\perp$, the above equation yields
\[ \Pi( \mathrm{grad} \mathrm{D}_f(p\|\pi)\mathbb{K}(p)^\dagger ) = \Pi(\nabla_p \mathrm{D}_f(p\|\pi)). \]
Since $\mathbb{K}(p)^\dagger$ is the Moore-Penrose inverse of $\mathbb{K}(p)$, we have $\Pi = \mathbb{K}(p)^\dagger \mathbb{K}(p) = \mathbb{K}(p)\mathbb{K}(p)^\dagger$. The above equation then leads to
\[ ( \mathrm{grad} \mathrm{D}_f(p\|\pi) \mathbb{K}(p)^\dagger ) \mathbb{K}(p) \mathbb{K}(p)^\dagger = \nabla_p \mathrm{D}_f(p\|\pi) \mathbb{K}(p)\mathbb{K}(p)^\dagger. \]
The left-hand side equals $\mathrm{grad} \mathrm{D}_f(p\|\pi)\mathbb{K}(p)^\dagger ,$ as $\mathrm{grad} \mathrm{D}_f(p\|\pi)\in\mathrm{T}_p \mathbb{P}(V) = \Id_n^\perp.$ Then,
\[ \left[\mathrm{grad}\mathrm{D}_f(p\|\pi) - \nabla_p \mathrm{D}_f(p\|\pi) \mathbb{K}(p)  \right] \mathbb{K}(p)^\dagger = 0. \]
Furthermore, as
\[ \mathrm{Ker}(\mathbb{K}(p)^\dagger) = \mathrm{Ker}(\mathbb{K}(p)) = \mathrm{span}(\Id_n), \]
we have
\[  \mathrm{grad}\mathrm{D}_f(p\|\pi) - \nabla_p \mathrm{D}_f(p\|\pi)\mathbb{K}(p)  \in \mathrm{span}(\Id_n). \]
On the other hand, notice that $\mathrm{Ran}(\mathbb{K}(p)) = \mathrm{Ker}(\mathbb{K}(p)^\top)^\perp = \mathrm{Ker}(\mathbb{K}(p))^\perp = \Id_n^\perp$, we have $ \nabla_p \mathrm{D}_f(p\|\pi) \mathbb{K}(p) \in \Id_n^\perp$. As a result,
\[ \mathrm{grad}\mathrm{D}_f(p\|\pi) - \nabla_p \mathrm{D}_f(p\|\pi) \mathbb{K}(p)  \in \mathrm{span}(\Id_n)\cap\Id_n^\perp = 0. \]

\section{Eigenvalues and eigenvectors of $Q$-matrix}\label{app:Q-matrix}

While the majority of the MCMC literature emphasizes the transition probability matrix $P$, rather than the transition rate matrix $Q$, the two are closely related. For completeness, we include a self-contained discussion.

Recall $\omega=(\omega_{ij})=\textrm{diag}(\pi)Q$ is a symmetric and row-zero-sum matrix.

\begin{lemma}[Sylvester's law of inertia]
Two symmetric square matrices $A,B$ of the same size have the same number of positive, negative and zero eigenvalues if and only if they are congruent, that is, $B=QAQ^T$ for some non-singular matrix $Q$.
\end{lemma}

\begin{lemma}
Given a nondegenerate target distribution $\pi$ and an irreducible transition rate matrix $Q$ that satisfy the detailed balance condition. Now define $\omega=\mathrm{diag}(\pi) Q$. Then
\begin{enumerate}[label=\alph*)]
    \item $\omega$ is a negative semi-definite matrix. Furthermore, $\mathrm{Ker}(\omega) = \mathrm{span}(\Id_n)$, and $\omega$ is negative definite on $\Id_n^\perp$.
    \item The eigenvalues to $Q$ consist of 0 with algebraic multiplicity 1, along with $(n-1)$ negative eigenvalues.
\end{enumerate}
\end{lemma}
\begin{proof}
For part a), follow the same proof technique in \Cref{lem: grph_lap}, we can show that for arbitrary $x=[x_1,\ldots,x_n]$, $x\omega x^\top \leqslant 0$, as well as $\mathrm{Ker}(\omega) = \mathrm{span}(\Id_n)$.

Combining these facts, together with $\omega$ being symmetric matrix, it is straightforward to prove that $\omega$ is negative definite on $\Id_n^\perp$. Since $0$ is apparently an eigenvalue to $\omega$ with the eigenvector $\Id_n$, Recall all eigenvalues to a real symmetric matrix are real, thus $\omega$ persists $(n-1)$ real negative eigenvalues.  

For part b), note that 
    \[Q=\mathrm{diag}(\pi)^{-1} \omega=\sqrt{\mathrm{diag}(\pi)}^{-1} (\sqrt{\mathrm{diag}(\pi)}^{-1} \omega \sqrt{\mathrm{diag}(\pi)}^{-1})\sqrt{\mathrm{diag}(\pi)},\]
Since $\omega$ and $(\sqrt{\mathrm{diag}(\pi)}^{-1} \omega \sqrt{\mathrm{diag}(\pi)}^{-1})$ are congruent, by the Sylverster's law of inertia, they have the same number of negative eigenvalues and zero eigenvalue. $(\sqrt{\mathrm{diag}(\pi)}^{-1} \omega \sqrt{\mathrm{diag}(\pi)}^{-1})$ and $Q$ are similar, thus they have the same eigenvalues with the same algebraic multiplicity. In summary, $Q$ has the same number of negative eigenvalues and zero eigenvalue as $\omega$.
\end{proof}

Alternatively, given an irreducible transition rate matrix $Q$, we may define $P=I_n+Q\Delta t$ for small enough $\Delta t$ such that $P$ is an irreducible, aperiodic transition probability matrix. Once applying Perron-Frobenius Theorem onto $P$ and combining the detailed balance, we obtain the same conclusion about $Q$.

\begin{lemma}[Perron-Frobenius Theorem]
Any irreducible, aperiodic transition matrix $P$ has an eigenvalue $\lambda_0=1$ with algebraic multiplicity 1 as the Perron-Frobenius eigenvalue, and all other eigenvalues $\lambda_i$ satisfy $\abs{\lambda_i}<1$.
\end{lemma}

\section{Derivation of our models}\label{app:derivation}

\subsection{General Jump Process \eqref{eq:Kolmogorov-Ham}}
To construct a sampling algorithm (a ``jump process'' of particles on $V$), we rewrite the equation of the state variables $p$ in the form of the forward master equation \eqref{eq:Kolmogorov}, while leaving the momentum variables $\psi$ as a parameter in the transition rate matrix $Q$. By splitting the positive part $(\psi_i-\psi_j)_+:=\max\{\psi_i-\psi_j, 0\}$ and the negative part $(\psi_i-\psi_j)_{-}:=-\min\{\psi_i-\psi_j, 0\}$, recall $(\omega_{ij}), (\theta_{ij}(p))$ are symmetric, we have
    \[
    \begin{aligned}
    &\sum_{j\neq i} \omega_{ij}\theta_{ij}(p) (\psi_i - \psi_j)=\sum_{j\neq i} \omega_{ij}\theta_{ij}(p) \left[(\psi_i-\psi_j)_+ - (\psi_i-\psi_j)_{-}\right]\\
    =&-\sum_{j\neq i} \omega_{ij}\theta_{ij}(p) (\psi_i-\psi_j)_{-} + \sum_{j\neq i}\omega_{ji} \theta_{ji}(p) (\psi_i-\psi_j)_{+}\\
    =&-\left[\sum_{j\neq i} \frac{\omega_{ij}\theta_{ij}(p) (\psi_i-\psi_j)_{-} }{p_i}\right]p_i + \displaystyle\sum_{j\neq i} \frac{\omega_{ji}\theta_{ji}(p) (\psi_i-\psi_j)_{+}}{p_j} p_j\\
    =&p \begin{bmatrix}
        \dfrac{\omega_{1i} \theta_{1i}(p) (\psi_i-\psi_1)_+}{p_1} & \cdots & -\left[\displaystyle\sum_{j\neq i} \dfrac{\omega_{ij}\theta_{ij}(p)  (\psi_i-\psi_j)_{-} }{p_i}\right] & \cdots & \dfrac{\omega_{ni}\theta_{ni}(p)  (\psi_i-\psi_n)_+}{p_n}
    \end{bmatrix}^\top
    \end{aligned}
    \]
This contributes the $i$-th column of the matrix $\bar{Q}^r_{\psi}$ in \eqref{eq:AMH-jump}. One can easily verify that $\bar{Q}^r_{\psi}$ is row-sum-zero. On $i$-th row, if $\psi_i\geqslant \psi_j$, then $(\psi_j - \psi_i)_+ = 0$ while $(\psi_i - \psi_j)_-=0$; if $\psi_i\leqslant \psi_j$, then $(\psi_j-\psi_i)_+ = \psi_j-\psi_i$ while $(\psi_i - \psi_j)_- = \psi_j-\psi_i$. In either cases, the cancellation results in row-sum-zero.

\subsection{\texttt{Chi-squared} method}

$\theta_{ij}=1$ and $\mathcal{U}(p)=\D_{f}(p \| \pi)=\frac{1}{2}\sum_{i=1}^n \frac{(p_i-\pi_i)^2}{\pi_i}$ are obtained from choosing $f(x)=\frac{1}{2}\abs{x-1}^2$ in \eqref{eq:phi-average} and \eqref{eq:phi-divergence}. \eqref{eq:chi-1} is obtained from \eqref{eq:NODE} by the direct plug-in.

From \eqref{eq:chi-1}, we take additional derivative w.r.t $t$,
    \begin{align*}
    \frac{\mathrm{d}^2p_i}{\mathrm{d}t^2}&=\frac{\mathrm{d}\psi_i}{\mathrm{d}t}\sum_{j\neq i}\omega_{ij} -\sum_{j\neq i}\frac{\mathrm{d}\psi_j}{\mathrm{d}t}\omega_{ji}\\
    &=-\gamma(t)\left(\psi_i \sum_{j\neq i}\omega_{ij}-\sum_{j\neq i}\psi_j\omega_{ji}\right)-\frac{p_i}{\pi_i}\sum_{j\neq i}\pi_i Q_{ij} + \sum_{j\neq i}\frac{p_j}{\pi_j}\pi_j Q_{ji} + \sum_{j\neq i}\omega_{ij} - \sum_{j\neq i}\omega_{ji}\\
    &=-\gamma(t) \frac{\mathrm{d}p_i}{\mathrm{d}t} - \sum_{j\neq i} (p_i Q_{ij} - p_j Q_{ji}),
    \end{align*}
where we use $\omega_{ij}=\omega_{ji}$, and obtain the Telegrapher's equation \eqref{eq:telegrapher}.

Denote $\Delta t$ by the step size, one possible time-discretization for the Telegrapher's equation \eqref{eq:telegrapher} reads
    \[
    \frac{p_i^{(k+2)}-2 p_i^{(k+1)}+ p_i^{(k)}}{(\Delta t)^2} + \gamma(t) \frac{p_i^{(k+1)}-p_i^{(k)}}{\Delta  t}=\sum_{j\neq i}(p_j^{(k)} Q_{ji}- p_i^{(k)} Q_{ij}).
    \]
Thus we obtain 
    \[
    p_i^{(k+2)}=p^{(k)}_i + (2-\gamma(t) \Delta t)(p_i^{(k+1)}-p_i^{(k)})+ (\Delta t)^2 \sum_{j\neq i}(p_j^{(k)} Q_{ji}- p_i^{(k)} Q_{ij}),
    \]
and note the discretization of \eqref{eq:NODE_1} is 
    \[
    \begin{aligned}
        \frac{p_i^{(k+1)}-p_i^{(k)}}{\Delta t}&=\sum_{j\neq i}\omega_{ij} (\psi_i^{(k)}-\psi_j^{(k)})\\
        &=\sum_{j\neq i} \frac{\pi_j Q_{ji} (\psi_i^{(k)}-\psi_j^{(k)})_+}{p_j^{(k)}}p_j^{(k)}+\left[-\sum_{j\neq i}\frac{\pi_i Q_{ij}(\psi_i^{(k)}-\psi_j^{(k)})_-}{p_i^{(k)}}\right]p_i^{(k)}.
    \end{aligned}
    \]
Combining both discretizations, the update of the \texttt{Chi-squared} method is given by
    \begin{multline*}
    p_i^{(k+2)}=p_i^{(k)}+ (\Delta t)^2 \sum_{j\neq i}\left[Q_{ji} + \frac{2-\gamma(t) \Delta t}{\Delta t}\frac{\pi_j Q_{ji} (\psi_i^{(k)}-\psi_j^{(k)})_+}{p_j^{(k)}}\right]p_j^{(k)}\\ 
    + (\Delta t)^2 \left[-Q_{ij} - \sum_{j\neq i}\frac{\pi_i Q_{ij}(\psi_i^{(k)}-\psi_j^{(k)})_-}{p_i^{(k)}}\right] p_i^{(k)}.
    \end{multline*}
This is a nonlinear jump process from $p_i^{(k)}$ to $p_i^{(k+2)}$. As a comparison, the jump process from MH is
    \[
    p_i^{(k+1)}=p_i^{(k)}+\Delta t \sum_{j\neq i} \Big[Q_{ji} p_j^{(k)}- Q_{ij} p_i^{(k)}\Big].
    \]
\subsection{\texttt{KL} method}

We pick $f(x)=x\log x$ in \eqref{eq:phi-average} and \eqref{eq:phi-divergence}, which results in $\theta_{ij}(p)=\dfrac{\frac{p_j}{\pi_j}-\frac{p_i}{\pi_i}}{\log\frac{\pi_i p_j}{\pi_j p_i}}$ and $\mathcal{U}(p)=\D_{f}(p \| \pi)=\sum_{i=1}^n p_i\log\frac{p_i}{\pi_i}$. We can compute 
    \[
    \frac{\partial \theta_{ij}(p)}{\partial p_i}=\frac{\frac{\log\frac{p_i}{\pi_i}-\log\frac{p_j}{\pi_j}}{\pi_i}-(\frac{p_i}{\pi_i}-\frac{p_j}{\pi_j})\frac{\pi_i}{p_i}\frac{1}{\pi_i}}{(\log\frac{p_i}{\pi_i}-\log\frac{p_j}{\pi_j})^2}=\frac{\log\left(\frac{\pi_j p_i}{\pi_i p_j}\right)-1+ \left(\frac{\pi_j p_i}{\pi_i p_j}\right)^{-1}}{\log^2\left(\frac{\pi_j p_i}{\pi_i p_j}\right)}\frac{1}{\pi_i}\,.
    \]
Plugging into \eqref{eq:NODE}, we obtain \eqref{eq:KL}.

\subsection{\texttt{log-Fisher} method}

We still select  $f(x)=x\log x$ for $\theta(x,y)=\frac{x-y}{f'(x)-f'(y)}$ just as in \texttt{KL} method, and plug it in the relative Fisher information $\I(p\| \pi)$
    \[
       \I(p \| \pi) = \frac{1}{4}\sum_{i,j=1}^n \omega_{ij}\theta_{ij}(p) (f'(\frac{p_i}{\pi_i})-f'(\frac{p_j}{\pi_j}))^2=\frac{1}{4}\sum_{i=1}^n \sum_{j\neq i} \omega_{ij}(\log \frac{p_i}{\pi_i}-\log\frac{p_j}{\pi_j})(\frac{p_i}{\pi_i}-\frac{p_j}{\pi_j}).
    \]
Note that $
    \frac{\partial}{\partial p_i} \I(p\| \pi_i)=\frac{1}{2\pi_i}\sum_{j\neq i}\omega_{ij}[\log\left(\frac{\pi_j p_i}{\pi_i p_j}\right)+1-\left(\frac{\pi_j p_i}{\pi_i p_j}\right)^{-1}]$. Hence, \eqref{eq:NODE} turns into \eqref{eq:log-Fisher}.

\section{Hessian and geodesic convexity}\label{app:Hessian}

We follow the setting and notations of discussions in \cite{Mielke2013Geodesic}. Consider the space $\P(V)$ of probability measures on a graph $G=(V,E,\omega)$, with the Onsager's response matrix $\mathbb{K}$ defined in \eqref{def:Onsager_response}, let $\{p(t)\}_{t \geq 0}$ be a geodesic on $\P(V)$, then it satisfies the geodesic equations:
    \[
    \left\{
    \begin{aligned}
        \dot{p}(t)&=\psi \mathbb{K}\\
        \dot{\psi}(t)&= -\nabla_p \left(\frac{1}{2}\psi  \mathbb{K} \psi^{\top}\right),
    \end{aligned}\right.
    \]

Any functional $\mathcal{U}(p)$ that is geodesically $\lambda$-strongly convex can be characterized by
    \[\frac{\mathrm{d}^2}{\mathrm{d}t^2}\mathcal{U}(p(t))\geqslant \lambda \dot{p}(t)\mathbb{K}^{\dagger}(p(t))(\dot{p}(t))^{\top}.\]

In our case, $\mathbb{K}(p)$ does not depend on $p$. Here we denote $\mathrm{D}^2\mathcal{U}(p)$ as the Hessian of $\mathcal U(p)$ in the Euclidean space. The geodesic equation is simplified as 
    \[\left\{
    \begin{aligned}
        \dot{p}(t)&=\psi K\\
        \dot{\psi}(t)&=0\,.
    \end{aligned}\right.
    \]
Thus $\mathcal{U}(p)$ is geodesically $\lambda$-strongly convex if and only if $\dot{p}\mathrm{D}^2\mathcal{U}(p)\dot{p}^{\top}\geqslant \lambda \dot{p}K^{\dagger}\dot{p}^{\top}$, or equivalently, $\psi K \mathrm{D}^2\mathcal{U}(p)K^{\top} \psi^{\top}\geqslant \lambda \psi K \psi^{\top}$. Thus under the flat metric $K$, we can say $\mathcal{U}(p)$ is geodescially $\lambda$-strongly convex, if there exists some $\lambda>0$, such that for any $t$, the Rayleigh quotient problem has a lower bound along the dynamics:
    \[\lambda\leqslant  \min_{\psi \Id_n^{\top}=0} \dfrac{\psi K \mathrm{D}^2\mathcal{U}(p(t))K^{\top}\psi^{\top}}{\psi K\psi^{\top}}.\]

In the \texttt{con-Fisher} method, recall $\mathcal{U}(p)=\I(p\| \pi)=\frac{1}{4}\sum_{i=1}^n \sum_{j\neq i}\omega_{ij}\theta_{ij}(\log \frac{p_i}{\pi_i}-\frac{p_j}{\pi_j})^2$. Thus $\partial_{ii} \mathcal{U}(p)= \sum_{j\neq i} \frac{\omega_{ij}\theta_{ij}}{p_i^2} [1-\log \left(\frac{\pi_j p_i}{\pi_i p_j}\right)]$ and $\partial_{ij} \mathcal{U}(p)= -\dfrac{\omega_{ij}\theta_{ij}}{p_i p_j}$, thus the Hessian $\mathrm{D}^2\mathcal{U}(p)$ is given by
    \[
\begin{bmatrix}
    \displaystyle\sum_{j\neq 1} \frac{\omega_{1j}\theta_{1j}}{p_1^2} [1-\log \left(\frac{\pi_j p_1}{\pi_1 p_j}\right)] & -\dfrac{\omega_{12}\theta_{12}}{p_1 p_2} & \cdots & -\dfrac{\omega_{1n}\theta_{1n}}{p_1 p_n}\\
    -\dfrac{\omega_{21}\theta_{21}}{p_2 p_1} & \displaystyle\sum_{j\neq 2} \frac{\omega_{2j}\theta_{2j}}{p_2^2} [1-\log \left(\frac{\pi_j p_2}{\pi_2 p_j}\right)] & \cdots & -\dfrac{\omega_{2n}\theta_{2n}}{p_2 p_n}\\
    \vdots & \cdots & \ddots & \vdots\\
    -\dfrac{\omega_{n1}\theta_{n1}}{p_n p_1} & -\dfrac{\omega_{n2}\theta_{n2}}{p_n p_2} & \cdots & \displaystyle\sum_{j\neq n} \frac{\omega_{nj}\theta_{nj}}{p_n^2} [1-\log \left(\frac{\pi_j p_n}{\pi_n p_j}\right)]
\end{bmatrix}.
\]

In the asymptotical limit $\abs{\frac{p_i}{\pi_i}-1}\propto \mathcal{O}(10^{-k})$ for some $k\gg 1$, then
    \[
        \log \frac{\pi_j p_i}{\pi_i p_j}\approx (c_i-c_j)10^{-k};\qquad
        \frac{1}{p_i^2}\approx \frac{1}{\pi_i^2}-\frac{2c_i}{p_i^2}10^{-k};\qquad\textrm{and}\qquad
        \frac{1}{p_i p_j}\approx \frac{1}{\pi_i \pi_j}-\frac{c_i+c_j}{\pi_i \pi_j}10^{-k}.
    \]
Thus the Hessian $\mathrm{D}^2\mathcal{U}(p)$ is approximated by
    \[
    \begin{aligned}
    &\begin{bmatrix}
     \frac{1}{\pi_1} & 0 & \cdots & 0\\
     0 & \frac{1}{\pi_2} & \vdots & 0\\
    0 & 0 & \ddots & 0\\
    0 & 0 & \cdots & \frac{1}{\pi_n}
\end{bmatrix}\begin{bmatrix}
    \sum_{j\neq 1}\omega_{1j}\theta_{1j} & -\omega_{12}\theta_{12} & \cdots & -\omega_{1n}\theta_{1n}\\
    -\omega_{21}\theta_{21} & \sum_{j\neq 2}\omega_{2j}\theta_{2j} & \vdots & -\omega_{2n}\theta_{2n}\\
    \vdots & \cdots & \ddots & \vdots\\
    -\omega_{n1}\theta_{n1} & \omega_{n2}\theta_{n2} & \cdots & \sum_{j\neq n}\omega_{nj}\theta_{nj}
\end{bmatrix}\begin{bmatrix}
    \frac{1}{\pi_1} & 0 & \cdots & 0\\
    0 & \frac{1}{\pi_2} & \vdots & 0\\
    0 & 0 & \ddots & 0\\
    0 & 0 & \cdots & \frac{1}{\pi_n}
\end{bmatrix}\\
&=\mathrm{diag}(\pi^{-1})K\mathrm{diag}(\pi^{-1}).
\end{aligned}
    \]
Moreover, $\mathrm{D}^2\mathcal{U}(p)\vert_{p=\pi}=\mathrm{diag}(\pi^{-1})K\mathrm{diag}(\pi^{-1}).$

Let's go back to the Rayleigh quotient problem. As $K$ is symmetric, and $\mathrm{Ker}(K) = \mathrm{span}(\Id_n)$, the singular value decomposition (SVD) of $K$ is given by
    \[ K = U \Sigma U^\top = \widehat{U} \widehat{\Sigma} \widehat{U}^\top. \]
Here $U\in\mathbb{R}^{n \times n}$ is an orthogonal matrix, $\Sigma = \mathrm{diag}(\sigma_1, \dots, \sigma_{n-1}, 0)$ with $\sigma_1 \geq \dots \geq \sigma_{n-1}>0$, $\widehat{U}\in\mathbb{R}^{n\times (n-1) }$ denotes the matrix formed by the first $n-1$ columns of $U$, and $\widehat{\Sigma} = \mathrm{diag}(\sigma_1, \dots, \sigma_{n-1})$. Then, we can reformulate the Rayleigh quotient as
    \[
     \min_{\psi\Id_n^\top = 0}  \frac{\psi \widehat{U} \widehat{\Sigma} \widehat{U}^\top D^2\mathcal U(p) \widehat{U} \widehat{\Sigma} \widehat{U}^\top \psi^\top}{\psi \widehat{U} \widehat{\Sigma} \widehat{U}^\top \psi^\top}.
    \]
Denote $\varphi = \psi \widehat{U} \sqrt{\widehat{\Sigma}}$. As both $\widehat{U}\in\mathbb{R}^{n\times (n-1)}$, $\widehat{\Sigma}\in\mathbb{R}^{(n-1) \times (n-1) }$ are matrices with full rank, $\varphi\in\mathbb{R}^{n-1}$ explores the entire space $\mathbb{R}^{n-1}$ as $\psi$ goes over $\mathbb{R}^n$. We can further rewrite it as 
    \begin{equation}\label{eq:lambda-min}
      \min_{\varphi\in\mathbb{R}^{n-1}}  \frac{\varphi \sqrt{\widehat{\Sigma}}\widehat{U}^\top D^2\mathcal U(p) \widehat{U} \sqrt{\widehat{\Sigma}} \varphi^\top}{\varphi \varphi^\top} = \lambda_{\min}(\sqrt{\widehat{\Sigma}}\widehat{U}^\top D^2\mathcal U(p) \widehat{U} \sqrt{\widehat{\Sigma}}).
    \end{equation}

\section{Other proofs}\label{app:proof} 

\begin{proof}[Proof to \Cref{lem:unique}]
Recall 
    \[\frac{\partial \mathrm I(p\|\pi)}{\partial p_i} = \frac{1}{2}  \sum_{j=1, j\neq i}^n \frac{\omega_{ij}}{\pi_i} (\log\frac{p_i}{\pi_i} - \log\frac{p_j}{\pi_j}) + \frac{\omega_{ij}}{p_i}(\frac{p_i}{\pi_i} - \frac{p_j}{\pi_j}). \]
It is straightforward to verify that when $p=\pi$, we have $\nabla_p \mathrm I(p\|\pi)= 0$.

On the other hand, suppose $\nabla_p \mathrm{I}(p\|\pi) = 0$ and assume $\frac{p_{i_1}}{\pi_{i_1}}\geqslant \frac{p_j}{\pi_j}$ for any $j\in V$. We will in the end show that $\frac{p_j}{\pi_j}=\frac{p_{i_1}}{\pi_{i_1}}$, thus $p=\pi$ in $\P(V)$.

For any $i\in V$, denote the neighbors of vertex $i$ as $N(i) = \{j\in V \ |\ \omega_{ij}>0\}$. Note that
  \[
   0=\frac{\partial \mathrm{I}(p\|\pi)}{\partial p_{i_1}} =\frac12 \sum_{j\in N(i_1)} \frac{\omega_{i_1 j}}{\pi_{i_1}} (\log\frac{p_{i_1}}{\pi_{i_1}} - \log\frac{p_j}{\pi_j}) + \frac{\omega_{i_1j}}{p_{i_1}}(\frac{p_{i_1}}{\pi_{i_1}} - \frac{p_j}{\pi_j})\geqslant 0,
   \]
as $\frac{p_{i_1}}{\pi_{i_i}}\geqslant \frac{p_{j}}{\pi_j}$ and $\omega_{i_1 j}>0$. This implies that for any $j\in N(i_1)$, $\frac{p_j}{\pi_j}=\frac{p_{i_1}}{\pi_{i_i}}$.

As the graph $G=(V,E,\omega)$ is connected, for any node $j$, there exists a path $i_1\to i_2 \to \cdots \to i_m=j$ for some $m\leqslant n$, such that $\omega_{i_{k}\, i_{k+1}}>0$ for $k=1,\ldots,m-1$. From the previous step, we have shown that $i_2\in N(i_1)$ implies that $\frac{p_{i_2}}{\pi_{i_2}}=\frac{p_{i_1}}{\pi_{i_i}}$. Thus we can repeat this process for $i_2,\ldots,i_m$, and prove that
    \[\frac{p_{i_1}}{\pi_{i_1}}=\frac{p_{i_2}}{\pi_{i_2}}=\cdots=\frac{p_{j}}{\pi_{j}}.\]
In conclusion, we have shown that \( \frac{p_j}{\pi_j} = \frac{p_{i_1}}{\pi_{i_1}} \) for all \( j \in V \). Since \( p \) is a probability function and must sum to one, it follows that \( p = \pi \).

\end{proof}

\begin{proof}[Proof of \Cref{thm:positivity}]
Fix $r_0$, $F_{r_0}(r)$ is invertible on $(0,r_0]$, and we denote its inverse by $F^{-1}_{r_0}(C)$, whose domain is $[0,\infty)$ and range is $(0,r_0]$. Note that for $r\in (0,r_0]$ and $C\in (0,\infty)$, 
    \begin{equation}\label{eq:Fproperty1}
       F_{r_0}(r)\leqslant C,\qquad \textrm{if and only if}\qquad F_{r_0}^{-1}(C)\leqslant r.
    \end{equation}

Pick $C \in (0,\infty)$. Let us suppose $F^{-1}_{r_0}(C)=r$ for some $r\in (0,r_0)$, that is $F_{r_0}(r)=C$. For any $r_1>r_0$, $F_{r_1}(r)>F_{r_0}(r)=C$ implies that $F_{r_1}^{-1}(C)>r=F_{r_0}^{-1}(C)$. Now for fixed $C$, define $G_C(r)=F_r^{-1}(C)$, thus $G_C(r)$ is a strictly increasing function, satisfying that $G_C(r)>0$ as long as $C>0$ and $r>0$. And 
    \begin{equation}\label{eq:Fproperty2}
        G_C(r_0)=F_{r_0}^{-1}(C)\leqslant r,\qquad\textrm{if and only if}\qquad F_{r_0}(r)\leqslant C.
    \end{equation}

Given that the Hamiltonian $\mathcal{H}(p(t),\psi(t))$ is nonincreasing as shown in \eqref{eq:Decay-Hamiltonian}, and $\mathcal{H}(p(0),\psi(0))$ is bounded, let $r_i=\frac{p_i}{\pi_i}$, for any $t$,  
    \[\mathcal{U}(p(t))=\frac{1}{4}\sum_{i,j=1}^n F_{r_i}(r_j)\omega_{ij}\leqslant \mathcal{H}(p(t),\psi(t))\leqslant C,\]
which implies that $F_{r_i}(r_j)\omega_{ij}\leqslant C$ for any pair $(i,j)$.

Assume $r_{i_1}\geqslant r_i$ for any $i\in V$. Note that there exists an index $i$ such that $p_i\geqslant \frac{1}{n}$. Then
    \[r_{i_1}=\max r_i\geqslant r_{i}=\frac{p_{i}}{\pi_{i}}\geqslant \frac{1}{n \pi_{i}}\geqslant \frac{1}{n\max \pi_i}>0.\]

Since the graph $G$ is connected, for any node $j$, there exists a path $i_1\to i_2\to \cdots \to i_m=j$ for some $m\leqslant  n$, such that $\omega_{i_{k}\, i_{k+1}}>0$ for $k=1,\ldots,m-1$. Let $\bar{\omega}=\displaystyle\min_{k} \omega_{i_k\, i_{k+1}}>0$. For $k=1,\ldots,m-1$, $F_{r_{i_k}}(r_{i_{k+1}})\omega_{i_k\, i_{k+1}}\leqslant C$ implies that $F_{r_{i_k}}(r_{i_{k+1}})\leqslant \frac{C}{\omega_{i_k i_{k+1}}}\leqslant \frac{C}{\bar{\omega}}$.

From node $i_1$ to $i_2$, given that $r_{i_1}\geqslant r_{i_2}$ and $F_{r_{i_1}}(r_{i_2})\leqslant \frac{C}{\bar{\omega}}$, by \eqref{eq:Fproperty1}, \eqref{eq:Fproperty2} and $G_C(r)$ is strictly increasing, we have
    \[r_{i_2}\geqslant F_{r_{i_1}}^{-1}(\frac{C}{\bar{\omega}})=G_{C/ \bar{\omega}}(r_{i_1})\geqslant G_{C/ \bar{\omega}}(\frac{1}{n\max\pi_i})>0.\]

From node $i_2$ to $i_3$, if $r_{i_2}\geqslant r_{i_3}$, we can repeat the above step and yield $r_{i_3}\geqslant G_{C/ \bar{\omega}}(r_{i_2})$; otherwise $r_{i_3}>r_{i_2}\geqslant G_{C / \bar{\omega}}(r_{i_1})\geqslant G_{C / \bar{\omega}}(r_{i_2})$. In either case,
    \[r_{i_3}\geqslant G_{C / \bar{\omega}}(r_{i_2})\geqslant G_{C / \bar{\omega}}(G_{C / \bar{\omega}}(r_{i_1}))\geqslant  G_{C / \bar{\omega}}^{(2)}(\frac{1}{n\max \pi_i})>0.\]
We may repeat this argument and yield that $r_{i_k}>0$ for any node $i_k$ on this path,
    \[r_{i_k}\geqslant G^{(k-1)}_{C / \bar{\omega}}(\frac{1}{n\max\pi_i})>0.\]
Thus $\varepsilon=(\min \pi_i)\cdot  G^{(n-1)}_{C / \bar{\omega}}(\frac{1}{n\max\pi_i})>0$ is the lower bound of $p_i(t)$. 

\end{proof}

\begin{proof}[Proof to \Cref{lem:lambda-strong}]
Follow the notations used in \Cref{app:Hessian}, Recall \eqref{eq:lambda-min} and 
    \[\mathrm{D}^2\mathcal{U}(p)\vert_{p=\pi}=\mathrm{diag}(\pi^{-1})K\mathrm{diag}(\pi^{-1})=\mathrm{diag}(\pi)^{-1} \widehat{U} \widehat{\Sigma} \widehat{U}^\top \mathrm{diag}(\pi)^{-1}.\] 
The minimal positive eigenvalue $\lambda$ at $p=\pi$ is
        \begin{align*}
     \lambda  &  = \lambda_{\min}(\sqrt{\widehat{\Sigma}}\underbrace{\widehat{U}^\top \mathrm{diag}(\pi)^{-1} \widehat{U}}_{S} \widehat{\Sigma} \underbrace{\widehat{U}^\top \mathrm{diag}(\pi)^{-1} \widehat{U}}_{S^\top} \sqrt{\widehat{\Sigma}})  \\
     & = \lambda_{\min}((S^\top 
     \sqrt{\widehat{\Sigma}})^\top \widehat{\Sigma} (S^\top \sqrt{\widehat{\Sigma}})).
    \end{align*}
Since the non-singular conjugate of positive definite matrix is still positive definite, we can show that $\lambda > 0$ as long as $S^\top \sqrt{\widehat{\Sigma}}$ is non-singular. 

In fact, suppose for $x \in\mathbb{R}^{n-1}$, $S^\top\sqrt{\widehat{\Sigma}}x^\top = 0$, thus $(\sqrt{\widehat{\Sigma}} x^\top)^\top S^\top \sqrt{\widehat{\Sigma}} x^\top = 0$. From the change of variable $y^\top = \widehat{U}\sqrt{\widehat{\Sigma}} x^\top \in \mathbb{R}^n,$ we obtain $y \mathrm{diag}(\pi)^{-1} y^\top = 0.$ This leads to $y=0$, i.e., $\widehat{U}\sqrt{\widehat{\Sigma}} x^\top = 0.$ Now since $\widehat{U}\in\mathbb{R}^{n\times (n-1)}$ is full rank, $\sqrt{\widehat{\Sigma}}$ is non-singular, we have $x=0$. As a result, $S^\top \sqrt{\Sigma}$ is non-singular.
\end{proof}

\section{Supplementary numerical results}\label{app:numerical}

\subsection{Comparisons between multi-chain implementations}\label{app:iter}

In this subsection, we report the comparison between multi-chain MH and \texttt{log-Fisher} under a fixed iteration budget, when sampling `rose', `fractal tree' and `checkerboard', as supplementary experiments to the comparison protocols between single-chain MH sampler and \texttt{log-Fisher} in \Cref{sub:image} under a fixed wall-clock time.

The targets are the same as in \Cref{sub:image} and parameters are specified as follows:
    \[\textrm{sampling size~}M=655360,\quad\textrm{default step size~}\Delta t= 0.1\,\quad\textrm{total iterations~}N=2.5\times 10^4.\]
The \texttt{log-Fisher} method is warm-started by nine multi-chain MH iterations (i.e., $t<1$, or 0.04\% of default time span $[0,2500]$). Thereafter we employ a constant damping parameter $\gamma(t)=2\sqrt{\lambda_*}$, where the Rayleigh quotient yields $\lambda_*=1.7\times 10^{-6}$ for `rose' example and $\lambda_*=2.6\times 10^{-6}$ for `fractal tree'. Adaptive time stepping shortens the evolutional time span. For the
`rose' target, the resulting interval is approximately $[0,1279.1]$, and the restart mechanism is activated $1478$ times, corresponding to $5.91\%$ of the iterations. For the `fractal tree' target, the corresponding interval is approximately
$[0,1302.8]$, with $1437$ restarts, or $5.75\%$ of the iterations. Notably, the effective time spans of the \texttt{Fisher-log} method for both datasets are approximately half that of MH, suggesting that the proposed damped Hamiltonian dynamics driven by the Fisher information can even outperform the MH method for a fixed evolutional dynamics time.

\begin{figure}[hp!tb]
     \centering
     \begin{subfigure}{0.85\textwidth}
         \centering
         \includegraphics[width=\textwidth]{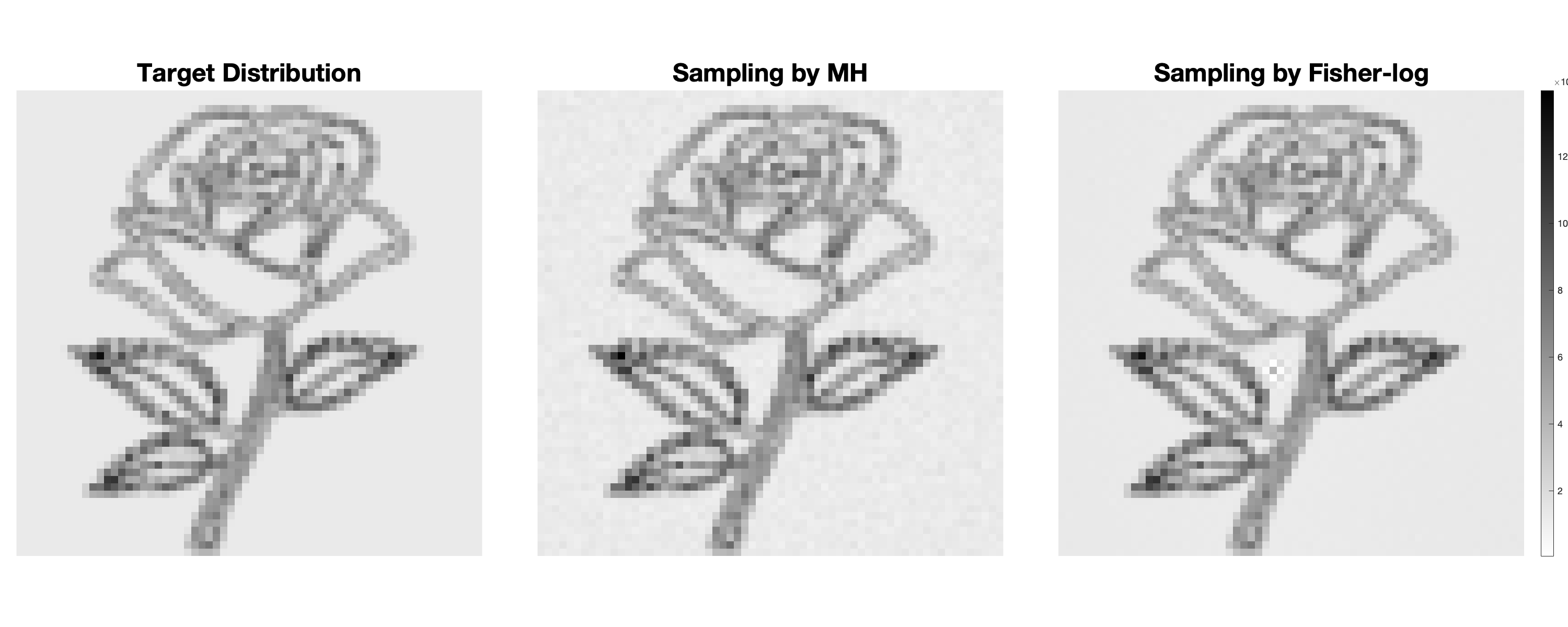}
     \end{subfigure}
     \hfill
     \begin{subfigure}[b]{0.45\textwidth}
         \centering
         \includegraphics[width=\textwidth]{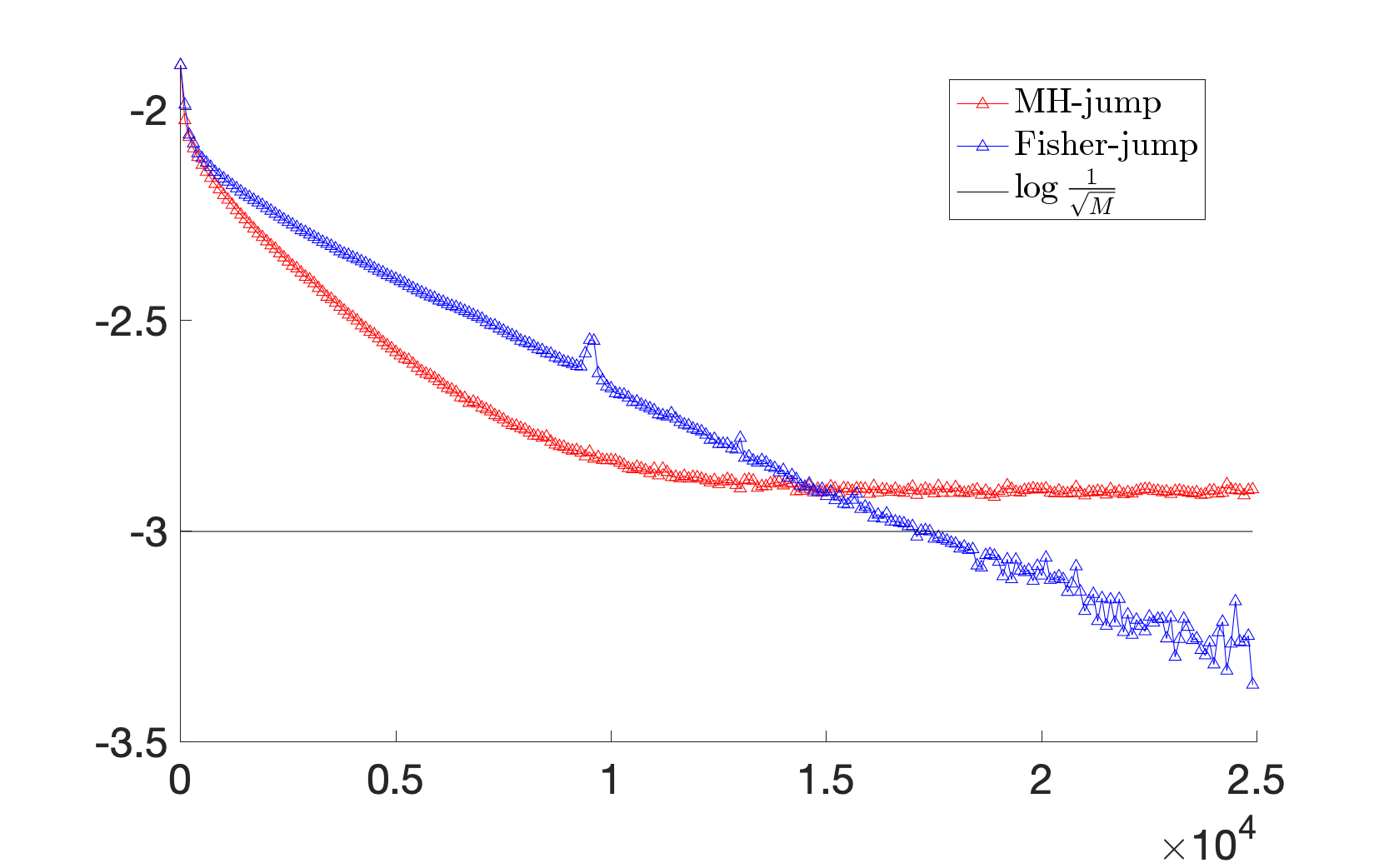}
     \end{subfigure}
     \hfill
     \begin{subfigure}[b]{0.45\textwidth}
         \centering
         \includegraphics[width=\textwidth]{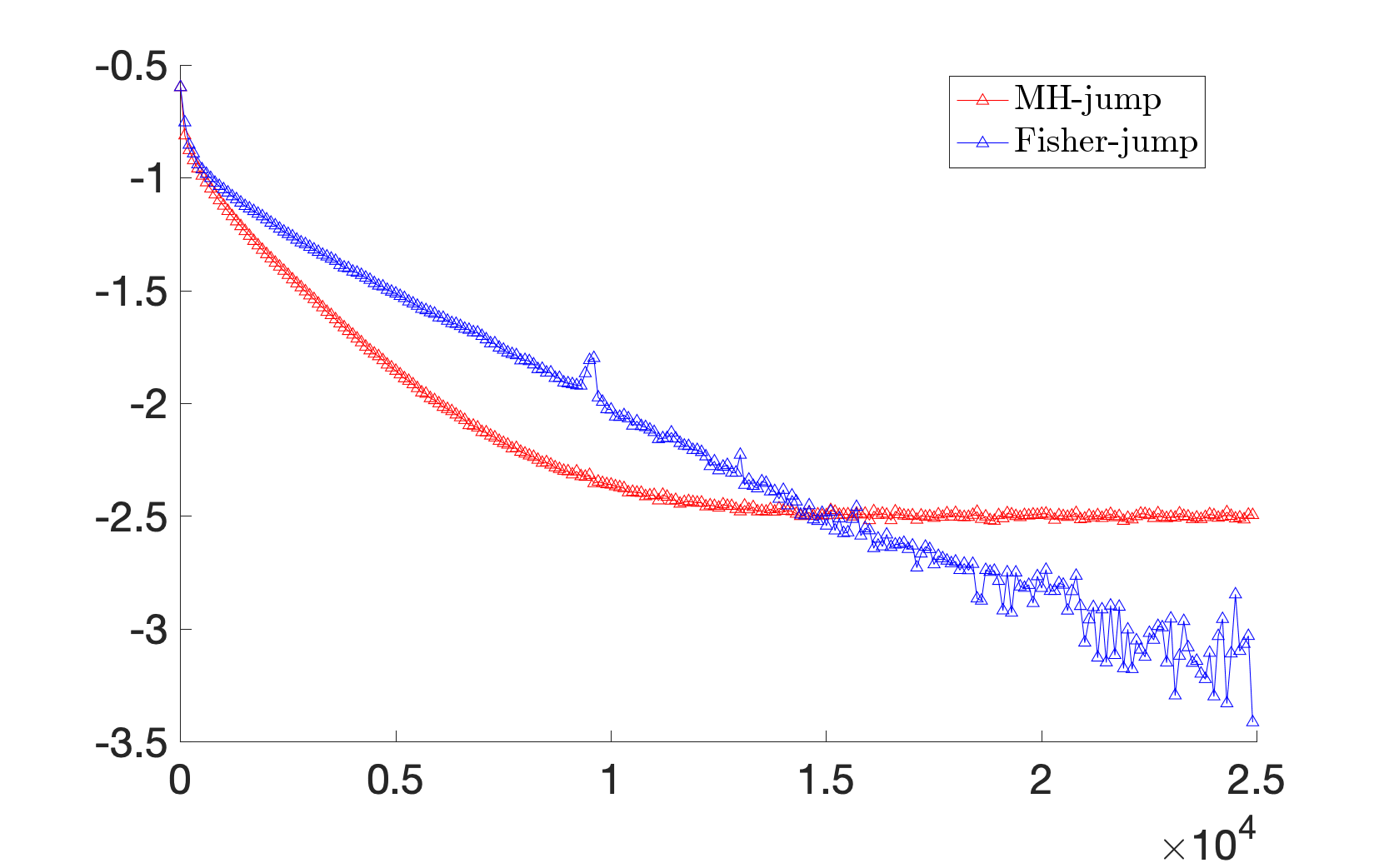}
     \end{subfigure}
     \hfill
     \begin{subfigure}{0.85\textwidth}
         \centering
         \includegraphics[width=\textwidth]{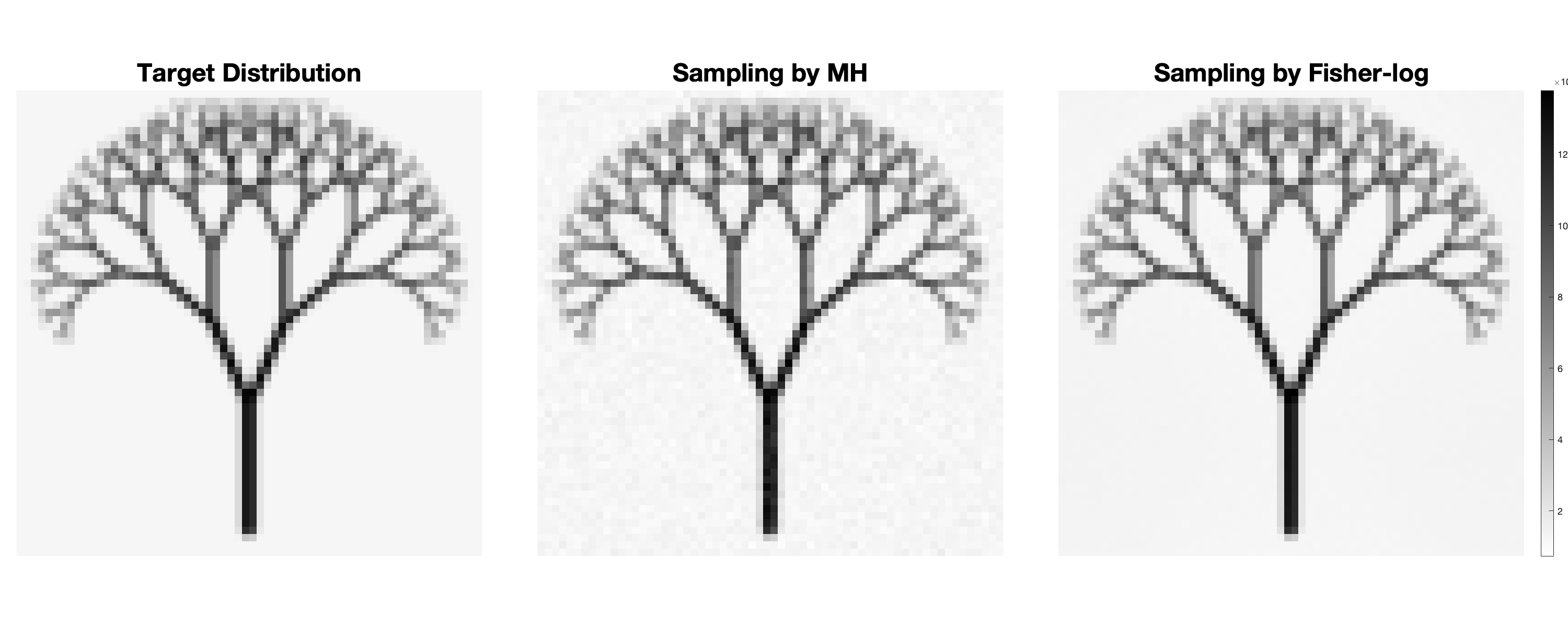}
     \end{subfigure}
     \hfill
    \begin{subfigure}[b]{0.45\textwidth}
         \centering
         \includegraphics[width=\textwidth]{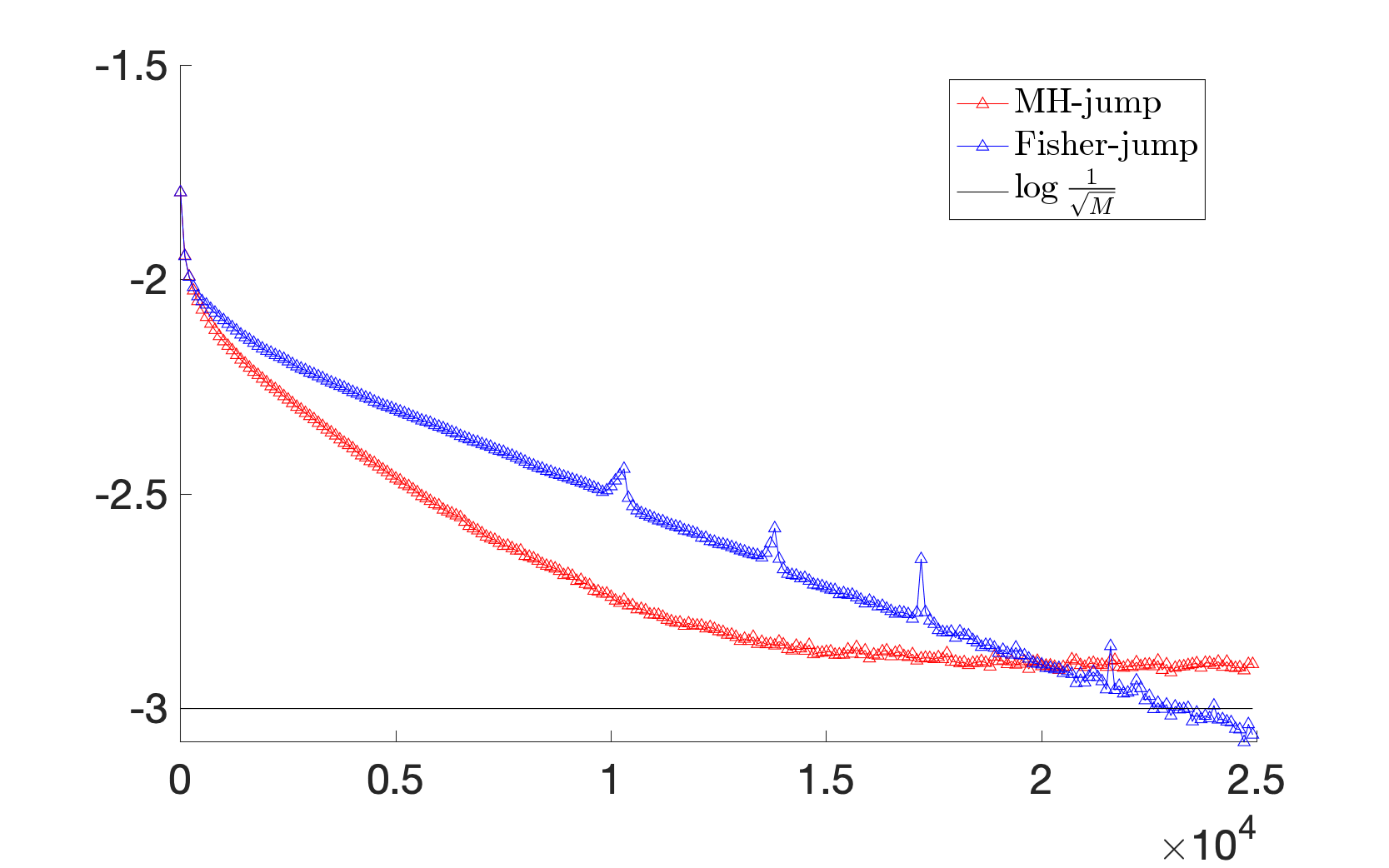}
     \end{subfigure}
     \hfill
     \begin{subfigure}[b]{0.45\textwidth}
         \centering
         \includegraphics[width=\textwidth]{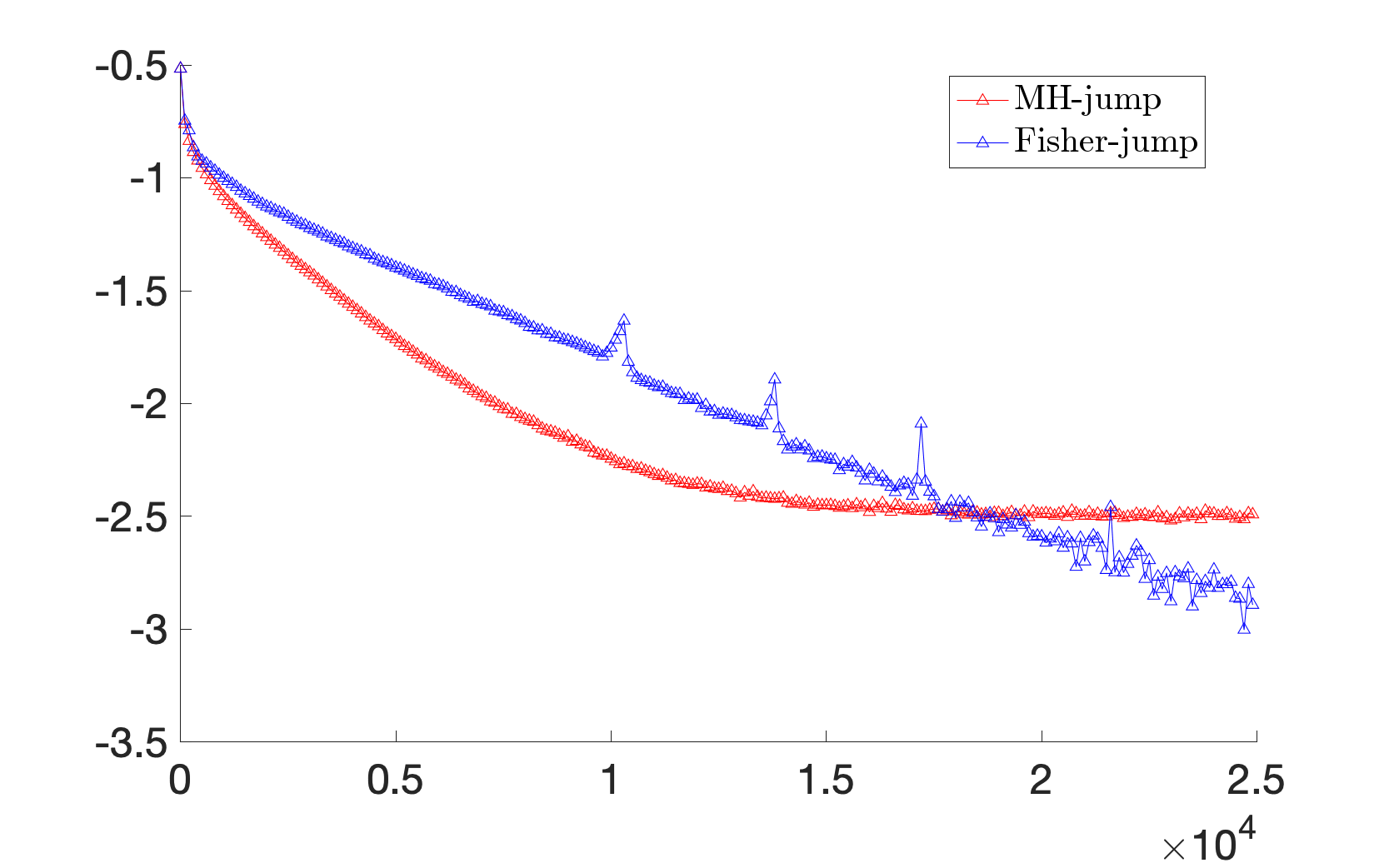}
     \end{subfigure}
     \caption{Sampling real datasets (`rose' and `fractal tree') on a $64\times 64$ lattice. The accompanying trajectories follow the same convention: x-axes are in iterations; the left panel in on $\log_{10}\norm{p(t)-\pi}_2$; the right panel displays $\log_{10}\abs{\sum_{i=1}^n p_i(t)\log \frac{p_i(t)}{Z\pi_i } - (-\log Z)}$. \label{fig:pixel-performance}}
\end{figure}

\begin{figure}[hp!tb]
    \centering
    \includegraphics[width=\linewidth]{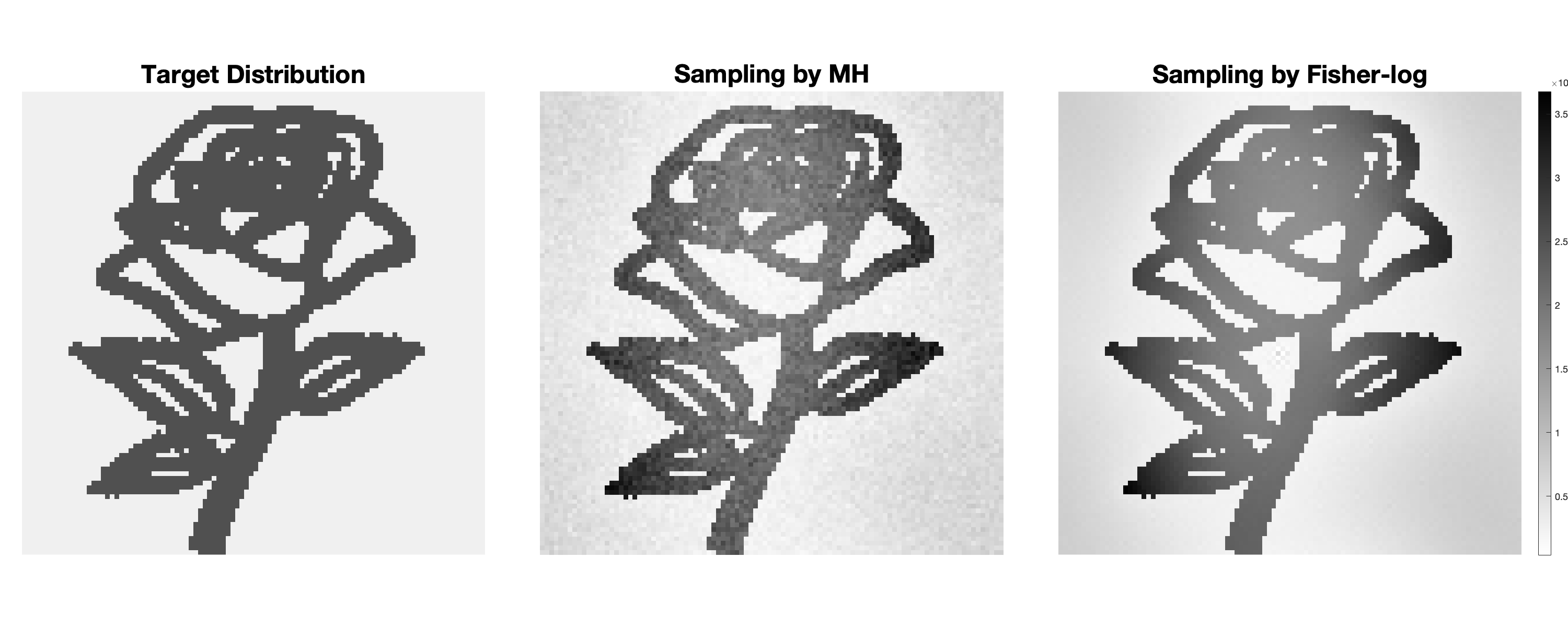}
    \caption{Sampling a piecewise-constant `rose' on a $100\times 100$ lattice. Across iterations, the empirical distributions generated by MH exhibits stronger cell-wise fluctuations than those by the \texttt{log-Fisher} method.}
    \label{fig:smooth-sampler}
\end{figure}

\Cref{fig:pixel-performance} compares the approximation errors for both the target distribution and the normalizing constant, under a fixed number of iterations. In additional to the quantitative error trajectories, a notable qualitative difference in the histograms, as shown in \Cref{fig:pixel-performance}, is that the empirical distribution generated by our method is smoother.

To make this effect more visible, we next test on a larger
$100\times100$ lattice. The target distribution is constructed by assigning mass $p_i\in \left\{\frac{1}{10},1\right\}$ to each node. Parameters are:
    \[\textrm{sampling size~}M=10^6,\quad\textrm{default step size~}\Delta t= 0.1\,\quad\textrm{total iterations~}N=3\times 10^4.\]
We again employ a constant damping parameter $\gamma(t)=2\sqrt{\lambda_*}$. Owing to the large state space ($10^4$ nodes) and the multi-chain nature of our proposed jump process, these experiments are carried out on the UCSB computing cluster. A notable qualitative difference between the \texttt{log-Fisher} method and MH sampling is observed in \Cref{fig:smooth-sampler}: throughout the jump process, the empirical distributions produced by the \texttt{log-Fisher} method appear significantly smoother, whereas those generated by MH exhibit a mosaic effect. We conjecture that this behavior is related to the acceptance–rejection matrix \eqref{eq:new-accept} inherent in the \texttt{log-Fisher} method,
    \begin{equation*}
    (\bar{A}_{\psi})_{ij}=\max\left\{\dfrac{1-\frac{\pi_i}{\pi_j}\frac{p_j}{p_i}}{\log\left(\frac{\pi_j}{\pi_i}\frac{p_i}{p_j}\right)}(\psi_j-\psi_i),0 \right\} A_{ij},
    \end{equation*}
which compares the local ratios $\frac{p_i}{\pi_i}$ and $\frac{p_j}{\pi_j}$ across neighboring nodes along edges. This local feedback is consistent with the more spatially coherent empirical distributions observed in the figures.

{\subsection{GPU acceleration on the Ising model}\label{app:gpu}

We implement our algorithm with a switchable array backend (NumPy on the CPU and CuPy on the GPU), so that the mean-field particle update runs unchanged on either device, and ran a simple timing test on two fixed two-dimensional Ising targets while varying the sample size (Table~\ref{tab:amh-cpu-gpu}). On the larger target ($4\times4$ sites, $n=2^{16}$ states), the GPU accelerates each iteration substantially, and the speedup grows as the number of particles increases. On the smaller target ($3\times3$ sites, $n=2^{9}$), by contrast, the per-iteration GPU time is essentially constant across the range of sample sizes, indicating that the cost is dominated by fixed per-iteration overhead---such as kernel launches, host--device synchronization, and sparse-operator setup---rather than by the sampling work itself; consequently, the GPU is only on par with the CPU here and
surpasses it only at the largest sample sizes. Overall, the GPU implementation is most beneficial for large state spaces and large particle counts, where the parallel workload amortizes this fixed overhead, whereas for small problems the CPU remains equally efficient. We leave the study of GPU acceleration for
large-scale sampling problems, such as larger Ising models, to future work.

\begin{table}[ht!]
\centering
\caption{The \texttt{log-Fisher} per-iteration wall-clock time (ms) for CPU and GPU implementations on two fixed Ising
models. Particle size is $ kn$ particles, where $n$ is the number of nodes in the hypercube and
$k=10\cdot2^m$, $m=1,\dots,6$. The reported speedup is the ratio of CPU
time to GPU time (values $<1$ means GPU is slower).
}
\label{tab:amh-cpu-gpu}
\begin{tabular}{@{}ll rrrrrr@{}}
\toprule
Model & Metric & \multicolumn{6}{c}{particle size} \\
\cmidrule(lr){3-8}
 & & $20n$ & $40n$ & $80n$ & $160n$ & $320n$ & $640n$ \\
\midrule
\multirow{3}{*}{\shortstack[l]{$3\times3$\\ $n{=}512$}}
 & CPU (ms)          & $4.8$   & $10.4$  & $7.1$   & $6.8$   & $15.1$  & $36.9$  \\
 & GPU (ms)          & $11.31$ & $11.35$ & $11.43$ & $11.57$ & $12.30$ & $12.19$ \\
 & speedup ($\times$)& $0.4$   & $0.9$   & $0.6$   & $0.6$   & $1.2$   & $3.0$   \\
\midrule
\multirow{3}{*}{\shortstack[l]{$4\times4$\\ $n{=}65536$}}
 & CPU (ms)          & $315.7$ & $384.1$ & $562.1$ & $933.1$ & $1705.6$ & $3265.1$ \\
 & GPU (ms)          & $18.34$ & $18.94$ & $20.81$ & $27.17$ & $39.22$  & $62.85$  \\
 & speedup ($\times$)& $17.2$  & $20.3$  & $27.0$  & $34.3$  & $43.5$   & $52.0$   \\
\bottomrule
\end{tabular}
\end{table}
}

\vskip 0.2in
\bibliography{AMCMC.bib}

@article{Metropolis1953Equations,
    author = {Metropolis, N. and Rosenbluth, A. W. and Rosenbluth, M. N. and Teller, A. H. and Teller, E.},
    title = {{Equation of State Calculations by Fast Computing Machines}},
    journal = {The Journal of Chemical Physics},
    volume = {21},
    number = {6},
    pages = {1087-1092},
    year = {1953},
    month = {06},
    issn = {0021-9606},
    doi = {10.1063/1.1699114},
}

@article {Hastings1970MC,
    AUTHOR = {Hastings, W. K.},
     TITLE = {Monte {C}arlo sampling methods using {M}arkov chains and their
              applications},
   JOURNAL = {Biometrika},
  FJOURNAL = {Biometrika},
    VOLUME = {57},
      YEAR = {1970},
    NUMBER = {1},
     PAGES = {97--109},
      ISSN = {0006-3444,1464-3510},
   MRCLASS = {65C05 (11K45 60J22)},
  MRNUMBER = {3363437},
       DOI = {10.1093/biomet/57.1.97},
}

@article {Schnakenberg1976Network,
    AUTHOR = {Schnakenberg, J.},
     TITLE = {Network theory of microscopic and macroscopic behavior of
              master equation systems},
   JOURNAL = {Rev. Modern Phys.},
  FJOURNAL = {Reviews of Modern Physics},
    VOLUME = {48},
      YEAR = {1976},
    NUMBER = {4},
     PAGES = {571--585},
      ISSN = {0034-6861,1539-0756},
   MRCLASS = {82.60},
  MRNUMBER = {443796},
       DOI = {10.1103/RevModPhys.48.571},
}

@article{Gelman1992Multiple,
author = {Andrew Gelman and Donald B. Rubin},
title = {{Inference from Iterative Simulation Using Multiple Sequences}},
volume = {7},
journal = {Statistical Science},
number = {4},
publisher = {Institute of Mathematical Statistics},
pages = {457 -- 472},
year = {1992},
doi = {10.1214/ss/1177011136},
}

@article {Roberts1996MALA1,
    AUTHOR = {Roberts, G. O. and Tweedie, R. L.},
     TITLE = {Geometric convergence and central limit theorems for
              multidimensional {H}astings and {M}etropolis algorithms},
   JOURNAL = {Biometrika},
  FJOURNAL = {Biometrika},
    VOLUME = {83},
      YEAR = {1996},
    NUMBER = {1},
     PAGES = {95--110},
      ISSN = {0006-3444},
   MRCLASS = {60J10 (60F05 62F15 65U05)},
  MRNUMBER = {1399158},
MRREVIEWER = {Arnoldo Frigessi},
       DOI = {10.1093/biomet/83.1.95},
}

@article {Roberts1996MALA2,
    AUTHOR = {Roberts, Gareth O. and Tweedie, Richard L.},
     TITLE = {Exponential convergence of {L}angevin distributions and their
              discrete approximations},
   JOURNAL = {Bernoulli},
  FJOURNAL = {Bernoulli. Official Journal of the Bernoulli Society for
              Mathematical Statistics and Probability},
    VOLUME = {2},
      YEAR = {1996},
    NUMBER = {4},
     PAGES = {341--363},
      ISSN = {1350-7265},
   MRCLASS = {62E25 (65C05)},
  MRNUMBER = {1440273},
MRREVIEWER = {Arnoldo\ Frigessi},
       DOI = {10.2307/3318418},
}

@article {Mengersen1996Rates,
    AUTHOR = {Mengersen, K. L. and Tweedie, R. L.},
     TITLE = {Rates of convergence of the {H}astings and {M}etropolis
              algorithms},
   JOURNAL = {Ann. Statist.},
  FJOURNAL = {The Annals of Statistics},
    VOLUME = {24},
      YEAR = {1996},
    NUMBER = {1},
     PAGES = {101--121},
      ISSN = {0090-5364,2168-8966},
   MRCLASS = {60J10 (65U05)},
  MRNUMBER = {1389882},
MRREVIEWER = {Arnoldo\ Frigessi},
       DOI = {10.1214/aos/1033066201},
}

@article {Benamou2000Computational,
    AUTHOR = {Benamou, Jean-David and Brenier, Yann},
     TITLE = {A computational fluid mechanics solution to the
              {M}onge-{K}antorovich mass transfer problem},
   JOURNAL = {Numer. Math.},
  FJOURNAL = {Numerische Mathematik},
    VOLUME = {84},
      YEAR = {2000},
    NUMBER = {3},
     PAGES = {375--393},
      ISSN = {0029-599X,0945-3245},
   MRCLASS = {65M60 (35Q35 76M30)},
  MRNUMBER = {1738163},
MRREVIEWER = {Enrique\ Fern\'andez Cara},
       DOI = {10.1007/s002110050002},
}

@article {Otto2001Geometry,
    AUTHOR = {Otto, Felix},
     TITLE = {The geometry of dissipative evolution equations: the porous
              medium equation},
   JOURNAL = {Comm. Partial Differential Equations},
  FJOURNAL = {Communications in Partial Differential Equations},
    VOLUME = {26},
      YEAR = {2001},
    NUMBER = {1-2},
     PAGES = {101--174},
      ISSN = {0360-5302,1532-4133},
   MRCLASS = {35K57 (35K65 76S05)},
  MRNUMBER = {1842429},
MRREVIEWER = {Antonio\ Fasano},
       DOI = {10.1081/PDE-100002243},
}

@article {Jacquier2002Bayesian,
    AUTHOR = {Jacquier, E. and Polson, N. G. and Rossi, P. E.},
     TITLE = {Bayesian analysis of stochastic volatility models with
              fat-tails and correlated errors},
   JOURNAL = {J. Econometrics},
  FJOURNAL = {Journal of Econometrics},
    VOLUME = {122},
      YEAR = {2004},
    NUMBER = {1},
     PAGES = {185--212},
      ISSN = {0304-4076,1872-6895},
   MRCLASS = {62P05 (62M10)},
  MRNUMBER = {2083256},
       DOI = {10.1016/j.jeconom.2003.09.001},
}

@article {Teh2004Energy,
    AUTHOR = {Teh, Y. W. and Welling, M. and Osindero, S. and Hinton,
              G. E.},
     TITLE = {Energy-based models for sparse overcomplete representations},
   JOURNAL = {Journal of Machine Learning Research},
    VOLUME = {4},
      YEAR = {2004},
    NUMBER = {7-8},
     PAGES = {1235--1260},
      ISSN = {1532-4435,1533-7928},
   MRCLASS = {68T05 (62G07 62H25 68T10)},
  MRNUMBER = {2103628},
       DOI = {10.1162/jmlr.2003.4.7-8.1235},
}

@book {Robert2004Monte,
    AUTHOR = {Robert, C. P. and Casella, G.},
     TITLE = {Monte {C}arlo statistical methods},
    SERIES = {Springer Texts in Statistics},
   EDITION = {Second},
 PUBLISHER = {Springer-Verlag, New York},
      YEAR = {2004},
     PAGES = {xxx+645},
      ISBN = {0-387-21239-6},
   MRCLASS = {62-01 (60J10 62F15 65Cxx)},
  MRNUMBER = {2080278},
MRREVIEWER = {Petru\ P.\ Blaga},
       DOI = {10.1007/978-1-4757-4145-2},
}

@article {Hinton2006Fast,
    AUTHOR = {Hinton, Geoffrey E. and Osindero, Simon and Teh, Yee-Whye},
     TITLE = {A fast learning algorithm for deep belief nets},
   JOURNAL = {Neural Comput.},
  FJOURNAL = {Neural Computation},
    VOLUME = {18},
      YEAR = {2006},
    NUMBER = {7},
     PAGES = {1527--1554},
      ISSN = {0899-7667,1530-888X},
   MRCLASS = {68T05 (92C20)},
  MRNUMBER = {2224485},
       DOI = {10.1162/neco.2006.18.7.1527},
}

@article {Jasra2007Population,
    AUTHOR = {Jasra, Ajay and Stephens, David A. and Holmes, Christopher C.},
     TITLE = {Population-based reversible jump {M}arkov chain {M}onte
              {C}arlo},
   JOURNAL = {Biometrika},
  FJOURNAL = {Biometrika},
    VOLUME = {94},
      YEAR = {2007},
    NUMBER = {4},
     PAGES = {787--807},
      ISSN = {0006-3444,1464-3510},
   MRCLASS = {99-01},
  MRNUMBER = {2416793},
       DOI = {10.1093/biomet/asm069},
}

@book {Villani2009OT,
    AUTHOR = {Villani, C\'edric},
     TITLE = {Optimal transport: Old and new},
    VOLUME = {338},
 PUBLISHER = {Springer-Verlag, Berlin},
      YEAR = {2009},
     PAGES = {xxii+973},
      ISBN = {978-3-540-71049-3},
   MRCLASS = {49-02 (28A75 37J50 49Q20 53C23 58E30)},
  MRNUMBER = {2459454},
MRREVIEWER = {Dario\ Cordero-Erausquin},
       DOI = {10.1007/978-3-540-71050-9},
}

@book {Levin2009Markov,
    AUTHOR = {Levin, David A. and Peres, Yuval and Wilmer, Elizabeth L.},
     TITLE = {Markov chains and mixing times},
 PUBLISHER = {American Mathematical Society, Providence, RI},
      YEAR = {2009},
     PAGES = {xviii+371},
      ISBN = {978-0-8218-4739-8},
   MRCLASS = {60J10 (60-01 60J05 60K35 60K37 68U20 68W20)},
  MRNUMBER = {2466937},
MRREVIEWER = {Olle\ H\"aggstr\"om},
       DOI = {10.1090/mbk/058},
}

@article{Archer2009Dynamical,
    author = {Archer, A. J.},
    title = {Dynamical density functional theory for molecular and colloidal fluids: A microscopic approach to fluid mechanics},
    journal = {The Journal of Chemical Physics},
    volume = {130},
    number = {1},
    pages = {014509},
    year = {2009},
    month = {01},
    issn = {0021-9606},
    doi = {10.1063/1.3054633},
}

@article {Goodman2010Ensemble,
    AUTHOR = {Goodman, Jonathan and Weare, Jonathan},
     TITLE = {Ensemble samplers with affine invariance},
   JOURNAL = {Commun. Appl. Math. Comput. Sci.},
  FJOURNAL = {Communications in Applied Mathematics and Computational
              Science},
    VOLUME = {5},
      YEAR = {2010},
    NUMBER = {1},
     PAGES = {65--80},
      ISSN = {1559-3940,2157-5452},
   MRCLASS = {65C05 (60J20)},
  MRNUMBER = {2600822},
MRREVIEWER = {Peter\ Math\'e},
       DOI = {10.2140/camcos.2010.5.65},
}

@article {Maas2011gradient,
    AUTHOR = {Maas, Jan},
     TITLE = {Gradient flows of the entropy for finite {M}arkov chains},
   JOURNAL = {J. Funct. Anal.},
  FJOURNAL = {Journal of Functional Analysis},
    VOLUME = {261},
      YEAR = {2011},
    NUMBER = {8},
     PAGES = {2250--2292},
      ISSN = {0022-1236,1096-0783},
   MRCLASS = {49Q20 (49J10 60J27)},
  MRNUMBER = {2824578},
MRREVIEWER = {Nung\ Kwan\ Yip},
       DOI = {10.1016/j.jfa.2011.06.009},
}

@article{Jacob2011Parallel,
author = {P. Jacob and C. P. Robert and M. H. Smith},
title = {Using Parallel Computation to Improve Independent Metropolis–Hastings Based Estimation},
journal = {Journal of Computational and Graphical Statistics},
volume = {20},
number = {3},
pages = {616--635},
year = {2011},
publisher = {Taylor \& Francis},
doi = {10.1198/jcgs.2011.10167},
}

@article {Erbar2012Ricci,
    AUTHOR = {Erbar, Matthias and Maas, Jan},
     TITLE = {Ricci curvature of finite {M}arkov chains via convexity of the
              entropy},
   JOURNAL = {Arch. Ration. Mech. Anal.},
  FJOURNAL = {Archive for Rational Mechanics and Analysis},
    VOLUME = {206},
      YEAR = {2012},
    NUMBER = {3},
     PAGES = {997--1038},
      ISSN = {0003-9527,1432-0673},
   MRCLASS = {58J65 (49Q20 60B05 60J05)},
  MRNUMBER = {2989449},
MRREVIEWER = {Pedro\ J.\ Catuogno},
       DOI = {10.1007/s00205-012-0554-z},
}

@article {Chow2012Fokker,
    AUTHOR = {Chow, Shui-Nee and Huang, Wen and Li, Yao and Zhou, Haomin},
     TITLE = {Fokker-{P}lanck equations for a free energy functional or
              {M}arkov process on a graph},
   JOURNAL = {Arch. Ration. Mech. Anal.},
  FJOURNAL = {Archive for Rational Mechanics and Analysis},
    VOLUME = {203},
      YEAR = {2012},
    NUMBER = {3},
     PAGES = {969--1008},
      ISSN = {0003-9527,1432-0673},
   MRCLASS = {35R02 (60H30 60J25)},
  MRNUMBER = {2928139},
MRREVIEWER = {Sergey\ Dashkovskiy},
       DOI = {10.1007/s00205-011-0471-6},
}

@book {Nesterov2013Introductory,
    AUTHOR = {Nesterov, Yurii},
     TITLE = {Introductory lectures on convex optimization: A basic course},
    SERIES = {Applied Optimization},
    VOLUME = {87},
 PUBLISHER = {Kluwer Academic Publishers, Boston, MA},
      YEAR = {2004},
     PAGES = {xviii+236},
      ISBN = {1-4020-7553-7},
   MRCLASS = {90-02 (90-01 90C25)},
  MRNUMBER = {2142598},
       DOI = {10.1007/978-1-4419-8853-9},
}

@article {Mielke2013Geodesic,
    AUTHOR = {Mielke, Alexander},
     TITLE = {Geodesic convexity of the relative entropy in reversible
              {M}arkov chains},
   JOURNAL = {Calc. Var. Partial Differential Equations},
  FJOURNAL = {Calculus of Variations and Partial Differential Equations},
    VOLUME = {48},
      YEAR = {2013},
    NUMBER = {1-2},
     PAGES = {1--31},
      ISSN = {0944-2669,1432-0835},
   MRCLASS = {60J27 (53C21 53C23 82B35)},
  MRNUMBER = {3090532},
MRREVIEWER = {Roman\ Urban},
       DOI = {10.1007/s00526-012-0538-8},
}

@article{Craiu2009Neighbor,
author = {Radu V. Craiu and Jeffrey Rosenthal and Chao Yang},
title = {Learn From Thy Neighbor: Parallel-Chain and Regional Adaptive MCMC},
journal = {Journal of the American Statistical Association},
volume = {104},
number = {488},
pages = {1454--1466},
year = {2009},
publisher = {Taylor \& Francis},
doi = {10.1198/jasa.2009.tm08393},
}

@article{Calderhead2014Parallel,
author = {Ben Calderhead },
title = {A general construction for parallelizing Metropolis-Hastings algorithms},
journal = {Proceedings of the National Academy of Sciences},
volume = {111},
number = {49},
pages = {17408-17413},
year = {2014},
doi = {10.1073/pnas.1408184111},
}

@book {Landau2015Guide,
    AUTHOR = {Landau, D. P. and Binder, K.},
     TITLE = {A guide to {M}onte {C}arlo simulations in statistical physics},
   EDITION = {Fourth},
 PUBLISHER = {Cambridge University Press, Cambridge},
      YEAR = {2015},
     PAGES = {xvii+519},
      ISBN = {978-1-107-07402-6},
   MRCLASS = {82B80 (65C05 82-01)},
  MRNUMBER = {3443799},
       DOI = {10.1017/CBO9781139696463},
}

@article {Ottobre2016MCMC,
    AUTHOR = {Ottobre, Michela},
     TITLE = {Markov chain {M}onte {C}arlo and irreversibility},
   JOURNAL = {Rep. Math. Phys.},
  FJOURNAL = {Reports on Mathematical Physics},
    VOLUME = {77},
      YEAR = {2016},
    NUMBER = {3},
     PAGES = {267--292},
      ISSN = {0034-4877,1879-0674},
   MRCLASS = {60J22 (60J60 65C05 65C30)},
  MRNUMBER = {3517569},
MRREVIEWER = {Vigirdas\ Mackevi\v cius},
       DOI = {10.1016/S0034-4877(16)30031-3},
}

@article {Erbar2016Gradient,
    AUTHOR = {Erbar, Matthias and Fathi, Max and Laschos, Vaios and
              Schlichting, Andr\'e},
     TITLE = {Gradient flow structure for {M}c{K}ean-{V}lasov equations on
              discrete spaces},
   JOURNAL = {Discrete Contin. Dyn. Syst.},
  FJOURNAL = {Discrete and Continuous Dynamical Systems},
    VOLUME = {36},
      YEAR = {2016},
    NUMBER = {12},
     PAGES = {6799--6833},
      ISSN = {1078-0947,1553-5231},
   MRCLASS = {60J27 (34A34 82C22)},
  MRNUMBER = {3567821},
       DOI = {10.3934/dcds.2016096},
       URL = {https://doi.org/10.3934/dcds.2016096},
}

@article {Ma2019Irreversible,
    AUTHOR = {Ma, Yi-An and Fox, Emily B. and Chen, Tianqi and Wu, Lei},
     TITLE = {Irreversible samplers from jump and continuous {M}arkov
              processes},
   JOURNAL = {Stat. Comput.},
  FJOURNAL = {Statistics and Computing},
    VOLUME = {29},
      YEAR = {2019},
    NUMBER = {1},
     PAGES = {177--202},
      ISSN = {0960-3174,1573-1375},
   MRCLASS = {99-01},
  MRNUMBER = {3905547},
       DOI = {10.1007/s11222-018-9802-x},
}

@article {Karatzas2021Trajectorial,
    AUTHOR = {Karatzas, Ioannis and Maas, Jan and Schachermayer, Walter},
     TITLE = {Trajectorial dissipation and gradient flow for the relative
              entropy in {M}arkov chains},
   JOURNAL = {Commun. Inf. Syst.},
  FJOURNAL = {Communications in Information and Systems},
    VOLUME = {21},
      YEAR = {2021},
    NUMBER = {4},
     PAGES = {481--536},
      ISSN = {1526-7555,2163-4548},
   MRCLASS = {60J27 (60G44 60H30 60J35 94A17)},
  MRNUMBER = {4273512},
       DOI = {10.4310/CIS.2021.v21.n4.a1},
}

@article{Ma2021Nesterov,
author = {Yi-An Ma and Niladri S. Chatterji and Xiang Cheng and Nicolas Flammarion and Peter L. Bartlett and Michael I. Jordan},
title = {{Is there an analog of Nesterov acceleration for gradient-based MCMC?}},
volume = {27},
journal = {Bernoulli},
number = {3},
publisher = {Bernoulli Society for Mathematical Statistics and Probability},
pages = {1942 -- 1992},
year = {2021},
doi = {10.3150/20-BEJ1297},
}

@article {Wang2022Accelerated,
    AUTHOR = {Wang, Yifei and Li, Wuchen},
     TITLE = {Accelerated information gradient flow},
   JOURNAL = {J. Sci. Comput.},
  FJOURNAL = {Journal of Scientific Computing},
    VOLUME = {90},
      YEAR = {2022},
    NUMBER = {1},
     PAGES = {Paper No. 11, 47},
      ISSN = {0885-7474,1573-7691},
   MRCLASS = {65C05 (62F15 65M32)},
  MRNUMBER = {4342674},
MRREVIEWER = {Pavel\ T.\ Stoynov},
       DOI = {10.1007/s10915-021-01709-3},
}

@article {Albritton2024Variational,
    AUTHOR = {Albritton, Dallas and Armstrong, Scott and Mourrat,
              Jean-Christophe and Novack, Matthew},
     TITLE = {Variational methods for the kinetic {F}okker-{P}lanck
              equation},
   JOURNAL = {Anal. PDE},
  FJOURNAL = {Analysis \& PDE},
    VOLUME = {17},
      YEAR = {2024},
    NUMBER = {6},
     PAGES = {1953--2010},
      ISSN = {2157-5045,1948-206X},
   MRCLASS = {35H10 (35A15 35D30 35K70 35Q84)},
  MRNUMBER = {4776290},
MRREVIEWER = {Felix\ X.-F.\ Ye},
       DOI = {10.2140/apde.2024.17.1953},
}

@article {Cao2023Convergence,
    AUTHOR = {Cao, Yu and Lu, Jianfeng and Wang, Lihan},
     TITLE = {On explicit {$L^2$}-convergence rate estimate for underdamped
              {L}angevin dynamics},
   JOURNAL = {Arch. Ration. Mech. Anal.},
  FJOURNAL = {Archive for Rational Mechanics and Analysis},
    VOLUME = {247},
      YEAR = {2023},
    NUMBER = {5},
     PAGES = {Paper No. 90, 34},
      ISSN = {0003-9527,1432-0673},
   MRCLASS = {60H10 (35Q82 82C31)},
  MRNUMBER = {4632836},
MRREVIEWER = {John\ Masson\ Noble},
       DOI = {10.1007/s00205-023-01922-4},
}

@misc{brigati2024explicit,
      title={{Explicit convergence rates of underdamped Langevin dynamics under weighted and weak Poincar\'e--Lions inequalities}}, 
      author={Giovanni Brigati and Gabriel Stoltz and Andi Q. Wang and Lihan Wang},
      year={2025},
      eprint={2407.16033},
      archivePrefix={arXiv},
      primaryClass={math.PR},
      url={https://arxiv.org/abs/2407.16033}, 
}

@inproceedings{Xu2023Normalizing,
 author = {Xu, Chen and Cheng, Xiuyuan and Xie, Yao},
 booktitle = {Advances in Neural Information Processing Systems},
 editor = {A. Oh and T. Naumann and A. Globerson and K. Saenko and M. Hardt and S. Levine},
 pages = {47379--47405},
 publisher = {Curran Associates, Inc.},
 title = {Normalizing flow neural networks by JKO scheme},
 url = {https://proceedings.neurips.cc/paper_files/paper/2023/file/93fce71def4e3cf418918805455d436f-Paper-Conference.pdf},
 volume = {36},
 year = {2023}
}

@article {Gao2024Master,
    AUTHOR = {Gao, Yuan and Liu, Jian-Guo and Li, Wuchen},
     TITLE = {Master equations for finite state mean field games with
              nonlinear activations},
   JOURNAL = {Discrete Contin. Dyn. Syst. Ser. B},
  FJOURNAL = {Discrete and Continuous Dynamical Systems. Series B. A Journal
              Bridging Mathematics and Sciences},
    VOLUME = {29},
      YEAR = {2024},
    NUMBER = {7},
     PAGES = {2837--2879},
      ISSN = {1531-3492,1553-524X},
   MRCLASS = {49N80 (49L20 60J20 90C35 91A16)},
  MRNUMBER = {4734366},
       DOI = {10.3934/dcdsb.2023204},
}

@article {Nesterov1983Accelerating,
    AUTHOR = {Nesterov, Yu. E.},
     TITLE = {A method for solving the convex programming problem with convergence rate {$O(1/k\sp{2})$}.},
   JOURNAL = {Dokl. Akad. Nauk SSSR},
  FJOURNAL = {Doklady Akademii Nauk SSSR},
    VOLUME = {269},
      YEAR = {1983},
    NUMBER = {no. 3,},
     PAGES = {543--547},
      ISSN = {0002-3264},
   MRCLASS = {90C25},
  MRNUMBER = {701288},
MRREVIEWER = {R.\ \c{S}erban},
}

@book {Ambrosio2008Gradient,
    AUTHOR = {Ambrosio, Luigi and Gigli, Nicola and Savar\'e, Giuseppe},
     TITLE = {Gradient flows in metric spaces and in the space of
              probability measures},
    SERIES = {Lectures in Mathematics ETH Z\"urich},
   EDITION = {Second},
 PUBLISHER = {Birkh\"auser Verlag, Basel},
      YEAR = {2008},
     PAGES = {x+334},
      ISBN = {978-3-7643-8721-1},
   MRCLASS = {49-02 (28A33 35K55 35K90 49Q20 60B05)},
  MRNUMBER = {2401600},
MRREVIEWER = {Pietro\ Celada},
}

@incollection{Jerrum1996Markov,
author = {Jerrum, Mark and Sinclair, Alistair},
title = {{The Markov chain Monte Carlo method: an approach to approximate counting and integration}},
year = {1996},
isbn = {0534949681},
publisher = {PWS Publishing Co.},
booktitle = {Approximation Algorithms for NP-Hard Problems},
pages = {482–520},
numpages = {39}
}

@incollection{Jordan1999Introduction,
  author = {Jordan, Michael I. and Ghahramani, Zoubin and Jaakkola, Tommi S. and Saul, Lawrence K.},
title = {An introduction to variational methods for graphical models},
year = {1999},
isbn = {0262600323},
publisher = {MIT Press},
address = {Cambridge, MA, USA},
booktitle = {Learning in Graphical Models},
pages = {105–161},
numpages = {57}
}

@incollection {Neal2011HMC,
    AUTHOR = {Neal, Radford M.},
     TITLE = {M{CMC} using {H}amiltonian dynamics},
 BOOKTITLE = {Handbook of {M}arkov chain {M}onte {C}arlo},
    SERIES = {Chapman \& Hall/CRC Handb. Mod. Stat. Methods},
     PAGES = {113--162},
 PUBLISHER = {CRC Press, Boca Raton, FL},
      YEAR = {2011},
      ISBN = {978-1-4200-7941-8},
   MRCLASS = {65C05 (37N20 60J22)},
  MRNUMBER = {2858447},
}

@incollection {Geyer2011Introduction,
    AUTHOR = {Geyer, Charles J.},
     TITLE = {Introduction to {M}arkov chain {M}onte {C}arlo},
 BOOKTITLE = {Handbook of {M}arkov chain {M}onte {C}arlo},
    SERIES = {Chapman \& Hall/CRC Handb. Mod. Stat. Methods},
     PAGES = {3--48},
 PUBLISHER = {CRC Press, Boca Raton, FL},
      YEAR = {2011},
      ISBN = {978-1-4200-7941-8},
   MRCLASS = {62-02 (60J22 65C05)},
  MRNUMBER = {2858443},
}

@book {Lee2018Introduction,
    AUTHOR = {Lee, John M.},
     TITLE = {Introduction to {R}iemannian manifolds},
    SERIES = {Graduate Texts in Mathematics},
    VOLUME = {176},
   EDITION = {Second},
 PUBLISHER = {Springer, Cham},
      YEAR = {2018},
     PAGES = {xiii+437},
      ISBN = {978-3-319-91754-2; 978-3-319-91755-9},
   MRCLASS = {53-01 (53B20 53B30 53C20 53C21)},
  MRNUMBER = {3887684},
MRREVIEWER = {Robert\ J.\ Low},
}

@article{teVrugt2020DDFT,
  author  = {te Vrugt, Michael and L{\"o}wen, Hartmut and Wittkowski, Raphael},
  title   = {Classical dynamical density functional theory: from fundamentals to applications},
  journal = {Advances in Physics},
  volume  = {69},
  number  = {2},
  pages   = {121--247},
  year    = {2020},
  doi     = {10.1080/00018732.2020.1854965}
}

@article{Shi2025Accelerating,
  author  = {Shi Chen and Qin Li and Oliver Tse and Stephen J. Wright},
  title   = {Accelerating optimization over the space of probability measures},
  journal = {Journal of Machine Learning Research},
  year    = {2025},
  volume  = {26},
  number  = {31},
  pages   = {1--40},
  url     = {http://jmlr.org/papers/v26/23-1288.html}
}

@inproceedings {Alimisis2020Continuous,
  title={{A continuous-time perspective for modeling acceleration in Riemannian optimization}},
  author={Alimisis, Foivos and Orvieto, Antonio and B{\'e}cigneul, Gary and Lucchi, Aurelien},
  booktitle={International Conference on Artificial Intelligence and Statistics},
  pages={1297--1307},
  year={2020},
  organization={PMLR}
}

@misc{maddison2018hamiltonian,
      title={{Hamiltonian Descent Methods}}, 
      author={Chris J. Maddison and Daniel Paulin and Yee Whye Teh and Brendan O'Donoghue and Arnaud Doucet},
      year={2018},
      eprint={1809.05042},
      archivePrefix={arXiv},
      primaryClass={math.OC},
      url={https://arxiv.org/abs/1809.05042}, 
}

@misc{zhang2018towards,
      title={{Towards Riemannian Accelerated Gradient Methods}}, 
      author={Hongyi Zhang and Suvrit Sra},
      year={2018},
      eprint={1806.02812},
      archivePrefix={arXiv},
      primaryClass={math.OC},
      url={https://arxiv.org/abs/1806.02812}, 
}

@misc{lou2024discrete,
      title={{Discrete Diffusion Modeling by Estimating the Ratios of the Data Distribution}}, 
      author={Aaron Lou and Chenlin Meng and Stefano Ermon},
      year={2024},
      eprint={2310.16834},
      archivePrefix={arXiv},
      primaryClass={stat.ML},
      url={https://arxiv.org/abs/2310.16834}, 
}

@article {CHOW20192440,
    AUTHOR = {Chow, Shui-Nee and Li, Wuchen and Zhou, Haomin},
     TITLE = {A discrete {S}chr\"odinger equation via optimal transport on
              graphs},
   JOURNAL = {J. Funct. Anal.},
  FJOURNAL = {Journal of Functional Analysis},
    VOLUME = {276},
      YEAR = {2019},
    NUMBER = {8},
     PAGES = {2440--2469},
      ISSN = {0022-1236,1096-0783},
   MRCLASS = {35R02 (05C90 35Q55 49Q20)},
  MRNUMBER = {3926122},
       DOI = {10.1016/j.jfa.2019.02.005},
}

@article {vonRenesse2012OTSchroedinger,
    AUTHOR = {von Renesse, Max-K.},
     TITLE = {An optimal transport view of {S}chr\"odinger's equation},
   JOURNAL = {Canad. Math. Bull.},
  FJOURNAL = {Canadian Mathematical Bulletin. Bulletin Canadien de
              Math\'ematiques},
    VOLUME = {55},
      YEAR = {2012},
    NUMBER = {4},
     PAGES = {858--869},
      ISSN = {0008-4395,1496-4287},
   MRCLASS = {81Q05 (37K05 82C70)},
  MRNUMBER = {2994690},
       DOI = {10.4153/CMB-2011-121-9},
}

@article{PhysRev.150.1079,
  title = {Derivation of the Schr\"odinger Equation from Newtonian Mechanics},
  author = {Nelson, Edward},
  journal = {Phys. Rev.},
  volume = {150},
  issue = {4},
  pages = {1079--1085},
  numpages = {0},
  year = {1966},
  month = {Oct},
  publisher = {American Physical Society},
  doi = {10.1103/PhysRev.150.1079},
  url = {https://link.aps.org/doi/10.1103/PhysRev.150.1079}
}

@article{EMCEE2013,
    author = "Foreman-Mackey, Daniel and Hogg, David W. and Lang, Dustin and Goodman, Jonathan",
    title = "{emcee: The MCMC Hammer}",
    eprint = "1202.3665",
    archivePrefix = "arXiv",
    primaryClass = "astro-ph.IM",
    doi = "10.1086/670067",
    journal = "Publ. Astron. Soc. Pac.",
    volume = "125",
    pages = "306--312",
    year = "2013"
}

@article{zuo2024gradient,
author = {Zuo, Xinzhe and Osher, Stanley and Li, Wuchen},
title = {Gradient-Adjusted Underdamped Langevin Dynamics for Sampling},
journal = {SIAM/ASA Journal on Uncertainty Quantification},
volume = {13},
number = {4},
pages = {1735-1765},
year = {2025},
doi = {10.1137/24M1702015},
}

@article{Majee2025MCMCnet,
doi = {10.1088/1361-6420/ae05c2},
year = {2025},
month = {sep},
publisher = {IOP Publishing},
volume = {41},
number = {9},
pages = {095013},
author = {Majee, Sudeb and Abhishek, Anuj and Strauss, Thilo and Khan, Taufiquar},
title = {MCMC-Net: accelerating Markov Chain Monte Carlo with neural networks for inverse problems},
journal = {Inverse Problems},
}

@misc{li2025geometriccalculationsprobabilitymanifolds,
      title={{Geometric calculations on probability manifolds from reciprocal relations in Master equations}}, 
      author={Wuchen Li},
      year={2025},
      eprint={2504.19368},
      archivePrefix={arXiv},
      primaryClass={math-ph},
      url={https://arxiv.org/abs/2504.19368}, 
}

@article{Kou2004Option,
 author = {Kou, S. G. and Wang, H.},
 ISSN = {00251909, 15265501},
 URL = {http://www.jstor.org/stable/30046226},
 journal = {Management Science},
 number = {9},
 pages = {1178--1192},
 publisher = {INFORMS},
 title = {{Option Pricing under a Double Exponential Jump Diffusion Model}},
 urldate = {2025-05-13},
 volume = {50},
 year = {2004}
}

@article{Su2016DE,
  author  = {Su, W. and Boyd, S.  and Cand{{\`e}}s, E. J. },
  title   = {{A Differential Equation for Modeling Nesterov's Accelerated Gradient Method: Theory and Insights}},
  journal = {Journal of Machine Learning Research},
  year    = {2016},
  volume  = {17},
  number  = {153},
  pages   = {1--43},
  url     = {http://jmlr.org/papers/v17/15-084.html}
}

@InProceedings{Salakhutdinov2009Deep,
  title = 	 {{Deep Boltzmann Machines}},
  author = 	 {Salakhutdinov, Ruslan and Hinton, Geoffrey},
  booktitle = 	 {Proceedings of the Twelfth International Conference on Artificial Intelligence and Statistics},
  pages = 	 {448--455},
  year = 	 {2009},
  volume = 	 {5},
  series = 	 {Proceedings of Machine Learning Research},
  url = 	 {https://proceedings.mlr.press/v5/salakhutdinov09a.html},
}

@InProceedings{Wibisono2018Sampling,
  title = 	 {{Sampling as optimization in the space of measures: The Langevin dynamics as a composite optimization problem}},
  author =       {Wibisono, Andre},
  booktitle = 	 {Proceedings of the 31st  Conference On Learning Theory},
  pages = 	 {2093--3027},
  year = 	 {2018},
  volume = 	 {75},
  series = 	 {Proceedings of Machine Learning Research},
  url = 	 {https://proceedings.mlr.press/v75/wibisono18a.html}
}

@InProceedings{Cheng2018Underdamped,
  title = 	 {{Underdamped Langevin MCMC: A non-asymptotic analysis}},
  author =       {Cheng, Xiang and Chatterji, Niladri S. and Bartlett, Peter L. and Jordan, Michael I.},
  booktitle = 	 {Proceedings of the 31st  Conference On Learning Theory},
  pages = 	 {300--323},
  year = 	 {2018},
  volume = 	 {75},
  series = 	 {Proceedings of Machine Learning Research},
  url = 	 {https://proceedings.mlr.press/v75/cheng18a.html}
}

@InProceedings{Taghvaei2019Accelerated,
  title = 	 {{Accelerated Flow for Probability Distributions}},
  author =       {Taghvaei, Amirhossein and Mehta, Prashant},
  booktitle = 	 {Proceedings of the 36th International Conference on Machine Learning},
  pages = 	 {6076--6085},
  year = 	 {2019},
  volume = 	 {97},
  series = 	 {Proceedings of Machine Learning Research},
  url = 	 {https://proceedings.mlr.press/v97/taghvaei19a.html}
}

@InProceedings{Li2022Hessian,
  title = 	 {{Hessian-Free High-Resolution {N}esterov Acceleration For Sampling}},
  author =       {Li, Ruilin and Zha, Hongyuan and Tao, Molei},
  booktitle = 	 {Proceedings of the 39th International Conference on Machine Learning},
  pages = 	 {13125--13162},
  year = 	 {2022},
  volume = 	 {162},
  series = 	 {Proceedings of Machine Learning Research},
  url = 	 {https://proceedings.mlr.press/v162/li22z.html},
}

@book{jackel2002monte,
  title={{Monte Carlo methods in finance}},
  author={J{\"a}ckel, P.},
  year={2002},
  publisher={John Wiley \& Sons}
}

@article{Le2026Swarm,
title = {Swarm dynamics for global optimization on finite sets},
journal = {Stochastic Processes and their Applications},
volume = {191},
pages = {104780},
year = {2026},
issn = {0304-4149},
doi = {https://doi.org/10.1016/j.spa.2025.104780},
author = {Nhat-Thang Le and Laurent Miclo}
}

\end{document}